\documentclass[11pt]{article}
\usepackage[margin=1in]{geometry}

\usepackage{graphicx} 
\usepackage{amsmath} 
\usepackage{amssymb}  
\usepackage{amsthm}

\usepackage{lipsum}
\usepackage{amsfonts}
\usepackage{graphicx}
\usepackage{epstopdf}
\usepackage{booktabs}

\usepackage{caption}
\usepackage{subcaption}

\usepackage{enumerate}

\usepackage{colortbl}

\usepackage{xcolor}
\usepackage{float}
\usepackage{mwe}
\usepackage{caption}
\usepackage{subcaption}
\usepackage[latin1]{inputenc}
\usepackage{booktabs}
\usepackage{multirow}
\usepackage{enumerate}
\usepackage[shortlabels]{enumitem}

\newcommand{\bg}{\bar{g}_{m,d} }
\newcommand{\bh}{\bar{h}_{m,d} }

\newcommand{\tg}{\tilde{g}_{m,d,r} } 
\newcommand{\thh}{\tilde{h}_{m,d,r} } 
\newcommand{\bigzero}{\mbox{\normalfont\Large 0}}

\usepackage{color}
\allowdisplaybreaks[1]

\usepackage{float}
\usepackage[latin1]{inputenc}
\usepackage{enumerate}
\usepackage[shortlabels]{enumitem}

\usepackage{algorithm}
\usepackage{algpseudocode}
\usepackage{hyperref} 

\theoremstyle{plain}  
\newtheorem{theorem}{Theorem}[section]

\newtheorem{lemma}[theorem]{Lemma}
\newtheorem{corollary}[theorem]{Corollary}
\newtheorem{proposition}[theorem]{Proposition}

\theoremstyle{definition}

\theoremstyle{remark} 

\newtheorem{remark}{Remark}[section]
\newtheorem{definition}{Definition}
\newtheorem{assumption}{Assumption}

\title{\LARGE \bf Shape-Constrained Regression\\
	using 	Sum of Squares Polynomials}
\author{Mihaela Curmei\thanks{Mihaela Curmei is with the department of Electrical Engineering and Computer Science at the University of California, Berkeley. Email: \texttt{mcurmei@eecs.berkeley.edu}} \and Georgina Hall\thanks{Georgina Hall is with the department of Decision Sciences at INSEAD. Email: \texttt{georgina.hall@insead.edu}}}

\begin{document}
\date{}
\maketitle


\begin{abstract}
We present a hierarchy of semidefinite programs (SDPs) for the problem of fitting a shape-constrained (multivariate) polynomial to noisy evaluations of an unknown shape-constrained function. These shape constraints include convexity or monotonicity over a box.  We show that polynomial functions that are optimal to any fixed level of our hierarchy form a consistent estimator of the underlying shape-constrained function. As a byproduct of the proof, we establish that sum-of-squares-convex polynomials are dense in the set of polynomials that are convex over an arbitrary box. A similar sum of squares type density result is established for monotone polynomials. In addition, we classify the complexity of convex and monotone polynomial regression as a function of the degree of the polynomial regressor. While our results show NP-hardness of these problems for degree three or larger, we can check numerically that our SDP-based regressors often achieve similar training error at low levels of the hierarchy. Finally,  on the computational side, we  present an empirical comparison of our SDP-based convex regressors with the convex least squares estimator introduced in \cite{hildreth1954point,holloway1979estimation} and show that our regressor is valuable in settings where the number of data points is large and the dimension is relatively small. We demonstrate the performance of our regressor for the problem of computing optimal transport maps in a color transfer task and that of estimating the optimal value function of a conic program. A real-time application of the latter problem to inventory management contract negotiation is presented.
\end{abstract}

\begin{small}



\end{small}

\section{Introduction} \label{sec:intro}

Shape-constrained regression is a fundamental problem in statistics and machine learning. It posits the existence of a shape-constrained function $f$ that maps \emph{feature vectors} to \emph{response variables}. Its goal is to obtain an \emph{estimator} (or \emph{regressor}) of this function, with the same shape constraints, from noisy feature vector-response variable pairings. The shape constraints we consider are of two types here: convexity constraints over a box and $K$-bounded-derivative constraints over a box, as defined in Section~\ref{sec:sose}. Bounded-derivative constraints include as subcases both the case where the regressor is constrained to be monotone and the case where it is constrained to be Lipschitz-continuous with a fixed Lipschitz constant. Combined, these shape constraints cover the wide majority of shape constraints arising in applications. A short and non-exhaustive list of areas where regression with shape constraints such as these appear include economics~\cite{meyer1968consistent,pricing}, psychology~\cite{gallistel2004learning}, engineering~\cite{leukemia,software_failure,shade}, and medicine~\cite{prince1990reconstructing}. 

In this paper, we study a set of shape-constrained (multivariate) polynomial regressors, the Sum of Squares Estimators (SOSEs), which are obtained via a semidefinite programming hierarchy. They are parametric in $d$, their degree, and $r$, the level of the hierarchy (see Section~\ref{subsec:sose.def}). While we are not the first paper to consider shape-constrained polynomials of this type \cite{magnani2009convex,wang2019calibrating}, we are the first to propose a systematic analysis of estimators defined in this way, from a variety of angles. More specifically, our contributions are the following:
\begin{enumerate}[(1)]
	\item We showcase a regime in which the SOSEs are competitive in terms of computation time. This corresponds to the setting where the number of data points is large, the dimension is relatively small, and predictions need to be made often and quickly. Within this regime, we provide experimental evidence that the SOSEs outperform, in terms of generalization error, two alternative shape-constrained regressors, the Convex Least-Squares Estimator or CLSE \cite{hildreth1954point,holloway1979estimation} and the Maximum-Affine Estimator or MAE \cite{ghosh2019max}, which are among the most prevalent convex regressors; see Section \ref{sec:sose}.
	\item We show that, for fixed $r$, the SOSEs are consistent estimators of the underlying shape-constrained function. In doing so, we prove that sos-convex polynomials (see Section \ref{subsec:approx.results} for a definition) are dense in the set of polynomials convex over a box. We also show that a similar result holds for monotonicity. These results can be viewed as analogs of sum of squares density results for nonnegative polynomials in \cite{lasserre2007sum,lasserre2007sos} for convex and monotone polynomials, and may be of independent interest; see Section~\ref{sec:consistency}.
	\item We compare the SOSEs against their limits as $r \rightarrow \infty$: these limits correspond to the solutions to convex and $K$-bounded-derivative polynomial regression. We provide a complete characterization of the complexity of solving these problems as a function of the degree $d$ of the regressor. Our results show NP-hardness for degree three or larger. We also propose a convex optimization-based method for bounding the gap between the optimal value of shape-constrained polynomial regression problems and that of the SOSE training problem (for fixed~$r$). Our numerical experiments show that the gap is very small already for small values of $r$; see Section~\ref{sec:sc.pr}.
	\item We propose three applications which correspond to settings where the SOSEs perform particularly well. The first application is well known and involves fitting a production function to data in economics. Our main contribution here is to show that we outperform the prevalent approach in economics, which uses Cobb-Douglas functions. The second application is to the problem of computing optimal transport maps, and more specifically the problem of color transfer. While using shape-constrained regression to compute optimal maps is not new \cite{paty2020regularity}, we are the first to tackle it using sum of squares-based methods. This is intriguing as the approach adopted by \cite{paty2020regularity}, which relies on a variant of the CLSE, is not as well-suited to color transfer as the SOSEs are. Our third application is, to the best of our knowledge, an entirely novel use of shape-constrained regression. It involves estimating the optimal value function of a conic program. We present a real-time application of this problem to inventory management contract negotiation; see Section \ref{sec:apps}.
\end{enumerate}
The sum of squares techniques we present in this paper have the advantage, unlike other techniques, of being very modular: they can handle with ease a wide variety of shape constraints, including convexity, monotonicity, Lipschitz continuity, or any combinations of these, globally or over regions. They also produce explicit algebraic certificates enabling users to independently verify that the regressors obtained possess the shape constraints of interest. This can help with interpretability and enhance user trust. Through this paper, we hope to encourage a broader use of these techniques for shape-constrained regression, particularly in settings, such as color transfer, for which they are well-suited.

\section{The Sum of Squares Estimators (SOSEs)} \label{sec:sose}

We define the SOSEs in Section \ref{subsec:sose.def}, investigate their computation time and performance in Section \ref{subsec:sose.computation}, and compare them against two other estimators in Section \ref{subsec:sose.comparison}.

\subsection{Definition of the SOSEs} \label{subsec:sose.def}

Given a feature vector $X \in \mathbb{R}^n$ and a response variable $Y \in \mathbb{R}$, we model by a function $f: \mathbb{R}^n \rightarrow \mathbb{R}$ the relationship between $X$ and $Y$. We assume that $X$ belongs to a full-dimensional box $B \subset \mathbb{R}^n$:
\begin{align} \label{eq:box}
B=\{(x_1,\ldots,x_n) \in \mathbb{R}^n~|~ l_i \leq x_i \leq u_i,~\forall i=1,\ldots,n\} \subset \mathbb{R}^n,
\end{align}
where $l_i<u_i$ for all $i=1,\ldots,n$. This choice is made as, in applications, each feature typically lies within a range. Our approach also has natural extensions to other semialgebraic sets. As mentioned in Section \ref{sec:intro}, we assume that some ``shape information'' on $f$ over $B$ is known (which is less restrictive than assuming shape information knowledge over $\mathbb{R}^n$ in full). In particular, $f$ is assumed to have at least one of the two following shape constraints.
\begin{definition}[Convex over a box]\label{def:cvx}
	A function $f: \mathbb{R}^n \rightarrow \mathbb{R}$ is \emph{convex over a box $B$} if for any $x,y \in B$ and for any $\lambda \in [0,1]$, we have
	$f(\lambda x+(1-\lambda)y) \leq \lambda f(x)+(1-\lambda)f(y).$
	For twice differentiable functions $f$ and full-dimensional boxes $B$, this is equivalent\footnote{We use here the standard notation $A \succeq 0$ to denote that the symmetric matrix $A$ is positive semidefinite.} to
	$H_f(x) \succeq 0, \forall x \in B,$
	where $H_f(x)$ is the Hessian of $f$ at $x$.
\end{definition}
A proof of the equivalence can be found, e.g., in \cite[Section 1.1.4]{bertsekas2009convex}.

\begin{definition}[K-bounded derivatives over a box]\label{def:bdr}
	Given $K_1^-$,$K_1^+$,$\ldots$,$K_n^-$,$K_n^+ \in \mathbb{R} \cup\{\pm \infty\}$ with $K_i^- \leq K_i^+, i=1,\ldots,n$, let $K=(K_1^-,K_1^+,\ldots,K_n^-,K_n^+) \text{ and } I^{\pm}=\{i \in \{1,\ldots,n\}~|~ K_i^{\pm} \text{ is finite}\}.$ A continuously-differentiable function $f:\mathbb{R}^n \rightarrow \mathbb{R}$ is said to have \emph{$K$-bounded derivatives} over a box $B$ if
	\begin{align} \label{eq:kbdder}
	\frac{\partial f(x)}{\partial x_i} \geq K_i^-, ~\forall x\in B,~ \forall i \in I^-, \quad \frac{\partial f(x)}{\partial x_i} \leq K_i^+, ~\forall x \in B,~\forall i \in I^+.
	\end{align}
\end{definition}
Assuming that $f$ has these shape constraints, the goal is then to recover an estimator of $f$ with the same shape constraints from $m$ observed feature vector-response variable pairings $(X_i,Y_i)_{i=1,\ldots,m}$. We propose such an estimator in this paper: these are the Sum of Squares Estimators (SOSEs). 

Before we formally define them, we recall that an $n$-variate polynomial $p$ is a \emph{sum of squares} (sos) polynomial if $p(x_1,\ldots,x_n)=\sum_{i=1}^r q_i^2(x_1,\ldots,x_n)$ for some $n$-variate polynomials $q_i$. Furthermore, for a $t \times t$ polynomial matrix $M(x)$ (that is, a matrix with entries that are polynomials), we say that $M(x)$ is an \emph{sos matrix} if there exists a $t' \times t$ polynomial matrix $V(x)$ such that $M(x)=V(x)^TV(x)$, or equivalently, if the (scalar-valued) polynomial $y^TM(x)y$ in $x$ and $y$ is a sum of squares polynomial. We denote by $P_{n,d}$, the set of polynomials in $n$ variables and of degree at most $d$, by $\Sigma_{n,2d}$ the set of polynomials in $P_{n,2d}$ that are sums of squares, and by $\Sigma_{n,2d,t}^M$ the set of sos matrices of size $t \times t$ and with entries that are polynomials in $P_{n,2d}$.

We are now ready to formally define the SOSEs, $\tilde{g}_{m,d,r}$ and $\tilde{h}_{m,d,r}$.  Let $b_i(x_i)=(u_i-x_i)(x_i-l_i)$ for $i=1,\ldots,n$ so that $B=\{(x_1,\ldots,x_n)~|~b_1(x_1) \geq 0, \ldots, b_n(x_n) \geq 0\}.$
\begin{definition} \label{def:tg}
	Given $m$ feature vector-response variables $(X_i,Y_i)_{i=1,\ldots,m}$ and two nonnegative integers $d,r$, we define $\tg$ as:
	\begin{equation}\label{eq:opt.tg}
	\begin{aligned}
	\tg \mathop{\mathrel{:}}=&\arg &&\min_{g \in P_{n,d}, ~S_0 \in \Sigma_{n,2r,n}^M, S_1\ldots,S_n \in \Sigma_{n,2r-2,n}^M} \sum_{i=1}^m (Y_i-g(X_i))^2\\
	&\text{s.t. } &&H_g(x)= S_0(x)+b_1(x)S_1(x)+\ldots+b_n(x)S_n(x).
	\end{aligned}
	\end{equation}
\end{definition}

\begin{definition}\label{def:th} Let $I^{\pm}$ be as in Definition \ref{def:bdr}. Given $m$ feature vector-response variables $(X_i,Y_i)_{i=1,\ldots,m}$ and two nonnegative integers $d,r$, we define $\thh$ as:
	\begin{equation}\label{eq:opt.th}
	\begin{aligned}
	\thh \mathop{\mathrel{:}}=&\arg &&\min_{h \in P_{n,d},~s_{i0}^+, s_{i0}^- \in \Sigma_{n,2r}, s_{ij}^+,s_{ij}^-\in \Sigma_{n,2r-2}} \sum_{i=1}^m (Y_i-h(X_i))^2\\
	&\text{s.t. } &&K_i^+-\frac{\partial h(x)}{\partial x_i}= s_{i0}^+(x)+ s_{i1}^{+}(x)b_1(x)+\ldots+s_{in}^+(x)b_n(x), \text{ for }i \in I^+,\\
	& &&\frac{\partial h(x)}{\partial x_i}-K_i^-= s_{i0}^-(x)+ s_{i1}^{-}(x)b_1(x)+\ldots+s_{in}^-(x)b_n(x),\text{ for }i \in I^-.\\
	\end{aligned}
	\end{equation}
\end{definition}

Note that the SOSEs are polynomials in $n$ variables and of degree $d$. They always exist and are unique, provided that we take $m$ large enough and that the data points $\{X_i\}_{i=1,\ldots,m}$ are linearly independent. They are parametric estimators, depending on two parameters, $d$ and $r$, as well as on the data points $(X_i,Y_i)_{i=1,\ldots,m}$, as reflected in the notation. They also possess the appropriate shape constraints. Indeed, $\tilde{g}_{m,d,r}$ is convex over $B$ and $\tilde{h}_{m,d,r}$ has $K$-bounded derivatives over $B$. This can be seen by noting that when $x \in B$, $b_i(x) \geq 0$, that sos polynomials are nonnegative, and that sos matrices are positive semidefinite.

	\subsection{Computing the SOSEs and Dependence on Input Parameters} \label{subsec:sose.computation}
	
	It is well known that testing membership to $\Sigma_{n,2d}$ can be reduced to a semidefinite program (SDP). Indeed, a polynomial $p(x_1,\ldots,x_n)$ of degree $2d$ is sos if and only if there exists a positive semidefinite matrix $Q$ (we write $Q \succeq 0$) of size $\mathbb{R}^{\binom{n+d}{d} \times \binom{n+d}{d}}$ such that $p(x)~=z(x)^TQz(x)$, where $z(x)=(1,x_1,\ldots,x_n,\ldots,x_n^d)^T$ is the vector of all monomials of degree at most $d$. Thus, computing the SOSEs amounts to solving SDPs. Of interest to us is how the size of the SDPs scales with $m,n,d,$ and $r$. For both (\ref{def:tg}) and (\ref{def:th}), the data points only appear in the objective: thus, the size of the SDPs is independent of the number $m$ of data points. Their size does depend however on $n,d,$ and $r$. For (\ref{def:tg}), the number of equality constraints is equal to $\binom{n+2}{2} \cdot \binom{n+\max \{2r,d-2\}}{\max \{2r,d-2\}}$ and the size of the $n+1$ semidefinite constraints is $n \cdot \binom{n+r}{r} \times n \cdot \binom{n+r}{r}$. For (\ref{def:th}), the number of equality constraints is equal to $2n \cdot \binom{n+ \max\{d-1,2r\}}{\max \{ d-1,2r\}}$ and the size of the $2n+2$ semidefinite constraints is $\binom{n+r}{r} \times \binom{n+r}{r}$. Bearing in mind that $\binom{n+k}{k}=\binom{n+k}{n}$ and that $\binom{n+k}{k}=O((n+k)^k)$, we get that, for fixed $n$, the sizes of (\ref{def:tg}) and (\ref{def:th}) grow polynomially in $d$ and $r$, and that for fixed $d$ and $r$, the sizes of (\ref{def:tg}) and (\ref{def:th}) grow polynomially in $n$. 
	
	Even though the size of the SDPs scales polynomially with the parameters of interest, in practice, SDPs can suffer from scalability issues~\cite{majumdar2020recent}. Thus, computing the SOSEs is generally faster when $n,d,r$ are small, though $m$ can be taken as large as needed, as the size of the SDPs is independent of $m$. In practice, $m$ and $n$ are fixed as artifacts of the application under consideration. The parameters $r$ and $d$ however are fixed by the user. They should be chosen using a statistical model validation technique, such as cross validation, to ensure that generalization error is low on the test data. As an example, we experimentally investigate the impact of the choice of $r$ and $d$ on the generalization error for estimating the function $$f_1(w_1,\ldots,w_n)=(w_1+\ldots+w_n) \log(w_1+\ldots+w_n).$$ We plot the results in Figure \ref{fig:train_test_rmse}. To obtain these plots, we use datasets generated as explained in Appendix \ref{appendix:specs} with $m=10,000$. An analogous plot with a different function $f_2$ is given in Appendix \ref{subsec:add.exp}. We vary the values of $n,d,r$ as indicated in Figure \ref{fig:train_test_rmse}. The train RMSE (resp. test RMSE) are concepts formally defined in Appendix \ref{appendix:specs}. Roughly speaking, the lower they are, the closer the values obtained by evaluating the SOSEs on the training (resp. testing) feature vectors are to the training (resp. testing) response variables. Figure \ref{fig:train_test_rmse} indicates that low values of $d$ and $r$ often lead to better test RMSE (i.e., generalization error) than larger $d,r$. In fact, low $d$ and $r$ seem to have a regularization effect.
	
	\begin{figure}[]
		\begin{subfigure}{\textwidth}
			\centering
			\includegraphics[width=0.8\linewidth]{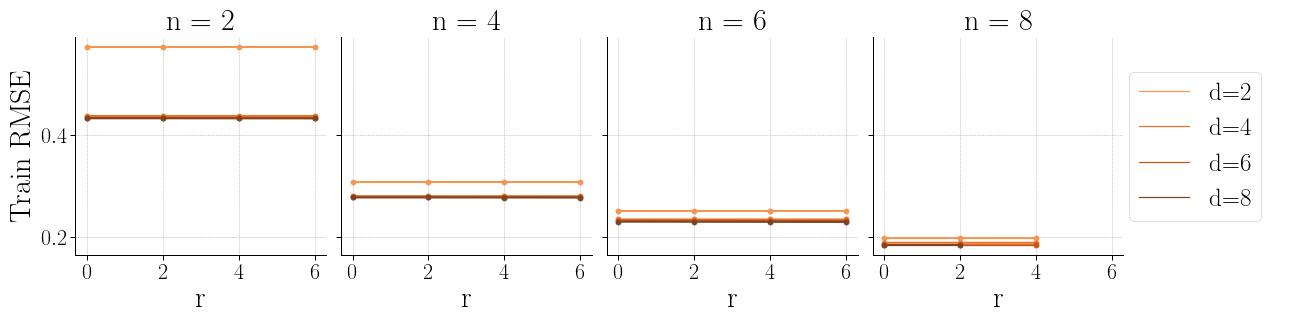}  
			\caption{\footnotesize Train RMSE.}
			\label{fig:train_rmse}
		\end{subfigure}
		\begin{subfigure}{\textwidth}
			\centering
			\includegraphics[width=0.8\linewidth]{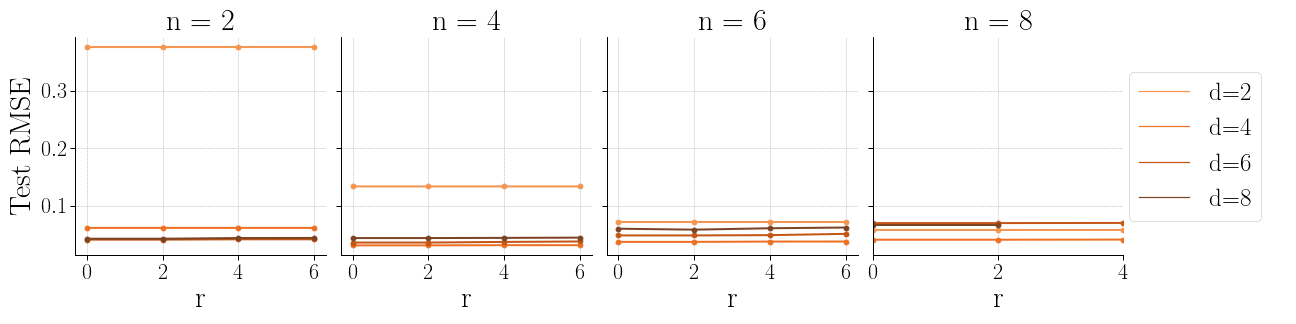}  
			\caption{\footnotesize Test RMSE.}
			\label{fig:test_rmse}
		\end{subfigure}
		\caption{Train and Test RMSE for (\ref{eq:opt.tg}) with $m=10,000$ data points generated as described in Appendix \ref{appendix:specs}, as the number $n$ of features, the degree $d$ of the polynomials, and the degree $r$ of the semidefinite programming hierarchy vary. Lighter color corresponds to lower $d$.}
		\label{fig:train_test_rmse}
	\end{figure}
	
	As well as being fast to compute when $m$ is large and $n$ is small, evaluating the SOSEs on a new feature vector is also fast to do as it simply amounts to evaluating a polynomial at a point. Thus, the SOSEs should be favored in applications where $m$ is large, $n$ is small, and predictions need to be made often and quickly (see Section \ref{sec:apps} for examples). In the next subsection, we give experimental evidence that the SOSEs tend to have low generalization error when compared to other prominent methods for convex regression. In the section that follows, we show favorable theoretical properties of the SOSEs, namely that they are consistent statistical estimators even when $r$ is fixed.
	
	\subsection{Comparison of the SOSEs Against Other Estimators} \label{subsec:sose.comparison}
	
	Many methods exist for shape-constrained regression, in particular for convex regression (see, e.g., \cite{papp2014shape} and the references within for a comprehensive literature review). We restrict ourselves to two here, the Convex Least Squares Estimator (CLSE) as in \cite{hildreth1954point, holloway1979estimation}, and the Max-Affine Estimator (MAE), as in, e.g., \cite{magnani2009convex}. The CLSE, $\hat{g}_m$, is a non-parametric convex piecewise-affine estimator, obtained by solving the quadratic program (QP): 
	\begin{equation} \label{eq:QP}
	\begin{aligned}
	&\min_{g_i \in \mathbb{R}, \xi_i \in \mathbb{R}^n} \sum_{i=1}^m (Y_i-g_i)^2\\
	&\text{s.t. } g_j \geq g_i+\xi_i^T(X_j-X_i), \text{ for }1 \leq i,j \leq m,
	\end{aligned}
	\end{equation}
	with optimal solution $\{g_i^*,\xi_i^*\}$, and taking 
	\begin{align} \label{eq:def.gm}
	\hat{g}_m(x)=\sup \{ g(x)~|~g \text{ convex }, g(X_i)=g_i^*, 1 \leq i, \leq~m\}, \text{ for all } x \in \mathbb{R}^n.
	\end{align}
	(Note that while the $\{g_i^*\}_{i=1,\ldots,m}$ are unique, the $\{\xi_i^*\}_{i=1,\ldots,m}$ are not, and consequently, defining $\hat{g}_m$ via (\ref{eq:def.gm}) is needed.) The MAE is a convex piecewise-affine estimator as well, but parametric with parameter the number of pieces, $k$. It is obtained as a solution of $\min_{\Theta_j \in \mathbb{R}^n, \beta_j \in \mathbb{R}} \sum_{i=1}^m (Y_i-\max_{1 \leq j \leq k} (X_i^T\Theta_j+\beta_j))^2$. Unlike the CLSE, the MAE is NP-hard to compute, but many heuristics have been developed (some with statistical guarantees) to obtain it. Here, we use the ones proposed by \cite{ghosh2019max}. 
	
	We restrict ourselves to the CLSE for three reasons. First, it is arguably the most prevalent shape-constrained regressor in the literature. Second, the computation time for many popular alternative regressors (e.g., based on isotonic regression, lattice methods, etc.) grows exponentially in the number $n$ of features and we wish to compare the SOSE against methods which, just like the SOSE with any fixed $d,r$, are polynomial in $n$. Finally, like the SOSE, the CLSE can be obtained by solving a convex program. We also consider the MAE, though it does not fit these criteria, as it can be viewed as a parametric version of the CLSE and the SOSE is also parametric. A downside in our opinion for both the CLSE and the MAE relates to the limited type of constraints that can be required. For example, as the CLSE and MAE are globally convex by definition, we cannot require that they be monotonous only, or convex only over a region, which some applications may call for. As they are piecewise-affine, we cannot require either that they be strongly convex, as needed, e.g., in the optimal transport application of Section \ref{subsec:opt.transport}. Extensions to incorporate $\ell$-strong convexity do exist \cite[Theorems 3.8, 3.14]{taylor}, but we would argue that they are not very straightforward to derive, contrarily to the SOSE where we would simply replace $H_g(x)$ in (\ref{def:tg}) by $H_g(x)-\ell I$, where $I$ is the $n \times n$ identity matrix.
	
	In terms of computation, the SOSE and the CLSE can be viewed as methods best-suited to complementary settings. While the SOSE is quicker to compute when $n$ is small and $m$ is large, the CLSE is quicker to compute when $n$ is large and $m$ is small. Indeed, the QP in (\ref{eq:QP}) has a number of variables that scales linearly and a number of constraints that scales quadratically with the number $m$ of data points (unlike the SDP whose size does not scale with $m$). This can be expected as the CLSE is non-parametric, while the SOSEs are not. We illustrate the differences in solving time in Figure \ref{fig:runtime_features_degree} (implementation details can be found in Appendix \ref{appendix:specs}). The SOSE is much faster to compute in settings where the number of data points is moderate to large ($m=2000$ to $m=10000$); see Figure~\ref{fig:runtime_features_degree}. In fact, for most of the QPs solved for these values of $m$, the 4-hour time-out that we put into place was reached. Evaluation of the CLSE, when it is defined as above, can also be quite slow, as it requires solving a linear program whose size scales with $m$ (see (\ref{eq:def.gm}) and, e.g., \cite{lim2012consistency}). When solving this linear program, one can encounter issues of infeasibility if the quadratic program is not solved to high accuracy. One can also encounter unboundedness issues if asked for a prediction on a point which does not belong to the convex hull of the training points. This is not the case for the SOSEs, though we point out that our statistical guarantees from Section \ref{sec:consistency} only apply to points in the convex hull of the training points. Typical heuristics for computing the MAE run much faster than the CLSE, as their solving time does not scale with $m$. They also are faster to compute than the SOSEs as they involve solving cheaper convex optimization problems. Similarly, evaluation of the MAE is very fast, as it only requires evaluating $k$ affine functions and taking the maximum of the $k$ evaluations, which is comparable to the SOSEs. However, as the MAE is a solution to a non-convex optimization problem, this opens the door to a series of complications which are not encountered with the SOSEs such as choosing an appropriate heuristic or initializing in an appropriate manner. 
	
	In terms of quality of prediction, we compare the SOSEs against the MAE and the CLSE in the regime where we advocate for the SOSE, namely $m$ large and $n$ smaller. The train and test RMSEs are given in Table \ref{tab:gen.error}. They are obtained using a dataset generated as explained in Appendix \ref{appendix:specs} with $f=f_1$, $r=1$ for the SOSEs, and varying $m,n,d,k$. The differences between Spect Opt, Rand Opt, and Rand are also explained in Appendix \ref{appendix:specs}; at a high level, they correspond to different initialization techniques for the MAE. The best results in terms of test RMSE are indicated in bold: these are always obtained by the SOSE. This happens even when we select the best performing hyperparameters for the MAE. We provide a similar table with $f=f_2$ in Appendix \ref{subsec:add.exp}.

	\begin{figure}[]
		\begin{subfigure}{\textwidth}
			\footnotesize
			\centering
			\includegraphics[width=0.8\linewidth]{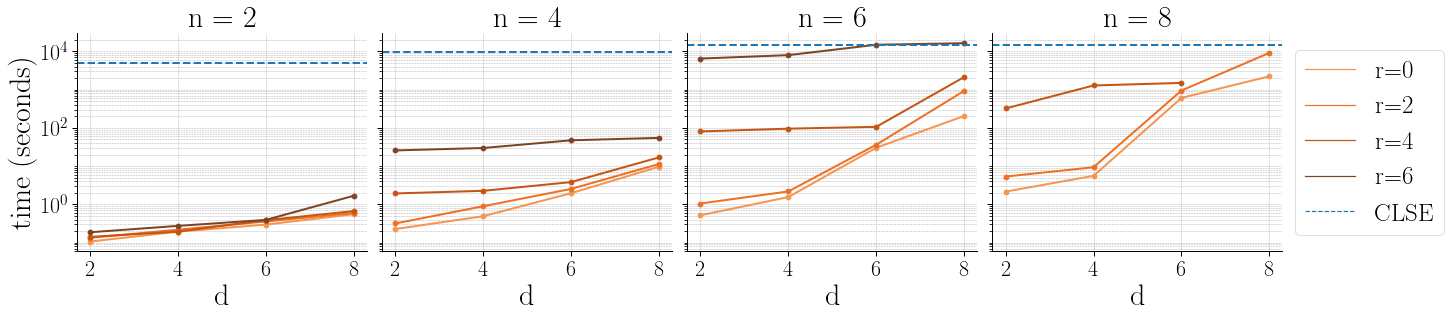}  
			\caption{\footnotesize Solver time of (\ref{eq:opt.tg}) for different $n,d,r$, using $m=10,000$ training samples. The dashed blue line corresponds to the solver time for the CLSE QP (\ref{eq:QP}).}
			\label{fig:runtime_features}
		\end{subfigure}
		\begin{subfigure}{\textwidth}
			\centering
			\includegraphics[width=0.8\linewidth]{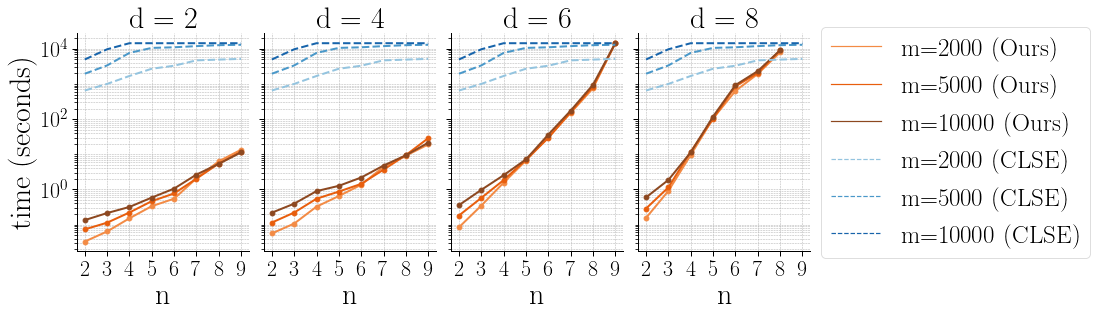}  
			\caption{\footnotesize Solver time needed to compute the CLSE and the SOSE in (\ref{eq:opt.tg}) with respect to $n$. For the SOSE, we take $r=2$ and a range of degrees $d$. Lighter colors correspond to fewer training points. All run-times capped at 4 hours per setup.}
			\label{fig:runtime_degrees}
		\end{subfigure}
		\caption{Runtime comparison of our method and the CLSE.}
		\label{fig:runtime_features_degree}
	\end{figure}
	
	\begin{table}[H]
		\footnotesize
		\centering
		\begin{tabular}{|c|c|c|c|c|>{\columncolor[gray]{0.9}}c|>{\columncolor[gray]{0.9}}c|>{\columncolor[gray]{0.9}}c|c|>{\columncolor[gray]{0.9}}c|c|c|c|>{\columncolor[gray]{0.9}}c|>{\columncolor[gray]{0.9}}c|>{\columncolor[gray]{0.9}}c|}
			\cline{3-16}
			\multicolumn{2}{c|}{~} & \multicolumn{6}{c|}{\textbf{MAE}} & \multicolumn{2}{c|}{\textbf{CLSE}} & \multicolumn{6}{c|}{\textbf{SOSE}}\\ \hline
			\multicolumn{2}{|c|}{RMSE} & \multicolumn{3}{c|}{Train} &   \multicolumn{3}{c|}{\cellcolor[gray]{0.9} Test} &  Train & Test  & \multicolumn{3}{c|}{Train} & \multicolumn{3}{c|}{\cellcolor[gray]{0.9} Test} \\ \hline
			\multirow{2}{*}{$m$} & \multirow{2}{*}{$n$}  & Spect & Rand & ~ & Spect & Rand & ~ & ~ & ~ & \multirow{2}{*}{$d=2$}& \multirow{2}{*}{$d=4$}& \multirow{2}{*}{$d=6$} & ~ & ~ & ~ \\
			~ & ~ & Opt & Opt & \multirow{-2}{*}{Rand}  & Opt & Opt & \multirow{-2}{*}{Rand} & ~ & ~ & ~ & ~ & ~ &\multirow{-2}{*}{$d=2$} & \multirow{-2}{*}{$d=4$} &\multirow{-2}{*}{$d=6$} \\ \hline
			\multirow{5}{*}{2000} & 2 & 0.715 & 0.433 & 0.438 & 0.559 & 0.072 & 0.088 & 0.438 & 0.087 & 0.574 & 0.438 & 0.434 & 0.376 & 0.061 & \textbf{0.041} \\
			& 3 & 0.400 & 0.347 & 0.345 & 0.191 & 0.073 & 0.059 & 0.355 & 0.106 & 0.401 & 0.348 & 0.346 & 0.211 & 0.042 & \textbf{0.050} \\ 
			& 4 & 0.295 & 0.278 & 0.416 & 0.087 & 0.050 & 0.328 & 0.311 & 0.160 & 0.307 & 0.280 & 0.279 & 0.133 & \textbf{0.031} & 0.036 \\
			& 5 & 0.264 & 0.259 & 0.359 & 0.068 & 0.062 & 0.276 & 0.317 & 0.231 & 0.275 & 0.260 & 0.258 & 0.110 & \textbf{0.040} & 0.055 \\
			& 6 & 0.237 & 0.235 & 0.599 & 0.040 & 0.045 & 0.580 & 0.305 & 0.228 & 0.251 & 0.235 & 0.232 & 0.071 & \textbf{0.037} & 0.048 \\ \hline
			\multirow{5}{*}{5000} & 2 & 0.712 & 0.435 & 0.439 & 0.557 & 0.049 & 0.077 & 0.439 & 0.073 & 0.578 & 0.437 & 0.434 & 0.374 & 0.052 & \textbf{0.024} \\ 
			& 3 & 0.405 & 0.352 & 0.350 & 0.190 & 0.043 & 0.037 & 0.373 & 0.128 & 0.408 & 0.352 & 0.352 & 0.167 & 0.021 & \textbf{0.016} \\
			& 4 & 0.303 & 0.290 & 0.392 & 0.084 & 0.039 & 0.286 & 0.336 & 0.184 & 0.321 & 0.291 & 0.290 & 0.132 & 0.098 & \textbf{0.011} \\ 
			& 5 & 0.257 & 0.253 & 0.446 & 0.053 & 0.035 & 0.365 & 0.332 & 0.222 & 0.271 & 0.253 & 0.252 & 0.095 & \textbf{0.013} & 0.017 \\ 
			& 6 & 0.235 & 0.232 & 0.299 & 0.032 & 0.032 & 0.205 & 0.338 & 0.265 & 0.247 & 0.232 & 0.231 & 0.070 & \textbf{0.026} & 0.034 \\ \hline
			\multirow{5}{*}{10000} & 2 & 0.721 & 0.438 & 0.445 & 0.557 & 0.035 & 0.075 & 0.448 & 0.083 & 0.586 & 0.443 & 0.439 & 0.373 & 0.051 & \textbf{0.019} \\ 
			& 3 & 0.405 & 0.349 & 0.349 & 0.191 & 0.034 & 0.032 & 0.376 & 0.134 & 0.411 & 0.351 & 0.350 & 0.211 & 0.075 & \textbf{0.050} \\
			& 4 & 0.300 & 0.289 & 0.332 & 0.085 & 0.029 & 0.183 & 0.347 & 0.197 & 0.320 & 0.289 & 0.288 & 0.131 & 0.011 & \textbf{0.004} \\
			& 5 & 0.260 & 0.255 & 0.489 & 0.054 & 0.030 & 0.445 & 0.358 & 0.252 & 0.274 & 0.254 & 0.253 & 0.110 & 0.032 & \textbf{0.006} \\ 
			& 6 & 0.233 & 0.231 & 0.316 & 0.031 & 0.026 & 0.222 & 0.365 & 0.295 & 0.244 & 0.231 & 0.230 & 0.070 & 0.018 & \textbf{0.002} \\ \hline
		\end{tabular}
		\caption{Comparison of the train and test RMSEs, for different values of $m$ and $n$, of the SOSE $\tg$ computed in (\ref{eq:opt.tg}) (with $r=1$), for the CLSE, and for the MAE under different initializations and choice of hyperparameters; see Appendix \ref{appendix:specs}. Best test RMSE marked in bold font.}
		\label{tab:gen.error}
	\end{table}

	\section{Sum of Squares Approximations and Consistency of the SOSEs} \label{sec:consistency}
	
	In Section \ref{subsec:approx.results}, we present two key results of the paper. While they are used as stepping stones towards our consistency results (Section \ref{subsec:consistency}), they may be of independent interest to the polynomial or algebraic optimization community as analogs of results in \cite{lasserre2007sos} for convex/monotone polynomials~(Section~\ref{subsec:implications}).

	\subsection{Two Algebraic Approximation Results} \label{subsec:approx.results}
	
	We refer to a polynomial $p$ whose Hessian $H_p$ is an sos matrix as \emph{sos-convex} (see, e.g., \cite{helton2010semidefinite}) to a polynomial $p$ whose partial derivatives satisfy $K_i^+-\frac{\partial p}{\partial x_i}$ is sos for $i \in I^+$ and $\frac{\partial p}{\partial x_i}-K_i^-$ is sos for $i \in I^-$ as having \emph{sos-$K$-bounded derivatives.} In the remainder of the paper, we assume that $I^+ \cap I^-=\emptyset$. We show that any polynomial that is convex (resp. that has $K$-bounded derivatives) over $[-1,1]^n$ can be closely approximated by an sos-convex polynomial (resp. sos-$K$-bounded derivative polynomial), obtained by slightly perturbing its higher order terms.

	\begin{theorem} \label{thm:conv}
		Let $g:\mathbb{R}^n \rightarrow \mathbb{R}$ be a polynomial that is convex over $[-1,1]^n$. For $d' \in \mathbb{N}$, define
		$$\Theta_{d'}(x_1,\ldots,x_n)=\sum_{1 \leq i \leq n} x_i^2 + \sum_{1 \leq i < j \leq n} \left(8x_i^{2d'+2} + x_i^{2d'}x_j^2 + x_i^2 x_j^{2d'}+8x_j^{2d'+2}\right).$$
		For any $\epsilon>0$, there exists $d_0'$ such that $g_{\epsilon}\mathrel{\mathop{:}}=g+\epsilon \Theta_{d'}$ is sos-convex for all $d' \geq d_0'$.
	\end{theorem}
	
	\begin{theorem} \label{thm:mon}
		Let $h: \mathbb{R}^n \rightarrow \mathbb{R}$ be a polynomial with $K$-bounded derivatives over $[-1,1]^n$. Let $\rho \in \{-1,0,1\}^n$ with $\rho_i=1$ (resp. $-1$, $0$) if $K_i^-$ (resp. $K_i^+$, neither) is finite. For $d' \in \mathbb{N}$, define
		$$\Xi_{d'}(x_1,\ldots,x_n)=\sum_{1 \leq i \leq n} \rho_ix_i+\sum_{1 \leq i <j \leq n} \left(4\rho_i x_i^{2d'+1}+2 \rho_j x_i^{2d'} x_j +2 \rho_i x_i x_j^{2d'}+4\rho_j x_j^{2d'+1}\right).$$ 
		For any $\epsilon>0$, there exists $d'_0$ such that $h_{\epsilon} \mathrel{\mathop{:}}=h+\epsilon \Xi_{d'}$ has sos-$K$-bounded derivatives for all $d' \geq d_0'$.
	\end{theorem}
	These theorems easily extend to \emph{any} full-dimensional box $B \subset \mathbb{R}^n$. They should be contrasted to the following result, which holds for a polynomial that is nonnegative over $[-1,1]^n$.
	
	\begin{theorem}[Corollary 3.3, \cite{lasserre2007sos}] \label{thm:lasserre}
		Let $f$ be a nonnegative polynomial over $[-1,1]^n$. For $d' \in \mathbb{N}$, define
		$$\Psi_{d'}(x_1,\ldots,x_n)=1+\sum_{i=1}^n x_i^{2d'}.$$
		For any $\epsilon>0$, there exists $d_0'$ such that $f+\epsilon \Psi_{d'}$ is sos for all $d' \geq d_0'$.
	\end{theorem} 
	
	Despite the similarities between Theorem \ref{thm:lasserre} and Theorems \ref{thm:conv} and \ref{thm:mon}, Theorem \ref{thm:lasserre} does \emph{not} straightforwardly imply Theorems \ref{thm:conv} and \ref{thm:mon}. Indeed, if we were to use Theorem~\ref{thm:lasserre} to show Theorem \ref{thm:conv}, we would have to find a polynomial $\Theta_{d'}$ such that $y^TH_{\Theta_{d'}}y=1+\sum_{i=1}^n x_i^{2d'}+\sum_{i=1}^{n} y_i^{2d'}$. This cannot be done as, by construction, the degree of $y_i$ in $y^T H_{\Theta_{d'}}y$ must be 2. We thus need to show another version of Theorem \ref{thm:lasserre} for polynomials of the structure $y^TM(x)y$, where $M(x) $ is a positive semidefinite matrix over $[-1,1]^n$. Even with this new version of Theorem \ref{thm:lasserre} in hand, showing Theorem \ref{thm:conv} is not plain sailing, as it requires us to find a polynomial whose Hessian has a very specific structure, namely diagonal with entries all equal to $1+\sum_{j=1}^n x_j^{2d'}$. Such a polynomial does not exist, so we need to find a way around this difficulty. Likewise, if we were to show Theorem \ref{thm:mon} as an immediate corollary of Theorem \ref{thm:lasserre}, we would need to find a polynomial $\Xi_{d'}$, whose partial derivatives with respect to $x_i$ are all equal to $\rho_i(1+\sum_{j=1}^n x_j^{2d'})$. Again, such a polynomial does not exist, due to the specific structures of partial derivatives. We thus proceed in a different way, starting with Propositions \ref{prop:conv} and \ref{prop:mon}, which are used as intermediate steps for the proofs of Theorems \ref{thm:conv} and \ref{thm:mon}.
	
	\begin{proposition} \label{prop:conv}
		Let $\Theta_{d'}$ be as defined in Theorem \ref{thm:conv}. For any $d'\geq 1$ and any $y=(y_1,\ldots,y_n)^T \in \mathbb{R}^n$,	$y^TH_{\Theta_d'}(x_1,\ldots,x_n)y - \sum_{i=1}^n y_i^2 (1+\sum_{j=1}^n x_j^{2d'})$ is sos.
	\end{proposition}
	
	\begin{proposition} \label{prop:mon}
		Let $\Xi_{d'}$ be as defined in Theorem \ref{thm:mon}. For any $d' \geq 1$ and for any $k=1,\ldots,n$, $ \rho_k \cdot \frac{\partial{\Xi_{d'}(x_1,\ldots,x_n)}}{\partial x_k}- \rho_k^2 \left(1+\sum_{i=1}^n x_i^{2d'}\right)$ is sos. 
	\end{proposition}
	
	These propositions are not a priori evident: the presence of cross-terms in $\Xi_{d'}$ and $\Theta_{d'}$ give rise to terms in the first and second-order derivatives which work against their being sums of squares, and we are requiring them to satisfy an even stronger condition. We now prove these two propositions, which requires the following lemma.
	
	\begin{lemma} \label{lem:matrix}
		Let $n \in \mathbb{N}$ and let $a_1> a_n>0$. Define $a_i=a_1 \cdot \left(\frac{a_n}{a_1}\right)^{\frac{i-1}{n-1}}>~0$ for $i=2,\ldots,n-1$. Let $M$ be the $n\times n$ symmetric matrix which is all zeros except for the entries
		$M_{i,i}=\frac{a_{i}}{2}$ for $i=1,2,n-1,n;$ $M_{i,i}=a_i$ for $i=3,\ldots,n-2;$ and $M_{(i-1),(i+1)}=M_{(i+1),(i-1)}=-\frac{a_i}{2}$ for $i=2,\ldots,n-1$.
		In other words, 
		\begin{align*}
		\footnotesize
		M=\begin{pmatrix} \frac{a_1}{2} & 0 & -\frac{a_2}{2} & 0 & \ldots \ldots & 0 & 0 & 0 & 0\\
		0 & \frac{a_2}{2} & 0 & -\frac{a_3}{2} & \ldots & 0 & 0 & 0 & 0\\
		-\frac{a_2}{2} & 0 & a_3 & 0 & \ldots & 0 & 0 & 0 & 0\\
		0 & -\frac{a_3}{2} & 0 & a_4 & \ldots & 0 & 0 & 0 & 0\\
		\vdots &\vdots & \vdots & \vdots & \ddots & \vdots & \vdots & \vdots & \vdots \\
		0 & 0 & 0 & 0 & \ldots & a_{n-3} & 0 & -\frac{a_{n-2}}{2} & 0 \\
		0 & 0 & 0 & 0 & \ldots & 0 & a_{n-2} & 0 & -\frac{a_{n-1}}{2}\\
		0 & 0 & 0 & 0 & \ldots & -\frac{a_{n-2}}{2}& 0 & \frac{a_{n-1}}{2} & 0\\
		0 & 0 & 0 & 0& \ldots & 0 & -\frac{a_{n-1}}{2} & 0 & \frac{a_n}{2}
		\end{pmatrix}.
		\end{align*}
		We have that $M$ is positive semidefinite. 
	\end{lemma}
\begin{proof}
	For $i=2,\ldots,n-1$, we define the $n \times n$ matrices $M^{(i)}$ to be all zeros except for four entries
	$M^{(i)}_{i-1,i-1}=\frac{a_{i-1}}{2}, M^{(i)}_{i+1,i+1}=\frac{a_{i+1}}{2},  M^{(i)}_{i-1,i+1}=M^{(i)}_{i+1,i-1}=-\frac{a_{i}}{2}.$
	Note that $\sum_{i=2}^{n-1} M^{(i)}=M$. We show that $M^{(i)}$, $i=2,\ldots,n-1$ is positive semidefinite, which proves the result. As $a_i>0$ for $i=1,\ldots,n,$ we have that 
	\begin{align*}
	M^{(i)} \succeq 0 \Leftrightarrow \begin{pmatrix} \frac{a_{i-1}}{2} & -\frac{a_i}{2} \\ -\frac{a_i}{2} & \frac{a_{i+1}}{2}\end{pmatrix} \succeq 0 \Leftrightarrow \frac{a_{i-1} \cdot a_{i+1}}{4} - \frac{a_i^2}{4}\geq 0.
	\end{align*}
	The result follows as we have $\frac{a_{i-1} \cdot a_{i+1}}{4}-\frac{a_i^2}{4}=\frac{a_1^2}{4} \cdot \left( \left( \frac{a_n}{a_1} \right)^{(i-1+i+1)/(n-1)} -\left( \frac{a_n}{a_1} \right)^{(2i)/(n-1)} \right)= 0.$
\end{proof}
	\begin{proof}[Proof of Proposition \ref{prop:conv}.] Let $d' \in \mathbb{N}$.
	For simplicity, we use $H$ for the Hessian $H_{\Theta_{d'}}$ of $\Theta_{d'}$. We have that 
	$y^THy=2\sum_{1 \leq k \leq n} y_k^2+\sum_{1 \leq k < \ell \leq n} (y_k^2 \tilde{H}_{kk} +2y_ky_{\ell}H_{kl}+y_{\ell}^2\tilde{H}_{\ell \ell}),$
	where
	\begin{align*}
	\tilde{H}_{kk}&=\frac{\partial^2(\Theta_{d'}(x)-\sum_{i=1}^n x_i^2)}{\partial x_{k}^2}= 8(2d'+2)(2d'+1)x_k^{2d'}+2d'(2d'-1)x_k^{2d'-2} x_{\ell}^2+2x_{\ell}^{2d'},\text{ for } k=1,\ldots,n\\
	H_{kl}&=\frac{\partial^2(\Theta_{d'}(x)-\sum_{i=1}^n x_i^2)}{\partial x_k \partial x_{\ell}}=\frac{\partial^2 \Theta_{d'}(x)}{\partial x_k \partial x_{\ell}}=4d' x_k^{2d'-1} x_{\ell}+4d' x_k x_{\ell}^{2d'-1} \text{ for } k, \ell=1,\ldots,n, ~k \neq \ell.
	\end{align*}
	
	We show that 
	$y^THy - 2\sum_{k=1}^n y_k^2 - \sum_{1 \leq k < \ell \leq n} \left( 16y_k^2 x_k^{2d'}+y_k^2 x_{\ell}^{2d'}+y_{\ell}^2 x_k^{2d'}+16y_{\ell}^2 x_{\ell}^{2d'}\right) $ is sos. This implies the result. To do so, we define, for $1 \leq k <\ell \leq n$,
	\begin{align*}
	q_{k\ell}(x_k, x_{\ell},y_k,y_{\ell})&\mathrel{\mathop{:}}= y_k^2\tilde{H}_{kk} +2y_ky_{\ell}H_{k \ell}+y_{\ell}^2 \tilde{H}_{\ell \ell}
	-16y_k^2x_k^{2d'}- y_k^2 x_{\ell}^{2d'}-y_{\ell}^2x_{k}^{2d'}-16y_{\ell}^2 x_{\ell}^{2d'}\\
	&=8d'(4d'+6)x_k^{2d'}y_k^2+2d'(2d'-1)x_k^{2d'-2}x_{\ell}^2y_k^2+x_{\ell}^{2d'}y_k^2+8d'(4d'+6)x_{\ell}^{2d'}y_{\ell}^2 \\
	&+2d'(2d'-1)x_{\ell}^{2d'-2} x_{k}^2 y_{\ell}^2 +x_{k}^{2d'} y_{\ell}^2+ 8d' x_k^{2d'-1} x_{\ell}y_k y_{\ell}+8d' x_k x_{\ell}^{2d'-1}y_{k}y_{\ell}
	\end{align*}
	and show that $q_{k\ell}$ is a sum of squares. More specifically, we show that there exists a positive semidefinite $Q \in \mathbb{R}^{(2d'+2) \times (2d'+2)}$ such that $q_{k\ell}(x_k, x_{\ell},y_k,y_{\ell})=w_{d'}(x,y)^T Q w_{d'}(x,y)$, where $w_{d'}(x,y)=\begin{bmatrix} y_k \cdot z_{d'}(x_k,x_{\ell}) \\ y_{\ell} \cdot z_{d'}(x_{\ell},x_k) \end{bmatrix}$ and $z_{d'}(x_k,x_{\ell})=(x_k^{d'}, x_k^{d'-1}x_{\ell}, \ldots, x_k x_{\ell}^{d'-1}, x_{\ell}^{d'})^T$ is a vector of monomials of size $d'+1$. Let $M$ be as in Lemma \ref{lem:matrix} with $n=d'$, $a_1=2d'(2d'-1)$, $a_{d'}=1$ and consider the following $(2d'+2) \times (2d'+2)$ matrices
	\begin{align*}
	\footnotesize{
		\tilde{M}=\left(\begin{array}{c|c}
		\begin{matrix}
		0 & \ldots & 0 \\
		\vdots & M & \\
		0 & & &
		\end{matrix}
		& \bigzero \\
		\hline
		\bigzero &
		\begin{matrix}
		0 & \ldots & 0 \\
		\vdots & M & \\
		0 & & &
		\end{matrix}
		\end{array}\right),
		\quad
		\tilde{P}=\left(\begin{array}{c|c}
		D
		& 0 \\
		\hline
		0 &
		D
		\end{array}\right),
	}
	\end{align*}
	where $D=\mbox{diag}\left(0,\frac{a_1}{2},\frac{a_2}{2},0, \ldots, 0, \frac{a_{d'}}{2}\right).$
	Note that $\tilde{M} \succeq 0$ by virtue of Lemma \ref{lem:matrix} and $\tilde{P} \succeq 0$ by definition of $\{a_i\}$. We further define the following $(2d'+2) \times (2d'+2)$ matrices:
	\begin{align*}
	\footnotesize{
		\tilde{N}_1=\left(\begin{array}{c|c}
		\begin{matrix}
		8d'(4d'+6) & 0 & \ldots &0 \\
		0 & 0 & \ldots & 0 \\
		\vdots & \vdots & \ddots & \vdots\\
		0 & 0 & \ldots &0
		\end{matrix}
		&
		\begin{matrix}
		0 & \ldots & 4d' &0 \\
		0 & \ldots & 0  & 0 \\
		\vdots & \ddots & \vdots  & \vdots\\
		0 & \ldots & 0 &0
		\end{matrix}\\
		\hline
		\begin{matrix}
		0 & 0 & \ldots &0 \\
		\vdots & \vdots & \ddots & \vdots\\
		4d' & 0 & \ldots & 0 \\
		0 & 0 & \ldots &0
		\end{matrix}
		&
		\begin{matrix}
		0 & \ldots & 0 & 0 \\
		\vdots & \ddots & \vdots & \vdots\\
		0 & \ldots & \frac{a_{d'-1}}{2} & 0 \\
		0 & \ldots  & 0 & 0
		\end{matrix}
		\end{array}\right), \quad
		\tilde{N}_2=\left(\begin{array}{c|c}
		\begin{matrix}
		0 & \ldots & 0 & 0 \\
		\vdots & \ddots & \vdots & \vdots\\
		0 & \ldots & \frac{a_{d'-1}}{2} & 0 \\
		0 & \ldots  & 0 & 0
		\end{matrix}
		&
		\begin{matrix}
		0 & 0 & \ldots &0 \\
		\vdots & \vdots & \ddots & \vdots\\
		4d' & 0 & \ldots & 0 \\
		0 & 0 & \ldots &0
		\end{matrix} \\
		\hline
		\begin{matrix}
		0 & \ldots & 4d' &0 \\
		0 & \ldots & 0  & 0 \\
		\vdots & \ddots & \vdots  & \vdots\\
		0 & \ldots & 0 &0
		\end{matrix}
		&
		\begin{matrix}
		8d'(4d'+6) & 0 & \ldots &0 \\
		0 & 0 & \ldots & 0 \\
		\vdots & \vdots & \ddots & \vdots\\
		0 & 0 & \ldots &0
		\end{matrix}
		\end{array}\right)
	}
	\end{align*}
	We also have that $\tilde{N}_1 \succeq 0$ and $\tilde{N}_2 \succeq 0$. Indeed, $a_{d'-1} >a_{d'}=1$ and
	$$\tilde{N}_1, \tilde{N}_2 \succeq 0 \Leftrightarrow \begin{pmatrix} 8d'(4d'+6) & 4d' \\ 4d' & \frac{a_{d'-1}}{2}  \end{pmatrix} \succeq 0 \Leftrightarrow 8d'(4d'+6) \cdot \frac{a_{d'-1}}{2} - 16d'^2> 16d'^2 -16d'^2=0.$$
	Now, take $Q=\tilde{M}+\tilde{P}+\tilde{N}_1+ \tilde{N}_2$. We have that $Q \succeq 0$ and that $q_{k\ell}(x_{k}, x_{\ell}, y_k, y_{\ell})=w_{d'}(x,y)^T Qw_{d'}(x,y)$. Thus $q_{k \ell}$ is sos and we can conclude. 
	\end{proof}
	
	\begin{proof}[Proof of Proposition \ref{prop:mon}.]
	Let $d' \in \mathbb{N}$ and $k \in \{1,\ldots,n\}$. We have that
	$$\frac{\partial \Xi_{d'}(x_1,\ldots,x_n)}{\partial x_k} = \rho_k^2+ \sum_{\ell \neq k} \left(4(2d'+1) \rho_k^2 x_k^{2d'}+2(2d') \rho_{\ell} \rho_k x_{k}^{2d'-1} x_{\ell} +2 \rho_{k}^2 x_{\ell}^{2d'}\right)$$ and we show that $$\frac{\partial \Xi_{d'}(x_1,\ldots,x_n)}{\partial x_k} -\rho_k^2 -\rho_k^2 \sum_{\ell \neq k}(4x_k^{2d'}+x_{\ell}^{2d'})$$ is a sum of squares. This implies the result. Note that in the case where $\rho_k=0$, this trivially holds. We now assume that $\rho_k=\pm 1$ (and so $\rho_k^2=1$) and define for $\ell \neq k$, 
	$$r_{kl}(x_k,x_{\ell}) \mathrel{\mathop{:}}=4(2d'+1) x_{k}^{2d'}+4d' \rho_{\ell} \rho_k x_k^{2d'-1}x_{\ell}+2x_{\ell}^{2d'}-4x_k^{2d'}-x_{\ell}^{2d'}=8d' x_{k}^{2d'}+4d' \rho_k \rho_{\ell} x_k^{2d'-1}x_{\ell}+x_{\ell}^{2d'}.$$
	We show that $r_{kl}(x_k,x_{\ell})$ is sos in $(x_k,x_{\ell})$ to conclude. Note that when $d'=1$, 
	$r_{kl}(x_k,x_{\ell})=8x_k^2 \pm 4x_kx_{\ell}+x_{\ell}^2=8\left(x_k \pm \frac14 x_{\ell}\right)^2+\frac12 x_{\ell}^2.$
	We now focus on $d' \geq 2$. As before, we let $z_{d'}(x_k,x_{\ell})=(x_{k}^{d'}, x_{k}^{d'-1}x_{\ell}, \ldots, x_{k}x_{\ell}^{d'-1}, x_{\ell}^{d'})^T$ be a vector of monomials of size $d'+1$. We consider $M$ as defined in Lemma \ref{lem:matrix} with $n=d'+1$, $a_1=8d'$, $a_{d'+1}=1$, and hence $a_i=8d' \cdot \left(\frac{1}{8d'} \right)^{(i-1)/d'}$ for $i=1,\ldots,d'$. As established in Lemma \ref{lem:matrix}, we have that $M \succeq 0$. Furthermore, we define 
	\begin{align*}
	\footnotesize{
		P=\begin{pmatrix} 4d' & 2d' & 0 & \ldots & 0 & 0 & 0\\ 
		2d' & \frac{a_2}{2} & 0 & \ldots & 0 & 0 & 0\\
		0 & 0 & 0 & \ldots & 0 & 0 &0 \\
		\vdots & \vdots & \vdots & \ddots &\vdots & \vdots & \vdots \\
		0 & 0 & 0 & \ldots & 0 & 0 & 0\\
		0 & 0 & 0 & \ldots & 0 & \frac{a_{d'}}{2} & 0\\
		0 & 0 & 0 & \ldots & 0 & 0 & \frac{a_{d'+1}}{2}
		\end{pmatrix}
	}
	\end{align*}
	
	We have that $P \succeq 0$. Indeed, $\frac{a_{d'}}{2}, \frac{a_{d'+1}}{2}, \frac{a_2}{2}>0$, and 
	\begin{align*}
	\begin{pmatrix} 4d' & 2d' \\ 2d' & \frac{a_2}{2} \end{pmatrix} \succeq 0 \Leftrightarrow 4d' \cdot \frac{a_2}{2} -4d'^2=4d'\cdot 4d' \left(\frac{1}{8d'} \right)^{1/d'}-4d'^2=4d'^2 \left( 4 \left( \frac{1}{8d'}\right)^{1/d'}-1 \right) \geq 0 \text{ for } d' \geq 2.
	\end{align*}
	Thus, $M+P \succeq 0$. Now, note that $z_{d'}(x_k,x_\ell)^T (M+P) z_{d'}(x_k,x_\ell)=r_{k \ell} (x_k,x_{\ell})$, and is thus sos.
	%
\end{proof}
	
	We give a lemma before proving Theorems \ref{thm:conv} and \ref{thm:mon}. Its proof is given in Appendix \ref{appdx:laurent_adjacent} for completeness and follows quite straightforwardly from the proof of \cite[Theorem 7.2]{laurent2009sums}.
	
	\begin{lemma} \label{lem:laurent_adjacent}
		Let $M(x_1,\ldots,x_n)$ be a $n \times n$ polynomial matrix which is positive semidefinite for all $x \in [-1,1]^n$ and let $y \in \mathbb{R}^n$. For $d' \in \mathbb{N}$, define $\Phi_{d'}(x,y)=\sum_{i=1}^n y_i^2 (1+\sum_{j=1}^n x_j^{2d'})$. For any $\epsilon>0$, there exists $d_0'$ such that $y^TM(x)y+\epsilon \Phi_{d'}$ is sos for all $d' \geq d_0'$.
	\end{lemma}
	
	\begin{proof}[Proof of Theorem \ref{thm:conv}.]
	Let $\epsilon>0$. From Lemma \ref{lem:laurent_adjacent}, as $g$ is convex over $[-1,1]^n$, there exists $d_0'$ such that $y^TH_g(x)y+\epsilon \Phi_{d'}$ is sos.  Combining this with Proposition \ref{prop:conv}, and the fact that $y^TH_{g_{\epsilon}}y=y^T H_g(x)y+ \epsilon \Phi_{d'}(x,y)+\epsilon \left( y^TH_{\Theta_{d'}}y- \Phi_{d'}(x,y) \right)$, we get that $y^TH_{g_{\epsilon}}y$ is sos for $d' \geq d_0'$.
	\end{proof}
	
	\begin{proof}[Proof of Theorem \ref{thm:mon}.] Let $\epsilon>0$ and $i \in I^+ \cup I^-$. From Theorem \ref{thm:lasserre}, as $h$ has $K$-bounded derivatives over $[-1,1]^n$, i.e., $\rho_i \cdot \left( \frac{\partial h(x)}{\partial x_i}-K_i^{\pm}\right) \geq 0$, for all $x \in [-1,1]^n$, there exists $d_0'$ such that $\rho_i \cdot \left( \frac{\partial h(x)}{\partial x_i}-K_i^{\pm}\right)+\epsilon \Psi_{d'}$ is sos for $d' \geq d_{0}'$. Combining this with Proposition \ref{prop:mon} and the fact that
	$$\rho_i \cdot \left( \frac{\partial h_{\epsilon}(x)}{\partial x_i}-K_i^{\pm}\right)=\rho_i \cdot \left( \frac{\partial h(x)}{\partial x_i}-K_i^{\pm}\right)+\epsilon \Psi_{d'}+\epsilon \left( \rho_i  \frac{\partial \Xi_{d'}(x)}{\partial x_i} - \Psi_{d'} \right),$$
	we get that  $\frac{\partial h_{\epsilon}(x)}{\partial x_i}-K_i^-$ (resp. $K^+_i-\frac{\partial h_{\epsilon}(x)}{\partial x_i}$) is sos for $i \in I^-$ (resp. $i \in I^+$).
	\end{proof}

	\subsection{Implications of the Approximation Results Beyond Shape-Constrained Regression} \label{subsec:implications}
	
	Theorems \ref{thm:conv} and \ref{thm:mon} may be of broader interest when reinterpreted as algebraic density results over the set of convex and monotone polynomials. Given a vector $\rho \in \{-1,0,1\}^n$, we say that a polynomial $f(x) \mathrel{\mathop{:}}=f(x_1,\ldots,x_n)$ is $\rho$-monotone if $\rho_i \cdot \frac{\partial f(x)}{\partial x_i} \geq 0$ for $i=1,\ldots,n$, and $\rho$-sos-monotone if $\rho_i \cdot \frac{\partial f(x)}{\partial x_i}$ is sos for $i=1,\ldots,n.$ Furthermore, for a polynomial $f$, we use $||f||_1$ for the $l_1$-norm of its coefficients.
	\begin{corollary}\label{cor:Lasserre}
		The following statements hold: 
		\begin{enumerate}[(i)]
			\item Let $f:\mathbb{R}^n \rightarrow \mathbb{R}$ be a polynomial that is convex over $[-1,1]^n$. There exists a sequence $\{f_{\epsilon}\}_{\epsilon}$ of sos-convex polynomials such that $||f-f_{\epsilon}||_{1} \rightarrow 0$ when $\epsilon \rightarrow 0$.
			\item Let $\rho \in \{-1,0,1\}^n$ and let $f:\mathbb{R}^n \rightarrow \mathbb{R}$ be a $\rho$-monotone polynomial over $[-1,1]^n$. There exists a sequence $\{f_{\epsilon}\}_{\epsilon}$ of $\rho$-sos-monotone polynomials such that $||f-f_{\epsilon}||_{1} \rightarrow 0$ when $\epsilon \rightarrow 0$.
		\end{enumerate}
	\end{corollary}
	
	\begin{proof}
	An explicit construction of such sequences of polynomials is obtained immediately by taking $f_{\epsilon}=f+\epsilon \Theta_{d'}$ for (i) and $f_{\epsilon}=f+\epsilon \Xi_{d'}$ for (ii), and noting that the coefficients of $\Theta_{d'}$,$\Xi_{d'}$, and thus $||\Theta_{d'}||_1$,$||\Xi_{d'}||_1$ do not depend on $\epsilon$ nor $d'$. 	
	\end{proof}
	
	Corollary \ref{cor:Lasserre} easily extends to any full-dimensional box $B \subset \mathbb{R}^n$. It states that sos-convex (resp. sos-monotone) polynomials are dense in the set of convex (resp. monotone) polynomials over $[-1,1]^n$ in the $l_1$ norm of coefficients, provided that we allow the degree of the sos-convex (resp. sos-monotone) polynomials to grow. It adds to a line of research studying the gap between convex and sos-convex polynomials (see, e.g., \cite{ahmadi2013complete,helton2010semidefinite}) and can be viewed as an analog---for convexity and monotonicity---to a result in \cite{lasserre2007sos} showing that the set of sos polynomials is dense in the set of nonnegative polynomials on $[-1,1]^n$. Corollary~\ref{cor:Lasserre} also holds if convexity and $\rho$-monotonicity are global (rather than over a box). It can thus be viewed as an analog of a result in \cite{lasserre2007sum} as well. Similarly to these papers, we propose simple and explicit expressions of sequences $\{f_{\epsilon}\}$ that approximate $f$, and these are obtained by adding a small perturbation to $f$ with high-degree terms. Importantly, the sequences $\{f_{\epsilon}\}$ built in our case are \emph{not} those given in the aforementioned papers (that is, they do not involve $\Psi_{d'}$, but $\Theta_{d'}$ and $\Xi_{d'}$). As explained before, the structure of Hessians and gradients precludes us from using $\Psi_{d'}$. For the convex case, these positive results regarding approximation of convex polynomials by sos-convex polynomials are the counterpart of existing negative results when the degree of the sos-convex polynomials is fixed. In that case, it can be shown as a consequence of \cite{blekherman2006there} that there are many more convex polynomials than there are sos-convex polynomials.


\subsection{Consistency of the SOSEs} \label{subsec:consistency}

In this section, we make three statistical assumptions on the way $(X_i,Y_i)_{i=1,\ldots,m}$ are generated, which are standard for consistency results.

\begin{assumption}\label{assmpt:generation.X}
	The vectors $X_1,\ldots,X_m \in \mathbb{R}^n$ are randomly generated and are independently and identically distributed (iid) with $E[||X_1||^2]<\infty$.
\end{assumption}

\begin{assumption}\label{assmpt:box}
	The support of the random vectors $X_1$, $\ldots$, $X_m$ is a full-dimensional box $B \subseteq \mathbb{R}^n$ as in (\ref{eq:box}), i.e., $P(X_i \in B)=1$. Furthermore, for any full-dimensional set $C \subseteq B$, $P(X_i \in C)>0$.
\end{assumption}

\begin{assumption}\label{assmpt:generation.Y}
	There exists a continuous function $f:B \rightarrow \mathbb{R}$ such that 
	$Y_i=f(X_i)+\nu_i$ for all $i=1,\ldots,m,$
	where $\nu_i$ are random variables with support $\mathbb{R}$ and the following characteristics:
	\begin{align*}
	&P(\nu_1 \in dz_1,\ldots, \nu_m \in dz_n|X_1,\ldots,X_m)=\prod_{i=1}^m P(\nu_i \in dz_i|X_i), \forall z_1,\ldots ,z_m \in \mathbb{R}, \\
	&E[\nu_i|X_i]=0 \text{ almost surely (a.s.), } \forall i=1,\ldots,m, \quad E[\nu_i^2]=\mathrel{\mathop{:}}\sigma^2<\infty~\forall i=1,\ldots,m.
	\end{align*}
\end{assumption}
Assumptions \ref{assmpt:generation.X} and \ref{assmpt:generation.Y} imply that the sequence $\{(X_i,Y_i)\}_{i=1,\ldots,m}$ is iid, that $E[\nu_1]=0$, and that $E[Y_1^2]<\infty$. Using these three assumptions, we show \emph{consistency} of the SOSEs. This is a key property of estimators stating that, as the number of observations grows, we are able to recover $f$.

\begin{theorem} \label{th:gmdr.consistent}
	Let $C$ be any closed full-dimensional subset of $B$ such that no point on the boundary of $B$ is in $C$. 
	Assuming that $f$ is twice continuously differentiable and convex over $B$, that $\tg$ is as defined in (\ref{eq:opt.tg}), and that Assumptions 1 through 3 hold, we have, for any fixed $r \in \mathbb{N}$,
	\begin{align}
	\sup_{x \in C} |\tg(x)-f(x)| \rightarrow 0 \text{ almost surely (a.s.) as } d, m\mathrel{\mathop{:}}=m(d) \rightarrow \infty,
	\end{align}
	i.e., $P(\lim_{d,m \rightarrow \infty} \sup_{x \in C} |\tg(x)-f(x)| \rightarrow 0)=1$.
\end{theorem}

\begin{theorem}\label{th:hmdr.consistent}
	Let $C$ be any closed full-dimensional subset of $B$ such that no point on the boundary of $B$ is in $C$. Assume that $f$ has $K$-bounded derivatives over $B$ with $I^+ \cap I^-=\emptyset$, that $\thh$ is as defined in (\ref{eq:opt.th}), and that Assumptions 1 through 3 hold, we have, for any fixed $r \in \mathbb{N},$
	\begin{align}
	\sup_{x \in C} |\thh(x)-f(x)| \rightarrow 0 \text{ a.s. as } d, m\mathrel{\mathop{:}}=m(d)\rightarrow \infty.
	\end{align}
\end{theorem}

\begin{remark}
	Under our assumptions, $d$ and $m$ need to go to infinity for consistency to hold and it is quite clear that these assumptions cannot be dropped. (The same assumptions are needed to show consistency of, e.g., unconstrained polynomial regression.) If we allow $r \rightarrow \infty$, it is easier to show consistency, as we are then able to leverage certain Positivstellens\"atze \cite{putinar1993positive,scherer2006matrix}. As $r$ is fixed in Theorems \ref{th:gmdr.consistent} and \ref{th:hmdr.consistent}, we do not make use of such results here. Note that one can fix $r$ to any value, including $r=0$, and still obtain consistency as long as $m,d \rightarrow \infty$. In practice, we choose $d$ and $r$ via cross-validation (see Section \ref{subsec:sose.computation}) as it is not clear that $r=0$ is preferable to larger $r$ in terms of generalization error.
\end{remark}


\begin{remark}
	One could extend these theorems to the box $B$ itself, provided that we make assumptions on the sampling of the pairs of points $(X_i,Y_i)_{i=1,\ldots,m}$ on the boundary of $B$.
\end{remark}

Theorems \ref{th:gmdr.consistent} and \ref{th:hmdr.consistent} rely on three propositions and require the introduction of two sets of polynomials $g_d$, $h_d$ and $\bg$, $\bh$. Let $C_{n,d}$ (resp. $K_{n,d}$) be the set of $n$-variate polynomials of degree $d$ that are convex (resp. have $K$-bounded derivatives) on $B$.

\begin{proposition} \label{prop:weierstrass}
	The following two statements hold:
	\begin{enumerate}[(i)]
		\item Let $f:\mathbb{R}^n \rightarrow \mathbb{R}$ be a twice continuously differentiable function that is convex over a box $B \subset \mathbb{R}^n$. Define $g_d$ to be one of the minimizers of $\min_{g \in C_{n,d}} \sup_{x \in B} |f(x)-g(x)|.$
		For any $\epsilon>0$, $\exists d$ such that
		$\sup_{x \in B} |g_d(x)-f(x)| <\epsilon.$
		\item Let $f:\mathbb{R}^n \rightarrow \mathbb{R}$ be a continuously differentiable function with $K$-bounded derivatives over a box $B \subset \mathbb{R}^n$. Define
		$h_d$ to be one of the minimizers of $\min_{g \in K_{n,d}} \sup_{x \in B} |f(x)-g(x)|.$
		For any $\epsilon>0,~ \exists d$ such that $\sup_{x \in B} |h_d(x)-f(x)| <\epsilon.$
	\end{enumerate}
\end{proposition}
The polynomials $g_d$ and $h_d$ are guaranteed to exist following Appendix A in \cite{bertsekasnonlin}, though they are not necessarily unique. Given $m$ feature vector-response variable pairs $(X_i,Y_i)_{i=1,\ldots,m}$, define now a convex-constrained regressor $\bg:\mathbb{R}^{n} \rightarrow \mathbb{R}$ and a bounded-derivatives regressor $\bh:\mathbb{R}^{n} \rightarrow \mathbb{R}$ as the\footnote{We once again assume that the feature vectors $\{X_i\}_{i=1,\ldots,m}$ are linearly independent and that $m$ is large enough.} solutions to the following least-squares regression problems:
\begin{equation}\label{eq:opt.bg}
\begin{aligned}
\bg \mathop{\mathrel{:}}=\arg &\min_{g \in P_{n,d}} &&\sum_{i=1}^m (Y_i-g(X_i))^2\\
&\text{s.t. } &&H_g(x) \succeq 0, \forall x\in B,
\end{aligned}
\end{equation}
and
\begin{equation}\label{eq:opt.bh}
\begin{aligned}
\bh \mathop{\mathrel{:}}=\arg &\min_{h \in P_{n,d}} &&\sum_{i=1}^m (Y_i-h(X_i))^2\\
&\text{s.t. } &&\frac{\partial f(x)}{\partial x_i} \geq K_i^-, ~\forall x\in B, i \in I^-, \quad \frac{\partial f(x)}{\partial x_i} \leq K_i^+, ~\forall x \in B,~i \in I^+.
\end{aligned}
\end{equation}

\begin{proposition} \label{prop:lym.glyn}
	Let $C$ be any closed full-dimensional subset of $B$ such that no point on the boundary of $B$ is in $C$ and let $d \in \mathbb{N}$. Assuming that Assumptions 1 through 3 hold, we have:
	\begin{enumerate} [(i)]
		\item If $\bg$ is as in (\ref{eq:opt.bg}), $g_d$ is as in Proposition \ref{prop:weierstrass}(i), then
		$\lim_{m \rightarrow \infty }\sup_{x \in C} |\bg(x)-g_d(x)| = 0 \text{ a.s.}$
		\item If $\bh$ is as in (\ref{eq:opt.bh}), $h_d$ is as in Proposition \ref{prop:weierstrass}(ii), then
		$\lim_{m \rightarrow \infty }\sup_{x \in C} |\bh(x)-h_d(x)|= 0 \text{ a.s.}$
	\end{enumerate}
\end{proposition}

	\begin{proposition} \label{prop:approx.consistency}
		Let $m,d \in \mathbb{N}$. Fix $r \in \mathbb{N}$ (including $r=0$).
		\begin{enumerate}[(i)]
			\item Let $\bg$ be as in (\ref{eq:opt.bg}) and $\tilde{g}_{m,d',r}$ as in (\ref{def:tg}). Then $\lim_{d' \rightarrow \infty} \sup_{x\in B} |\bg(x)-\tilde{g}_{m,d',r}(x)|=0$.
			\item Let $\bh$ be as in (\ref{eq:opt.bh}) and $\tilde{h}_{m,d',r}$ as in (\ref{def:th}). Then, $\lim_{d' \rightarrow \infty} \sup_{x\in B} |\bh(x)-\tilde{h}_{m,d',r}(x)|=0$.
		\end{enumerate}
	\end{proposition}

The proofs of Theorems \ref{th:gmdr.consistent} and \ref{th:hmdr.consistent} follow immediately from the triangle inequality and the three propositions above. The proof of Proposition \ref{prop:weierstrass} is given in Appendix \ref{appendix:weierstrass}, that of Proposition \ref{prop:lym.glyn} in Appendix \ref{appendix:lim.glyn}, and that of Proposition \ref{prop:approx.consistency} in Appendix \ref{appendix:approx.consistency}. Proposition \ref{prop:weierstrass} can be viewed as a shape-constrained Weierstrass theorem. To the best of our knowledge, the result and its proof are new to the literature. The proof of Proposition \ref{prop:lym.glyn} is similar to the proof of consistency of the CLSE in \cite{lim2012consistency}. However,  we cannot directly apply the results in \cite{lim2012consistency}, as they assume that $X_1,\ldots,X_m$ are sampled from $\mathbb{R}^n$, which is an assumption that we cannot make (in light of our Weierstrass-type results). This requires us to rework parts of the proof given in \cite{lim2012consistency}. The proof of Proposition \ref{prop:approx.consistency} relies on Theorems \ref{thm:conv} and \ref{thm:mon} and is not as straightforward as would be perhaps expected.

\section{Connection to Training-Optimal Shape-Constrained Polynomial Regressors} \label{sec:sc.pr}

Recall the definitions of the regressors $\tg, \thh, \bg$, and $\bh$ as defined in (\ref{def:tg}), (\ref{def:th}), (\ref{eq:opt.bg}), and (\ref{eq:opt.bh}). It follows from Positivstellens\"atze by \cite{putinar1993positive} and by \cite{scherer2006matrix} that $\lim_{r \rightarrow \infty} \tg=\bg$ and that $\lim_{r \rightarrow \infty} \thh=\bh$. It is thus natural to wonder how the SOSEs, $\tg$ and $\thh$, compare to their limits, $\bg$ and $\bh$.

 The first notable difference between the SOSEs and their limits relates to their computation. As seen in Section \ref{subsec:sose.computation}, $\tg$ and $\thh$ can be obtained for any fixed $r$ by solving an SDP. By contrast, $\bg$ and $\bh$, however, are NP-hard to compute in general, as we prove next. More specifically, we provide a complete classification of the complexity of computing $\bg$ and $\bh$ based on their degree $d$. We work in the standard Turing model of computation (see, e.g., \cite{sipser1996introduction}), where the input to the problem $(\{X_i,Y_i\}_{i=1,\ldots,m}, B \subset \mathbb{R}^n)$ is described by rational numbers.
	\begin{definition}
		Let CONV-REG-$d$ (resp. BD-DER-REG-$d$) be the following decision problem. Given a box $B \subset \mathbb{R}^n$ as in (\ref{eq:box}) (i.e., pairs $(l_i,u_i)$, $i=1,\ldots,n$), data points $(X_i,Y_i) \in B \times \mathbb{R}$, for $i=1,\ldots,m$, and a scalar $t$ (resp. a vector $K$ as in Definition \ref{def:bdr}, and a scalar $t$), decide whether there exists a polynomial $p$ of degree $d$ that is convex over $B$ (resp. with $K$-bounded derivatives over $B$) such that 
		$\sum_{i=1}^{m} (Y_i-p(X_i))^2 \leq t.$
	\end{definition}
	
	We use the standard notation $P$ to refer to the class of polynomially-solvable decision problems in the Turing model of computation. 
	
	\begin{theorem}\label{th:charac}
		BD-DER-REG-$d$ is strongly NP-hard for $d\geq 3$ and is in P for $d\leq 2$. CONV-REG-$d$ is strongly NP-hard for $d\geq 3$, can be reduced to a semidefinite program (SDP) of polynomial size for $d=2$, and is in P for $d=1$.
	\end{theorem}
	
	The proof of Theorem \ref{th:charac} is in Appendix~\ref{appdx:charc}. Unless $P=NP$, this Theorem \ref{th:charac} suggest that for $d \geq 3$ and some fixed $r$, it cannot be the case in general that $\bg$ and $\bh$ are equal to $\tg$ and $\thh$. When certain additional structural assumptions are made, this can change: for example, as seen above, when $d=1$ or $d=2$. Likewise, when $\bh$ is a separable polynomial (see the results by Polya and Szego reviewed in, e.g., \cite{powers2000polynomials}) or when $\bg$ is a separable plus quadratic polynomial (see \cite{ahmadi2021sums}), we have that $\bh=\thh$ and $\bg =\tg$ for known and small values of $r$. In the case where no additional structural assumptions are made, i.e., when the estimators can be different a priori, one may ask whether one set of estimators is better suited to shape-constrained regression than the other.  
		
		One can compare, for some fixed $r$, the performance of $\tg$ (resp. $\thh$) against $\bg$ (resp. $\bh$) without computing $\bg$ (resp. $\bh$), which (unless $P=NP$) would require exponential time. It is quite clear, as the feasible set of (\ref{def:tg}) (resp. (\ref{def:th})) inner approximates that of (\ref{eq:opt.bg}) (resp. \ref{eq:opt.bh}), that, for any $r$, the \emph{training} RMSE of $\bg$ and $\bh$ is always going to be lower than that of $\tg$ and $\thh$. In this sense, $\bg$ and $\bh$ are \emph{training-optimal} shape-constrained polynomial regressors. Is the training performance of $\tg$ and $\thh$ close for a given value of $r$? To ascertain this, we propose a procedure, relying on convex optimization (SDPs in the case of convex regression, and QPs in the case of bounded-derivative regression), which checks on any instance how close $\tg, \thh$ and $\bg, \bh$ are in terms of training error. We present it for $\bg$ and $\tg$, but the same idea applies to $\bh$ and $\thh$. We select $N'$ points $z_1,\ldots,z_{N'} \in B$ (e.g., uniformly at random) and solve
		\begin{equation} \label{eq:sample.opt}
		\begin{aligned}
		&\min_{g \in P_{n,d}} &&\sum_{i=1}^m (Y_i-g(X_i))^2\\
		&\text{s.t. } &&H_g(z_i) \succeq 0, \text{ for } i=1,\ldots,N'.
		\end{aligned}
		\end{equation}
		Its objective value is a lower bound on the optimal value of (\ref{eq:opt.bg}), as the convexity constraint of (\ref{eq:opt.bg}) has been relaxed. We then compare this lower bound to the upper bound on the optimal value of (\ref{eq:opt.bg}) given by the optimal value of the SDP in (\ref{def:tg}). If the upper and lower bounds are close, then $\tg$ is close to training-optimal. We examine this approach on a dataset generated as described in Appendix \ref{appendix:specs} with $f=f_1$ and $f=f_2$. In Table \ref{tab:ratio}, we give the ratio of the objective of (\ref{eq:sample.opt}) against that of (\ref{eq:opt.tg}) for these datasets and different values of $d, N'$ (we take $n=6,r=1, m=1000$). We observe that, even for very low $r$, the objective value of (\ref{eq:sample.opt}) is very close to that of (\ref{def:tg}), particularly when $N'$ is large. Similar results occur when considering the applications in Section \ref{sec:apps}: for Section \ref{subsec:KLEMS}, using $N'=100$, $d=4$, $r=2$, we get an average ratio across industries of $0.92$; for Section \ref{subsec:opt.transport}, using $N'=100$, $d=5$, $r=1$, we get an average of $0.96$ across various values of $\ell$ and $L$ (see Appendix \ref{appendix:add.comp}). This indicates that in practice, at least, the SOSEs are close to training-optimal in all our applications, even for small $r$.
		
		It is important however to note that low training error is not the criterion against which we measure the quality of our regressors. Of much more importance is low testing error. When accounting for this criterion, lower training error can sometimes be viewed as a hindrance rather than an advantage, as it can indicate overfitting to the data at hand. In practice, we observe (see Figures \ref{fig:train_test_rmse} and \ref{fig:train_test_rmse_2}) that the SOSEs do no worse for small $r$ than they do for larger $r$ in terms of testing error. (In fact, in some cases, increasing $r$ can lead to a marginal increase of testing error.) Thus, the SOSEs are not necessarily less valuable than the training-optimal $\bg$ and $\bh$. On the contrary, they seem to perform equally well or marginally better in terms of testing error, even for small $r$, and have the advantage of being efficiently computable. This makes them estimators of interest in their own right, independently of $\bg$ and $\bh$.

		\begin{table}
				\centering
			\small
			\begin{tabular}{|c|c|c|c|c|}
				\hline
				& \multicolumn{2}{c}{Fitting to $f_1$} \vline & \multicolumn{2}{c}{Fitting to $f_2$} \vline \\
				\hline
				degree $d$ & N'=100 & N'=1000 & N'=100 & N'=1000 \\
				\hline
				2 &  0.999 &  0.999 &  0.995 & 0.996 \\
				3 & 0.978 & 0.997 & 0.989 & 0.996 \\
				4 & 0.936 & 0.984 & 0.9428 & 0.971 \\
				5 & 0.792 & 0.937 & 0.834 & 0.932 \\
				6 & 0.446 & 0.801 & 0.433 & 0.791\\
				\hline	 
			\end{tabular}
			\caption{Average ratio of the objective value of the sample problem against that of the sos-shape constrained problem for $n=6$, $r=1$, and $m=1000$, for different values of $N'$ and $d$ and functions $f_1$ and $f_2$.}
			\label{tab:ratio}
		\end{table}

		\section{Applications of Shape-Constrained Regression to Economics, Real-Time Optimization, and Optimal Transport} \label{sec:apps}
		In this section, we present three applications of the SOSEs. The first is fitting a production function to data (Section \ref{subsec:KLEMS}). It is a well-known application of shape-constrained regression in economics and we show that it outperforms the prevalent approach there. The second is predicting the optimal value of a conic program and is, to the best of our knowledge, novel (Section \ref{subsec:opt.val}). The third is in optimal transport, more specifically color transfer (Section \ref{subsec:opt.transport}), and has not been tackled using sum of squares methodology previously, despite it being a very relevant setting for the SOSEs.
		
		\subsection{Fitting a Production Function to Data}\label{subsec:KLEMS}
		
		The goal of this application is to estimate the functional relationship between the yearly inputs of Capital ($K$), Labor ($L$), and Intermediate goods ($I$) to an industry, and the yearly gross-output production $Out$ of that industry. As $Out$ is assumed to be a decreasing function in $K,L,$ and $I$, as well as concave in $K$, $L,$ and $I$ by virtue of diminishing returns, an estimator constrained to have this shape is desirable. Traditionally, in the economics literature, this is done by fitting a \emph{Cobb-Douglas production function} to the data, i.e., finding $(a,b,c,d)$ such that the function
		$Out=a \cdot K^b \cdot L^c \cdot I^d$
		is as close as possible to the observed data. The advantage of such an approach is that it can be couched as a linear regression problem by working in log-space, with the shape constraints being imposed via the constraints $b,c,d\geq0$, $b+c+d\leq 1$ and $a \geq 0$. We compare the SOSE to the Cobb-Douglas estimator. To fit the estimators, we consider the USA KLEMS data (available at \url{http://www.worldklems.net/data.htm}), which contains yearly gross-output production data $Out$ for 65 industries in the US, from 1947 to 2014 as well as  yearly inputs of Capital $K$, Labor $L$, and Intermediate goods $I$, adjusted for inflation. Since the data is temporal, we perform a temporal split for our training-testing splits. We then fit the Cobb-Douglas estimator and the SOSE with degree $d=4$ and $r=2$ and the aforementioned shape constraints to the data. The results obtained are given in Figure \ref{fig:KLEMS}. As can be seen, our method outperforms the traditional Cobb-Douglas technique on 50 out of the 65 industries, sometimes quite significantly.
		\begin{figure}[h]
			\begin{center}
				\begin{subfigure}{.4\textwidth}
					\centering
					\includegraphics[width=1\linewidth]{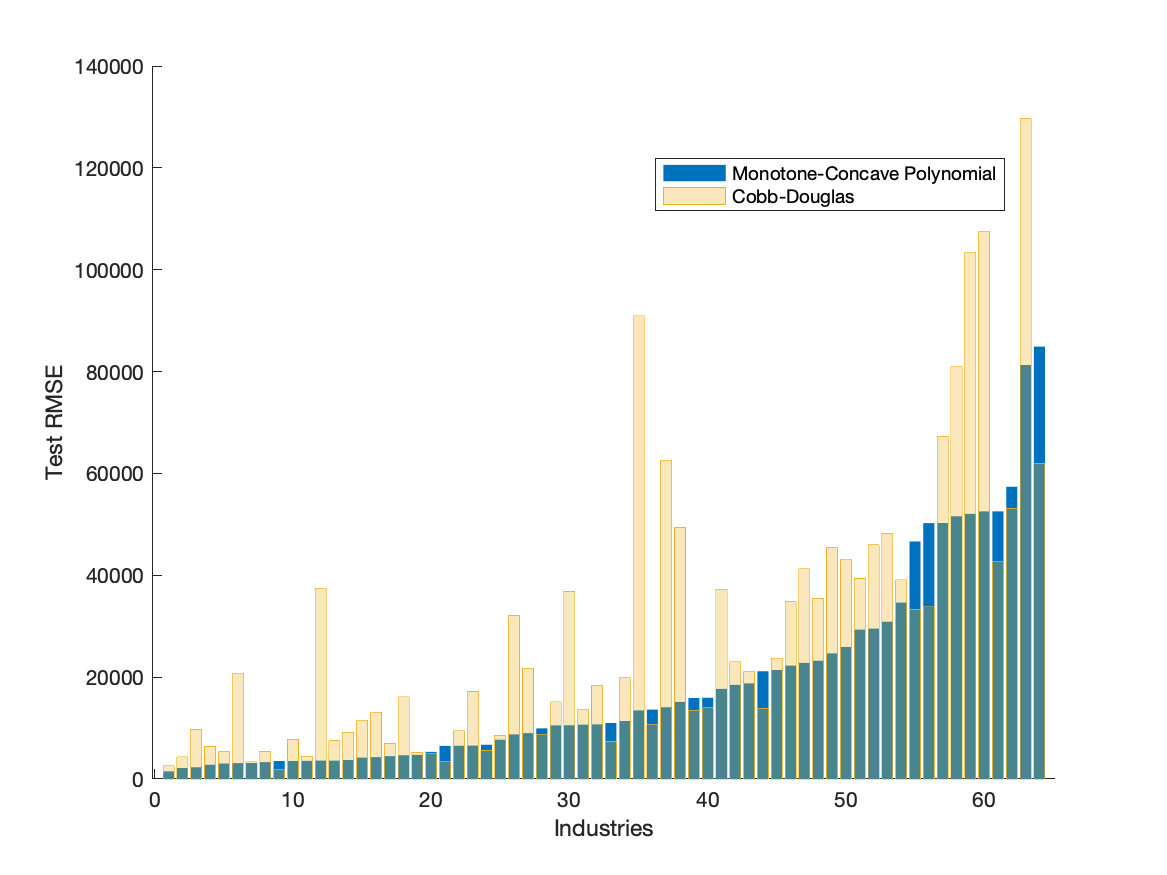}  
					\caption{}
					\label{fig2:a}
				\end{subfigure}
				\begin{subfigure}{.4\textwidth}
					\centering
					\includegraphics[width=1\linewidth]{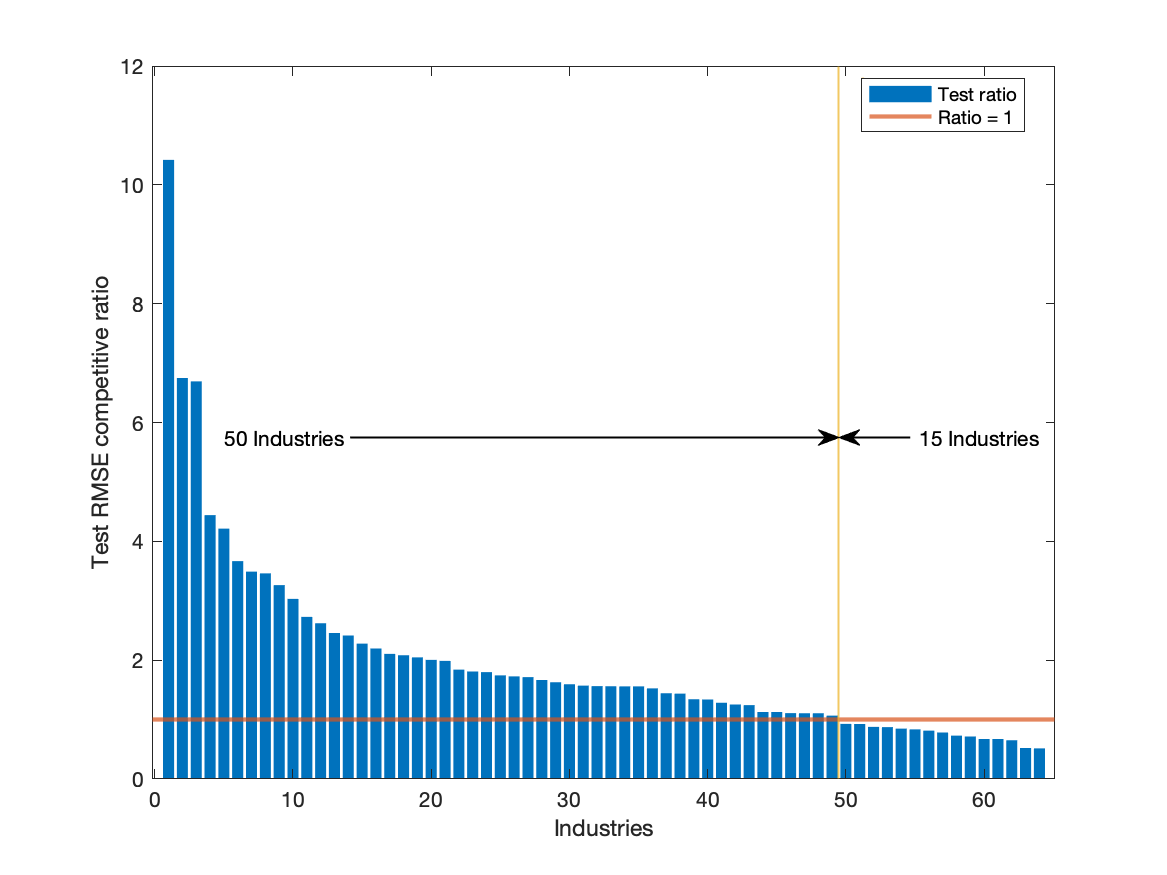}  
					\caption{}
					\label{fig2:b}
				\end{subfigure}
			\end{center}
			\caption{Comparison of the Test RMSE for the Cobb-Douglas production functions and the SOSE with the same shape constraints (concavity and monotonicity) across 65 industries; see Section \ref{subsec:KLEMS}. In Figure \ref{fig2:a}, the test RMSE values obtained, and in Figure \ref{fig2:b}, the ratio of the Cobb-Douglas RMSE over the SOSE RMSE across industries.}
			\label{fig:KLEMS}
		\end{figure}


		\subsection{Predicting the Optimal Value of a Conic Program}\label{subsec:opt.val}
		
		Let $\mathcal{K} \subseteq \mathbb{R}^n$ be a proper cone, $\langle \cdot, \cdot \rangle$ be an inner product on $\mathbb{R}^n \times \mathbb{R}^n$, $A \in \mathbb{R}^{m \times n}$, $b\in \mathbb{R}^m$, and $c\in \mathbb{R}^n$. We denote by $\mathcal{K}^*$ the dual cone associated to $(\mathcal{K}, \langle \cdot, \cdot \rangle)$ and by $A^*$ the adjoint operator of $A$. We also use $x \succeq_{\mathcal{K}} y$ to mean that $x-y \in \mathcal{K}$ and we say that $f:\mathbb{R}^n \mapsto \mathbb{R}$ is $\mathcal{K}$-nonincreasing if $x \succeq_{\mathcal{K}} y \Rightarrow f(x) \leq f(y)$. We now consider the pair of primal-dual conic programs:
		
		\begin{minipage}{.34\linewidth}
			\begin{equation} \tag{$P$} \label{eq:conic.P}
			\begin{aligned}
			v_P(b,c) \mathrel{\mathop{:}}=&\inf_{x \in \mathbb{R}^n} ~\langle c,x \rangle\\
			&\text{s.t. } Ax \preceq_{\mathcal{K}}b
			\end{aligned}
			\end{equation}
		\end{minipage}
		\quad
		\begin{minipage}{.56\linewidth}
			\begin{equation} \tag{$D$} \label{eq:conic.D}
			\begin{aligned}
			v_D(b,c)\mathrel{\mathop{:}}=&\sup_{y \in \mathbb{R}^m} ~-\langle b,y\rangle\\
			&\text{s.t. }  A^*y+c=0, \quad y \succeq_{\mathcal{K}^*} 0.
			\end{aligned}
			\end{equation}
		\end{minipage}
		
		\noindent We assume that strong duality holds, which implies that $v_P(b,c)=v_D(b,c)=\mathrel{\mathop{:}} v(b,c)$. As it turns out, $v(b,c)$ satisfies a number of shape constraints.
		
		\begin{proposition}\label{prop:opt.val.proof}
			Assume that strong duality holds for the primal-dual pair (\ref{eq:conic.P})-(\ref{eq:conic.D}). Then, the function $v(b,c)$ is (i) convex in $b$, (ii) concave in $c$, (iii) $\mathcal{K}$-nondecreasing in $b$.
		\end{proposition}
		
		The proof is given in Section \ref{appendix:sec.apps}. It suggests a possible use of the SOSE, and shape-constrained regression more generally: compute the SOSE of $v(b,c)$ and then use it to obtain quick predictions of the optimal value of (\ref{eq:conic.P}) for any new set of parameters $(b,c)$, without having to solve (\ref{eq:conic.P}) explicitly. This can prove useful for real-time decision making, as we see now.
		
		\subsubsection{An Application to Real-Time Inventory Management Contract Negotiation} \label{subsubsec:inventory}
		
		In a single-product inventory with a finite-time horizon, the state of the inventory at time $t=1,2,\ldots,T$ is specified by the amount $x_t \in \mathbb{R}$ of product in the inventory at the beginning of period $t$. During the period, the retailer orders $q_t \geq 0$ units of product from the supplier (we assume it arrives immediately) and satisfies external demand for $d_t \geq 0$ units of the product. Thus, the state equation of the inventory is given by $x_{t+1}=x_t+q_t-d_t$. We assume $x_1=0$ and allow for $x_t \leq 0$ (backlogged demand). We further enforce a minimum amount $L$ that the retailer needs to buy, i.e., $\sum_{t=1}^T q_t \geq L$. The retailer wishes to minimize the overall inventory management cost. To this effect, we let $h,p,c,s \geq 0$ with $h+p \geq s$ be the respective costs
		per period and per unit, of storing the product, backlogged demand, replenishing the inventory, and salvaging the product \cite{ben2009robust}. Following \cite{ben2005retailer}, we further assume that the supplier and retailer agree on a \emph{flexible commitment contract}: at time $t=0$, the retailer must commit to projected future orders, $w_t \in \mathbb{R}, 1\leq t \leq T$. These do not have to be fully respected, but a penalty $\alpha^{\pm}$ will be incurred per unit of excess/recess of the actual orders $q_t$ as compared to commitments $w_t$ and a penalty $\beta^{\pm}$ will be incurred for variations in the commitment $w_t$ across periods. The problem that the inventory manager has to solve to obtain the minimum-cost inventory is thus:
		\begin{equation}\label{eq:cost.inventory}
		\begin{footnotesize}
		\begin{aligned}
		&\min_{x_t,q_t,w_t} &&\sum_{t=1}^T \left(h \max \{x_{t+1},0\} + p \max \{0,-x_{t+1}\} + \alpha^+ \max \{ q_t-w_t,0\}+\alpha^- \max \{w_t-q_t,0\} \right)\\
		& && +\sum_{t=1}^T c q_t- s \max \{x_{T+1},0\}+ \sum_{t=2}^T \left( \beta^+ \max \{w_t-w_{t-1},0\} + \beta^- \max \{w_{t-1}-w_t,0\}\right)\\
		&\text{s.t. } &&x_{t+1}=x_t+q_t-d_t, t=1,\ldots,T, \quad x_1=0, \quad \sum_{t=1}^T q_t \geq L, \quad q_t\geq 0,~t=1,\ldots,T.
		\end{aligned}
		\end{footnotesize}
		\end{equation}
		The demand $d\mathrel{\mathop{:}}=(d_1,\ldots,d_T)^T$ is assumed to be uncertain however, belonging either to a box $\mathcal{S}_B \subseteq \mathbb{R}^T$ or to an ellipsoid $\mathcal{S}_E \subseteq \mathbb{R}^T$. We also suppose that $q_t$ depends affinely on $d_t$, as done in \cite{ben2005retailer}, i.e., we write:
		$q_t=q_t^0+\sum_{\tau=1}^{t-1} q_t^{\tau} d_{\tau}, t=1,\ldots,T,$ where $\{q_t^{\tau}\}_{t,\tau}$ are additional variables. Using the state equation to get rid of variables $\{x_t\}$ and introducing new variables $\{y_t\}$, $\{u_t\}$, $\{v_t\}$, $\{z_t^{\pm}\}$ which encode the $\max$ functions, we rewrite the worst-case formulation of (\ref{eq:cost.inventory}) as the following semi-infinite program (see \cite{ben2005retailer} for details):
		\begin{equation}\label{eq:semi.inf.opt}
		\begin{footnotesize}
		\begin{aligned}
		&\min_{C,w_t,z_t^{\pm}, y_t^{\tau}, u_{t}^{\tau}, v_{t}^{\tau},q_{t}^{\tau}} &&\beta^+ \cdot \sum_{t=2}^T z^+_t+\beta^- \sum_{t=2}^{T} z^-_t +  C \\
		&\hspace{11mm} \text{s.t. }&&z_t^+ \geq w_t-w_{t-1},~ z_t^+ \geq 0,\quad z_t^- \geq w_{t-1}-w_{t},~z_t^- \geq 0,\quad t=2,\ldots,T,\\
		&\hspace{11mm} \forall d \in \mathcal{S}_{B/E}: &&C \geq \sum_{t=1}^T \left(y_0^t+cq_0^t+\alpha^+ u_0^t+\alpha^- v_0^t +\sum_{\tau=1}^{t-1} \left(y^t_{\tau} +cq^t_{\tau}+\alpha^+ u^t_{\tau}+\alpha^- v_{\tau}^t\right) d_{\tau}\right)\\
		& &&y_0^{t}+\sum_{\tau=1}^{t-1} y_{\tau}^{t} d_{\tau} \geq   \bar{h}_t \sum_{\tau=1}^{t}\left( q_0^{\tau}+\sum_{\sigma=1}^{\tau -1} q_{\sigma}^{\tau}d_{\sigma}-d_{\tau}\right),~t=1,\ldots,T, \\
		& &&y_0^{t}+\sum_{\tau=1}^{t-1} y_{\tau}^{t} d_{\tau} \geq p \sum_{\tau=1}^{t} \left( -q_0^{\tau}-\sum_{\sigma=1}^{\tau -1} q_{\sigma}^{\tau}d_{\sigma}+d_{\tau}\right), ~t=1,\ldots,T, \\
		& && u_0^{t}+\sum_{\tau=1}^{t-1} u_{\tau}^{t} d_{\tau} \geq q_0^t+\sum_{\tau=1}^{t-1}q_\tau^t d_{\tau}-w_t,~ u_0^{t}+\sum_{\tau=1}^{t-1} u_{\tau}^{t} d_{\tau} \geq 0, ~t=1,\ldots,T, \\
		& && v_0^{t}+\sum_{\tau=1}^{t-1} v_{\tau}^{t} d_{\tau} \geq w_t-q_0^t-\sum_{\tau=1}^{t-1}q_\tau^t d_{\tau},~v_0^{t}+\sum_{\tau=1}^{t-1} v_{\tau}^{t} d_{\tau} \geq 0, ~t=1,\ldots,T,  \\
		& &&\sum_{t=1}^{T} \left( q_0^t+\sum_{\tau=1}^{t-1}q^t_{\tau}d_\tau\right)\geq L,~q_0^t+\sum_{\tau=1}^{t-1}q^t_{\tau}d_\tau \geq 0, ~t=1,\ldots,T.
		\end{aligned}
		\end{footnotesize}
		\end{equation}
		Depending on whether $d\in \mathcal{S}_B$ or $\mathcal{S}_E$, (\ref{eq:semi.inf.opt}) can be cast as either an LP or an SOCP by rewriting the constraints involving $d$. For example, if $\mathcal{S}_B$ is given by $\mathcal{S}_B=\{d \in \mathbb{R}^T~|~ |d_t-\bar{d}_t| \leq \rho_t, t=1,\ldots,T\},$ the constraint $\sum_{t=1}^{T} \left( q_0^t+\sum_{\tau=1}^{t-1}q^t_{\tau}d_\tau\right)\geq L,~ \forall d \in \mathcal{S}_B$
		is equivalent to 
		\begin{align} \label{eq:constraint.L}
		\sum_{t=1}^T q_t^0+\sum_{\tau=1}^T \left( \sum_{t=1}^T q_t^{\tau}\right)\bar{d}_{\tau}-\sum_{\tau=1}^T \gamma_{\tau} \rho_{\tau} \geq L,\quad -\gamma_{\tau} \leq \sum_{t=1}^T (-q_t^{\tau}) \leq \gamma_{\tau},~\tau=0,\ldots, T,
		\end{align}
		where $\{\gamma_\tau\}_{\tau}$ are new variables; again, see \cite[Section 3.3]{ben2005retailer} for more details. We consider now the following scenario: at time $t=0$, before the inventory starts, the retailer and the supplier negotiate the details of their contract, that is, they need to agree on values of $\alpha^{\pm}$, $\beta^\pm$, and $L$. To do so effectively, the retailer needs to know in real time the worst-case cost incurred for different choices of these values. Solving the LP/SOCP to do so is impossible as, e.g., running one such LP with $T=100$ takes $\approx 15$ minutes, whilst running the SOCP with $T=100$ takes more than $60$ minutes. A tool which produces an approximation of the minimum inventory cost as a function of $(\alpha^\pm, \beta^\pm, L)$ would thus be valuable, which is what shape-constrained regression enables.
		
		\begin{proposition}\label{prop:shape.inventory}
			Let $v(\alpha^+,\alpha^-,\beta^+,\beta^-,L)$ be the optimal value of the LP obtained by taking $d\in \mathcal{S}_B$ in (\ref{eq:semi.inf.opt}). We have that (i) $v$ is convex in $L$, (ii) $v$ is concave jointly in $(\beta^+,\beta^-)$,(iii) $v$ is nondecreasing in $\alpha^+,\alpha^-,\beta^+,\beta^-,$ and $L$.
			The same results hold for the optimal value of the SOCP obtained by taking  $d \in \mathcal{S}_E$ in \ref{eq:semi.inf.opt} (assuming strong duality).
		\end{proposition}
		
		Thus, when preparing for the negotiation, we generate $m$ training points by sampling uniformly at random $m$ tuples $(\alpha^+,\alpha^-,\beta^+,\beta^-,L)$ and computing the minimum inventory cost $v$ by solving (\ref{eq:semi.inf.opt}) with a fixed set $\mathcal{S}_B$. We then fit the degree-$d$ SOSE of $v$ (with $r=2$) to these points. During the negotiation, to obtain the minimum inventory cost for new values of $(\alpha^{\pm},\beta^{\pm},L)$, we simply evaluate the SOSE at these values. This task only takes milliseconds, in contrast to solving the LPs or SOCPs. In Table \ref{tab:ours.upr.aarc}, we give the relative accuracy of the optimal values predicted by the SOSE, both on average (train and test) and in the worst-case (test only), as compared to the true optimal values obtained via the LP. (Similar results hold for $\mathcal{S}_E$.) As can be seen, the accuracy loss of the predictions is 1-2\% on average. We also compare the SOSE against the \emph{Unconstrained} Polynomial Regressor (UPR). Against the UPR, the SOSE performs best for $m$ around 400 and values of $d$ around 4 or 6. This implies that when using the SOSE as compared to the UPR, we need not solve as many LPs offline to obtain good-quality predictions. Furthermore, the SOSE appears to be more robust across variations in the data, with consistent average performance over training and testing sets and best worst-case performance, particularly for higher $d$.

		\begin{table}[]
			\begin{center}
				\footnotesize
				\begin{tabular}{l@{\hspace{6pt}}l|c@{\hspace{6pt}}c@{\hspace{6pt}}c|c@{\hspace{6pt}}c@{\hspace{6pt}}c}
					\toprule
					& & \multicolumn{3}{c|}{\textbf{SOSE}} & \multicolumn{3}{c}{\textbf{UPR}}\\
					\midrule
					&   &         Ave (\%)  &          Ave (\%)  &             Max (\%) &        Ave (\%)  &            Ave (\%)&             Max (\%) \\
					m &   d &         Train &           Test      &    Test  &     Train &      Test         &    Test  \\
					\midrule
					100  & 2 &         2.522 &	2.487 &    \textbf{6.329} &           1.545 & 	1.899  &         6.723 \\
					& 4 &          1.018 & 1.027 &    \textbf{3.506 }&            0.000 &	4.366 &             31.656 \\
					& 6 &         0.853 &	0.867 &    \textbf{3.102} &            0.000 & 	3.781 &             24.169 \\
					\hline
					200  & 2 &         2.530 & 	2.451 &             6.060 &           1.662 &	1.885 &      \textbf{4.400} \\ 
					& 4 &          1.111 &	1.065 &     \textbf{3.191} &           0.227 &	0.645 &              4.636 \\
					& 6 &        0.939 &	0.907 &     \textbf{1.849} &            0.000 &          3.174 &             19.809 \\
					\hline
					400  & 2 &         2.550 &	2.490 &             5.892 &           1.698 &	1.803 &        \textbf{4.325} \\
					& 4 &          1.065 & 1.026 &   \textbf{2.731} &           0.304 &	0.477 &              3.706 \\
					& 6 &         0.898 &	0.870 &   \textbf{1.889 }&            0.000 &        9.34 &            225.441 \\
					\hline
					600  & 2 &       2.577	& 2.478 &             5.849 &           1.788	& 1.752 &      \textbf{4.209} \\
					& 4 &          1.074 & 1.023 &             2.314 &           0.322 &	0.412 &      \textbf{1.618} \\
					& 6 &         0.904 &	0.869 &   \textbf{1.652} &            0.109 &	0.615 &              9.521 \\
					\bottomrule
				\end{tabular}
			\end{center}
			\caption{Relative accuracy of the values predicted by the SOSE (with $r=2$) and the UPR against those obtained by the LP for the contract negotiation application given in (\ref{subsec:opt.val}). Best test max deviation marked in bold font (smaller is better).}
			\label{tab:ours.upr.aarc}
		\end{table}
		
		\subsection{Shape-Constrained Optimal Transport Maps and Color Transfer}\label{subsec:opt.transport}
		
		An \emph{optimal transport map} is a function that maps one probability measure to another while incurring minimum cost. In many applications, it is of interest to determine an optimal transport map given two measures and a cost function \cite{peyre2019computational}. Interestingly, the problem of computing an optimal transport map can be related back to shape-constrained regression, as optimal transport maps are known to have specific shapes when the cost function under consideration or the measures they are defined over have certain properties. For example, if the cost function is the $l_2$-norm and one of the measures is continuous with respect to the Lebesgue measure, the Brenier theorem states that the optimal transport map is uniquely defined as the gradient of a convex function~\cite{brenier1991polar}. Following \cite{paty2020regularity}, rather than observing these properties of the map a posteriori, we use these shape constraints as regularizers when computing the optimal transport maps. This gives rise to shape-constrained regression problems. To solve these, \cite{paty2020regularity} propose an approach that can be viewed as a CLSE-based approach. We propose to use instead the SOSE, which we show is particularly well-suited to this application. 
		
		To better illustrate our method, we focus on the concrete application of \emph{color transfer}, though our methodology is applicable more widely to, e.g., the other applications mentioned in \cite{paty2020regularity} and voice transfer. The color transfer problem is defined by two images, the \emph{input image} and the \emph{target image}  \cite{rabin2014adaptive}. The goal is to transfer the colors of the target image to the input image, thereby creating a third image, the \emph{output image}; see Figure \ref{fig:illustration.color.transfer}. We now describe how the color transfer problem can be reformulated as a sequence of shape-constrained regression problems, following \cite{paty2020regularity}. Each pixel in the input and target images is associated to an RGB color triple in $[0,1]^3$, once each entry is normalized by 256. We define a discrete measure 
		$\mu= \sum_{i=1}^{\tilde{N}}a_i \delta_{x_i} ~\left( \text{resp. } \nu= \sum_{j=1}^{\tilde{M}}b_j \delta_{y_j}\right)$ over the space of colors in the input (resp. target) image by taking $x_i \in [0,1]^3, i=1,\ldots \tilde{N},$ (resp. $y_j \in [0,1]^3, j=1,\ldots,\tilde{M},$) to be the distinct color triples in the input (resp. target) image, with $\tilde{N}$ (resp. $\tilde{M}$) being less than or equal to the number of pixels in the input (resp. target) image, and $a_i$ (resp. $b_j$) to be the ratio of number of pixels of color $x_i$ (resp. $y_j$) to the total number of pixels. The idea is then to search for a function $f^*:[0,1]^3 \rightarrow \mathbb{R}$ that minimizes the 2-Wasserstein distance between the push-forward of $\mu$ under $\nabla f^*$ and $\nu$, under certain shape constraints. This is written as the following optimization problem:
		\begin{equation}\label{eq:problem.opt.transport}
		\begin{aligned}
		&\inf_{f,P \in \mathbb{R}^{\tilde{N} \times \tilde{M}}} &&\sum_{i,j} P_{ij}||\nabla f(x_i)-y_j||^2\\
		&\text{s.t. } && P\geq 0,~P1_{\tilde{M}}=a,~P^T1_{\tilde{N}}=b\\
		& &&\nabla f \text{ is $L$-lipschitz and } f \text{ is $\ell$-strongly convex over $[0,1]^3$},
		\end{aligned}
		\end{equation}
		where $L$ and $\ell$ are parameters of the problem. (As a reminder, a function $g:\mathbb{R}^m \rightarrow \mathbb{R}^n$ is $L$-Lipschitz over a box $B$ if $||g(x)-g(y)|| \leq L\cdot ||x-y||$, for all $x,y \in B$, where $||\cdot||$ is some norm. Similarly, a function $f:\mathbb{R}^n \mapsto \mathbb{R}$ is $\ell$-strongly convex over $B$ if $H_f(x) \succeq \ell I$, for all $x \in B$.) We derive from the optimal solution $f^*$ to (\ref{eq:problem.opt.transport}), the optimal transport map (or color transfer map) $\nabla f^*: [0,1]^3 \rightarrow [0,1]^3.$ To obtain the output image, we simply apply $\nabla f^*$ to the RGB triple of each pixel in the input image to obtain a new RGB triple (i.e., the new color of the pixel) for that pixel. In this context, smaller $L$ gives rise to more uniform colors whereas larger $\ell$ increases the contrast; see Figure \ref{fig:comparison.color.transfer}.
		
		In its current form however, problem (\ref{eq:problem.opt.transport}) is not quite a shape-constrained regression problem of the type (\ref{eq:opt.tg}) or (\ref{eq:opt.th}). This is due to the matrix variable $P$ which makes the problem non-convex. To circumvent this issue, we use alternate minimization: we fix $f$ and solve for $P$ using, e.g., Sinkhorn's algorithm (see \cite{peyre2019computational}). We then fix $P$ and solve for $f$. If we parametrize $f$ as a polynomial (with $P$ fixed), we obtain a shape-constrained polynomial regression problem:
		\begin{equation} \label{eq:shape.constr.regression.opt.transp}
		\begin{aligned}
		&\inf_{f \in P_{3,d}} &&\sum_{i,j} P_{ij}||\nabla f(x_i)-y_j||^2\\
		&\text{s.t. } && \ell \cdot I \preceq H_f(x) \leq L \cdot I,~\forall x\in [0,1]^3,
		\end{aligned}
		\end{equation}
		which we solve using the sos techniques from Section \ref{sec:sose}. We iterate this process until convergence.
		
		An example of the output images obtained via this process is given in Figure \ref{fig:illustration.color.transfer}. Additional illustrations can be found in Figure \ref{fig:comparison.color.transfer} for different values of $l$ and $L$ with $d=4$ and $r=3$. The color transfer application works particularly well for the SOSE as the number of features is small (equal to 3), the number of data points is very large, as it corresponds to the number of pixels in the images, and as a large number of new predictions need to be made (one per pixel of the input image). In contrast, the CLSE approach considered in \cite{paty2020regularity} requires the authors to segment the images via k-means clustering to limit computation time. Pre-processing of this type can lead to undesirable artifacts in the output image and grainy texture, which our method avoids. 
		
		\begin{figure}[H]
			\centering
			\begin{subfigure}{.2\textwidth}
				\centering
				\includegraphics[width=.75\linewidth]{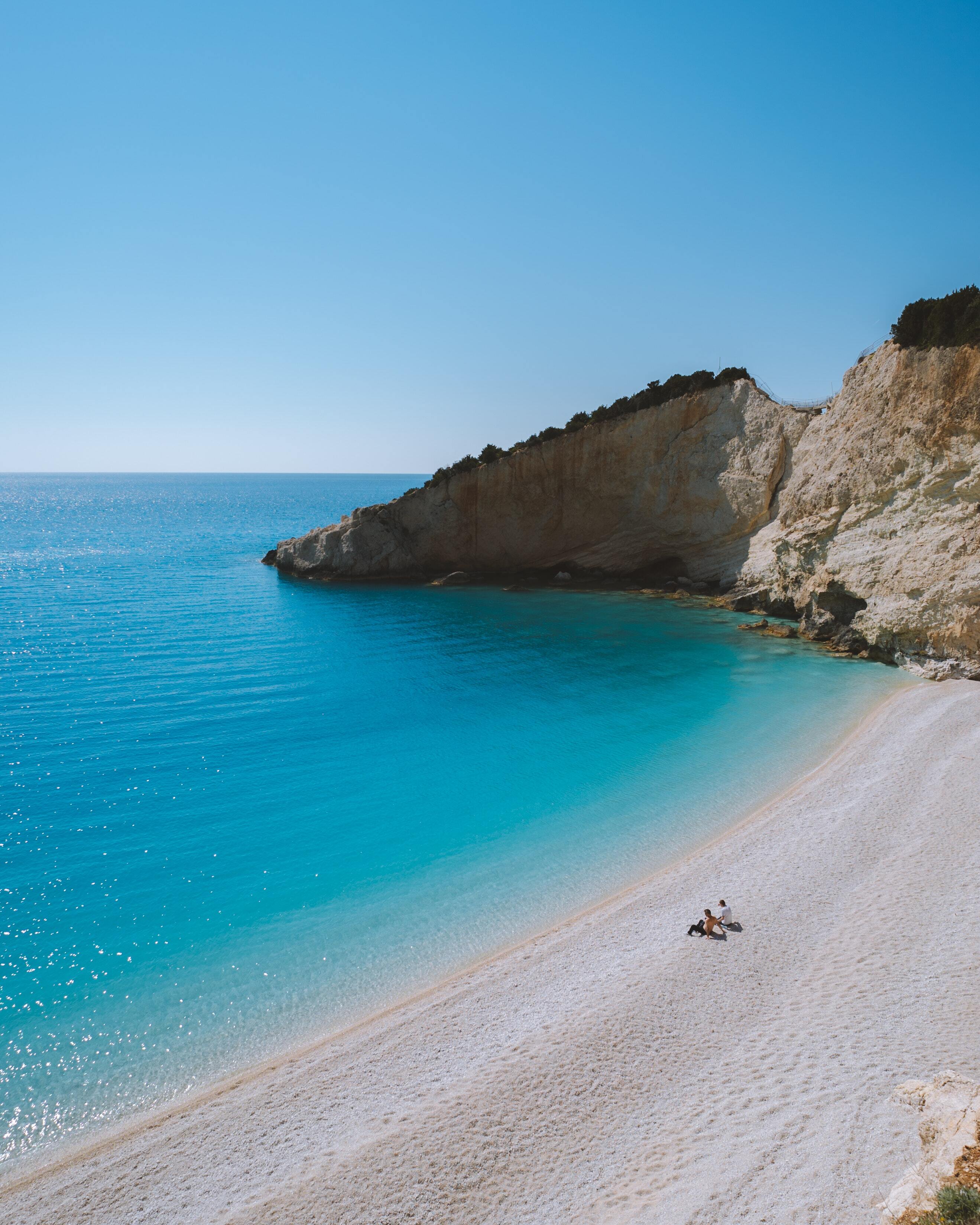}  
				\caption{\footnotesize Input image}
			\end{subfigure}
			\begin{subfigure}{.2\textwidth}
				\centering
				\includegraphics[width=1\linewidth]{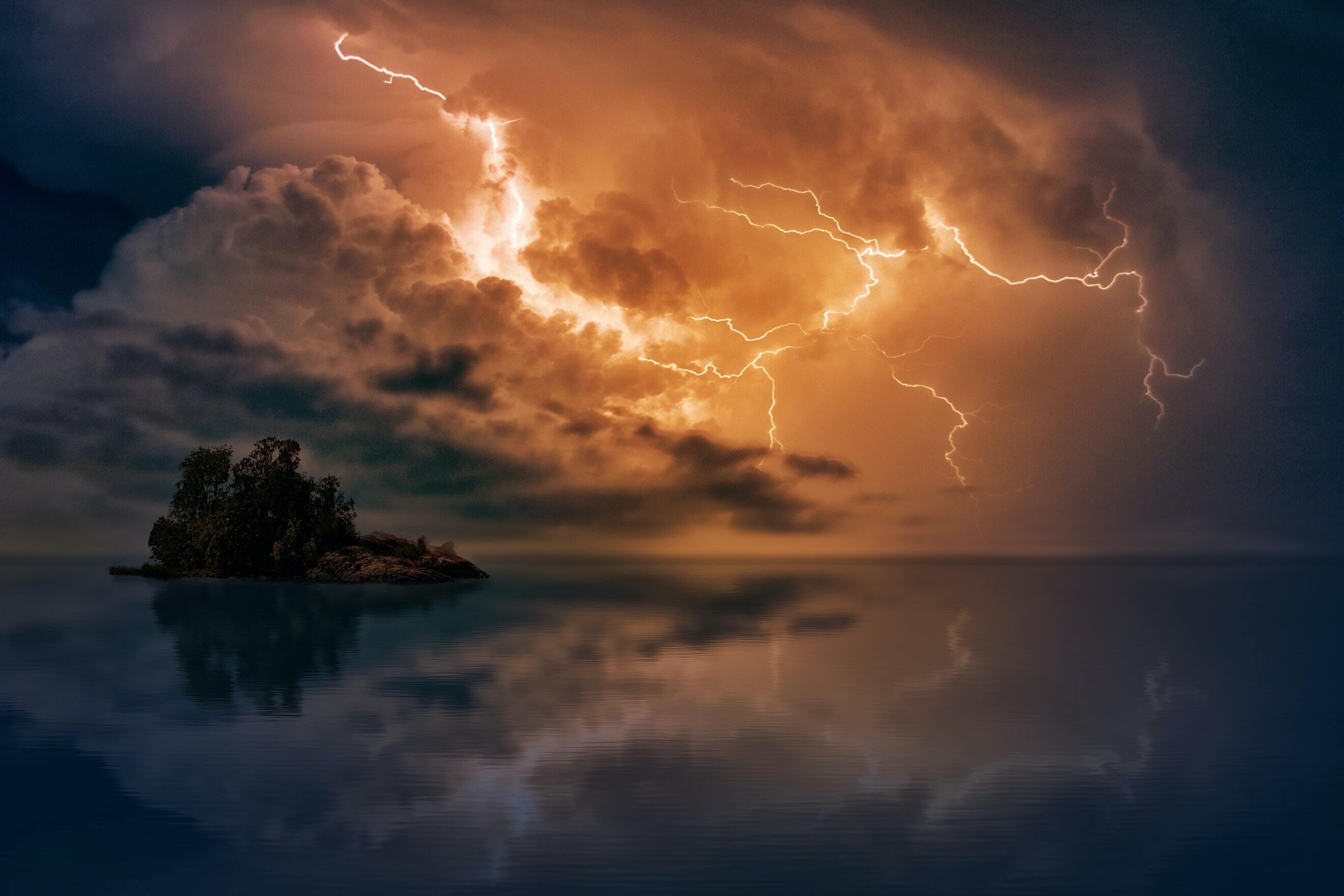}  
				\caption{\footnotesize Target image}
			\end{subfigure}
			\begin{subfigure}{.2\textwidth}
				\centering
				\includegraphics[width=0.85\linewidth]{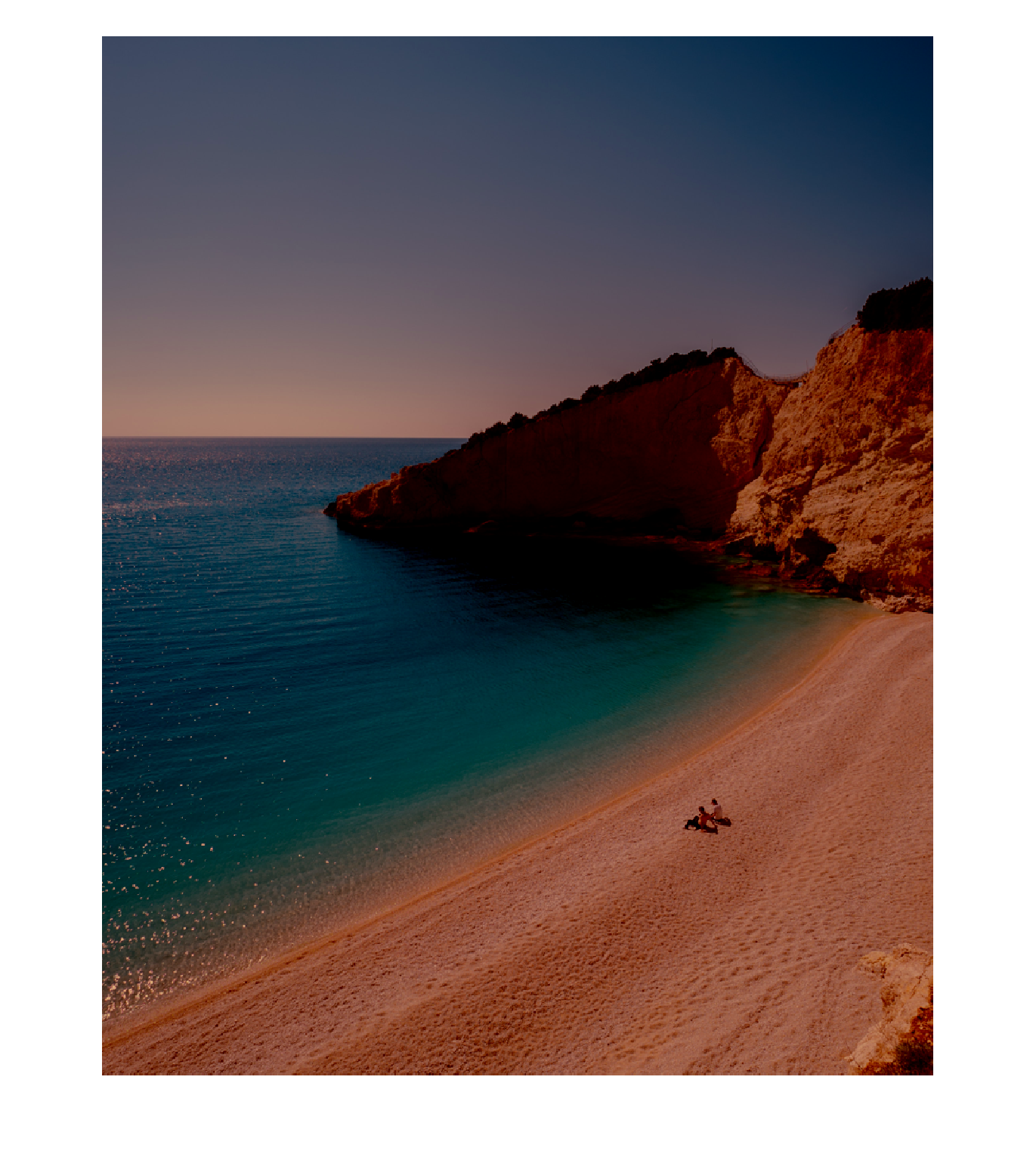}  
				\caption{\footnotesize Output image}
			\end{subfigure}
			\caption{The color transfer problem in (\ref{subsec:opt.transport}) with $\ell=1$ and $L=10$.}
			\label{fig:illustration.color.transfer}
		\end{figure}

	\newpage
	
	%
	
	


	\appendix
	

			\section{Experimental Setup and Additional Figures for Section \ref{sec:sose}} \label{appendix:section2}
			
			\subsection{Experimental Setup}\label{appendix:specs}
			\paragraph{Machine Specifications.} All experiments were performed on a MacBook Pro (2.6 GHz 6-Core Intel Core i7 processor and 16 GB RAM).
			
			\paragraph{Dataset generation for synthetic experiements.} We generate $m$ data points $X_i$ independently and uniformly from the $n$-dimensional box $[0,1]^n$. We define two separate functions:
			$$f_1(w_1,\ldots,w_n)=(w_1+\ldots+w_n) \log(w_1+\ldots+w_n) \text{ and } f_2(w_1,\ldots,w_n)=e^{\sqrt{w_1^2+\ldots+w_n^2}},$$
			which are both convex but not monotonous in any direction. To generate $Y_i$, we let $Y_i=f(X_i)+\nu_i$ where $f$ is either $f_1$ or $f_2$ and the $\nu_i's$ are sampled independently from a standard normal. We refer to the square root of the optimal value of the optimization problem (\ref{eq:opt.bg}) as the train root mean squared error (or train RMSE). The test RMSE is obtained by generating an additional random 1000 points $X_i$ and evaluating the predicted values against the true values $f(X_i)$ without adding noise. The optimization solver times are given in seconds. When fitting these functions, we use (\ref{eq:opt.tg}) and produce $\tg$ in light of the shape constraints.
			
			\paragraph{Obtaining the SOSEs.} To solve the SDPs (\ref{def:tg}), we used MATLAB R2020b, MOSEK~\cite{MOSEK}, and YALMIP (release R20200930)~\cite{yalmip}. If we further wished to speed up the solving times of our SDPs, we could use recent approaches to SDP \cite{o2016conic,li2018qsdpnal,papp2019sum,papp2013semidefinite,majumdar2020recent} which are more scalable. We do not do so here as our results are already competitive.
			
			\paragraph{Obtaining the CLSE.}
			
			We use the implementation proposed by \cite{chen2020multivariate}. The original implementation can be found at \url{https://github.com/wenyuC94/ConvexRegression/}. It uses Julia 1.5 programming language with the default parameters for their best performing algorithm using Augmentation Rule 2 specified in the demo code: \url{https://github.com/wenyuC94/ConvexRegression/blob/master/demo_2stage.jl}. 
			To improve the performance of CLSE on our task we increased the number of iterations to 1000. Note that the CLSE thus computed is not exactly the CLSE defined in Section \ref{sec:sose}---in particular, for the version used in \cite{chen2020multivariate}, evaluation can be done by taking the maximum of affine pieces evaluated at the new point. We have made the choice to use this version to show that, even with computational improvements, there are regimes where the SOSEs are faster to compute than the CLSE.
			
			\paragraph{Obtaining the Max-Affine Estimator.} We used the algorithm (Algorithm 1) described in \cite{ghosh2019max}. We contacted the author for an original implementation; in the absence of a response we implemented the algorithm outlined in the paper  pseudocode. The algorithm was implemented in Python 3.8.
			Our code can be found at \url{https://github.com/mcurmei627/dantzig/SOSE_CLSE_MA/MaxAffine}. Algorithm 1 however necessitates initial parameter estimates which can be obtained via two different initialization processes. We refer to \emph{Spect MAE} when we proceed with the initialization given in Algorithms 2 and 3. We refer to \emph{Rand MAE} when we proceed with a random initialization. Both initializations require specifying a number $M$ of seeds. To obtain this number, we tried $M=2^j$ for $j=1,\ldots,8$ and picked the best one based on cross-validation. We also need to specify $k$, the number of affine pieces. The columns marked \emph{Spect Opt} and \emph{Rand Opt} pick $k$ via cross-validation as well. The column marked \emph{Rand} takes $k=\binom{n+2}{2}$. The idea here is to let the amount of parametrization be of the same order as that of the SOSE, which is $\binom{n+d}{d}$.

			\subsection{Additional Experiments for Section \ref{subsec:sose.computation} and \ref{subsec:sose.comparison}} \label{subsec:add.exp}
			
			Figure \ref{fig:train_test_rmse_2} (resp. Table \ref{tab:gen.error_2}) is an analog to Figure \ref{fig:train_test_rmse} (resp. Table \ref{tab:gen.error}) for $f=f_2$. 
			\begin{figure}[]
				\begin{subfigure}{\textwidth}
					\centering
					\includegraphics[width=0.8\linewidth]{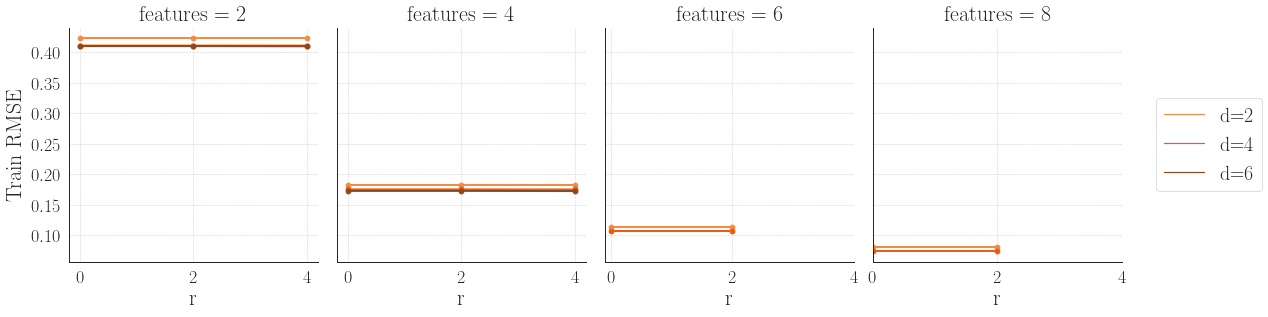}  
					\caption{\footnotesize Train RMSE.}
					\label{fig:train_rmse_2}
				\end{subfigure}
				\begin{subfigure}{\textwidth}
					\centering
					\includegraphics[width=0.8\linewidth]{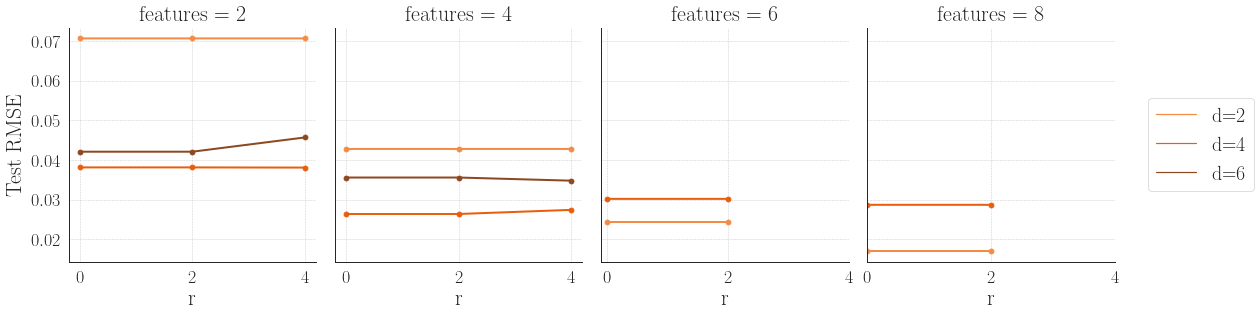}  
					\caption{\footnotesize Test RMSE.}
					\label{fig:test_rmse+2}
				\end{subfigure}
				\caption{Train and Test RMSE for (\ref{eq:opt.tg}) with $m=10,000$ data points generated as described in Appendix \ref{appendix:specs}, as the number $n$ of features, the degree $d$ of the polynomials, and the degree $r$ of the semidefinite programming hierarchy. Lighter color corresponds to lower $d$.}
				\label{fig:train_test_rmse_2}
			\end{figure}
			
			\begin{table}[!ht]
				\footnotesize
				\centering
				\begin{tabular}{|c|c|c|c|c|>{\columncolor[gray]{0.9}}c|>{\columncolor[gray]{0.9}}c|>{\columncolor[gray]{0.9}}c|c|>{\columncolor[gray]{0.9}}c|c|c|c|>{\columncolor[gray]{0.9}}c|>{\columncolor[gray]{0.9}}c|>{\columncolor[gray]{0.9}}c|}
					\cline{3-16}
					\multicolumn{2}{c|}{~} & \multicolumn{6}{c|}{\textbf{MAE}} & \multicolumn{2}{c|}{\textbf{CLSE}} & \multicolumn{6}{c|}{\textbf{SOSE}}\\ \hline
					\multicolumn{2}{|c|}{RMSE} & \multicolumn{3}{c|}{Train} &   \multicolumn{3}{c|}{\cellcolor[gray]{0.9} Test} &  Train & Test  & \multicolumn{3}{c|}{Train} & \multicolumn{3}{c|}{\cellcolor[gray]{0.9} Test} \\ \hline
					\multirow{2}{*}{$m$} & \multirow{2}{*}{$n$}  & Spect & Rand & ~ & Spect & Rand & ~ & ~ & ~ & \multirow{2}{*}{$d=2$}& \multirow{2}{*}{$d=4$}& \multirow{2}{*}{$d=6$} & ~ & ~ & ~ \\
					~ & ~ & Opt & Opt & \multirow{-2}{*}{Rand}  & Opt & Opt & \multirow{-2}{*}{Rand} & ~ & ~ & ~ & ~ & ~ &\multirow{-2}{*}{$d=2$} & \multirow{-2}{*}{$d=4$} &\multirow{-2}{*}{$d=6$} \\ \hline
					\multirow{5}{*}{2000}& 2 & 0.473 & 0.414 & 0.418 & 0.223 & 0.078 & 0.109 & 0.410 & 0.056 & 0.422 & 0.411 & 0.410 & 0.071 & \textbf{0.038} & 0.042 \\
					~ & 3 & 0.294 & 0.254 & 0.272 & 0.146 & 0.070 & 0.108 & 0.244 & 0.049 & 0.259 & 0.254 & 0.251 & 0.059 & \textbf{0.036} & 0.045 \\ 
					~ & 4 & 0.230 & 0.183 & 0.192 & 0.150 & 0.075 & 0.096 & 0.153 & 0.057 & 0.182 & 0.175 & 0.173 & 0.043 & \textbf{0.026} & 0.036 \\ 
					~ & 5 & 0.170 & 0.147 & 0.162 & 0.113 & 0.079 & 0.096 & 0.111 & 0.060 & 0.141 & 0.135 & 0.132 & \textbf{0.031} & 0.033 & 0.039 \\
					~ & 6 & 0.145 & 0.110 & 0.128 & 0.105 & 0.069 & 0.082 & 0.147 & 0.112 & 0.113 & 0.108 & 0.099 & \textbf{0.024} & 0.030 & 0.042 \\ \hline
					\multirow{5}{*}{5000} & 2 & 0.467 & 0.410 & 0.417 & 0.221 & 0.059 & 0.090 & 0.412 & 0.065 & 0.416 & 0.408 & 0.408 & 0.070 & \textbf{0.028} & 0.028 \\
					~ & 3 & 0.291 & 0.256 & 0.264 & 0.148 & 0.063 & 0.083 & 0.264 & 0.099 & 0.257 & 0.253 & 0.251 & 0.057 & \textbf{0.022} & 0.030 \\
					~ & 4 & 0.213 & 0.183 & 0.196 & 0.126 & 0.067 & 0.091 & 0.194 & 0.098 & 0.182 & 0.177 & 0.176 & 0.041 & \textbf{0.019} & 0.024 \\
					~ & 5 & 0.182 & 0.145 & 0.162 & 0.120 & 0.064 & 0.092 & 0.159 & 0.097 & 0.139 & 0.136 & 0.134 & 0.030 & \textbf{0.023} & 0.028 \\ 
					~ & 6 & 0.146 & 0.116 & 0.128 & 0.099 & 0.064 & 0.076 & 0.136 & 0.091 & 0.113 & 0.109 & 0.105 & 0.022 & \textbf{0.021} & 0.031 \\ \hline
					\multirow{5}{*}{10000} & 2 & 0.466 & 0.413 & 0.418 & 0.220 & 0.067 & 0.090 & 0.413 & 0.063 & 0.416 & 0.409 & 0.409 & 0.067 & 0.023 & \textbf{0.022} \\ 
					~ & 3 & 0.291 & 0.256 & 0.264 & 0.148 & 0.058 & 0.081 & 0.263 & 0.084 & 0.257 & 0.252 & 0.252 & 0.055 & \textbf{0.019} & 0.020 \\ 
					~ & 4 & 0.235 & 0.183 & 0.199 & 0.148 & 0.057 & 0.095 & 0.212 & 0.122 & 0.181 & 0.176 & 0.176 & 0.040 & \textbf{0.014} & 0.016 \\ 
					~ & 5 & 0.182 & 0.145 & 0.164 & 0.118 & 0.059 & 0.092 & 0.166 & 0.105 & 0.139 & 0.136 & 0.135 & 0.029 & \textbf{0.016} & 0.020 \\ 
					~ & 6 & 0.148 & 0.118 & 0.128 & 0.097 & 0.058 & 0.070 & 0.148 & 0.105 & 0.112 & 0.110 & 0.106 & 0.021 & \textbf{0.015} & 0.027 \\ \hline
				\end{tabular}
				\caption{Comparison of the train and test RMSEs, for different values of $m$ and $n$, of the SOSE $\tg$ computed in (\ref{eq:opt.tg}) (with $r=1$), for the CLSE, and for the MAE under different initializations and choice of hyperparameters; see Appendix \ref{appendix:specs}. Best test RMSE marked in bold font.}
				\label{tab:gen.error_2}
			\end{table}

		\section{Proofs of Results in Section \ref{sec:consistency} }

			\subsection{Proof of Lemma \ref{lem:laurent_adjacent} } \label{appdx:laurent_adjacent}
			
			For this proof, we need to introduce some notation. Let $\mathbb{N}_d^n$ (resp. $\tilde{\mathbb{N}}_d^n$) be the set of vectors $\alpha$ in $\mathbb{N}^n$ such that $|\alpha| \leq d$ (resp. $=d$). Given a sequence $(w_{\alpha})_{\alpha \in \mathbb{N}^n}$, we define $M(w)$ to be the matrix indexed by $\mathbb{N}^n$ with $(\alpha,\beta)^{th}$ entry $w_{\alpha+\beta}$ for $\alpha,\beta \in \mathbb{N}^n$. Likewise, given a sequence $(w_{\alpha})_{\alpha \in \mathbb{N}^n_{2d}}$,  we define $M_{d}^n(w)$ to be the moment matrix of order $d$ indexed by $\mathbb{N}_d^n$, with $(\alpha,\beta)^{th}$ entry $w_{\alpha+\beta}$, for $\alpha,\beta \in \mathbb{N}_d^n$. Finally, given a sequence $(w_{\alpha, \alpha'})_{\alpha \in \mathbb{N}^n_{2d}, \alpha' \in \tilde{\mathbb{N}}^n_2}$, we define $\tilde{M}_d^n(w)$ to be the moment matrix of order $d$ indexed by $\mathbb{N}_d^n \times \tilde{\mathbb{N}^n_1}$ with $((\alpha, \alpha'),(\beta, \beta'))^{th}$ entry $w_{(\alpha+\beta, \alpha'+\beta')}$ for $\alpha,\beta \in \mathbb{N}^n_{d}, \alpha', \beta' \in \tilde{\mathbb{N}}^n_1$. 
			
			As we will mostly be working with $\tilde{M}_d^n(w)$ moving forward, we introduce the \emph{Riesz linear functional} which is a little more intelligible in this setting. Given a sequence $(w_{\alpha, \alpha'})_{\alpha \in \mathbb{N}^n_{2d}, \alpha' \in \tilde{\mathbb{N}}^n_2}$, the Riesz linear functional $L_w: \mathbb{R}_{2d}[x] \cdot \mathbb{R}_2[y] \rightarrow \mathbb{R}$ is defined by:
			$$\sum_{(\alpha, \alpha') \in \mathbb{N}_d^n \times \tilde{\mathbb{N}}_2^n} f_{\alpha, \alpha'} x^{\alpha} y^{\alpha'} \mapsto L_w(f)=\sum_{(\alpha, \alpha') \in \mathbb{N}_d^n \times \tilde{\mathbb{N}}_2^n} f_{\alpha, \alpha'} w_{\alpha,\alpha'}.$$ When there is no ambiguity, we drop the index $w$ from $L_w$. We use the following lemmas.
			\begin{lemma} \label{lemma:ineq}
				Let $(w_{\alpha, \alpha'})_{\alpha \in \mathbb{N}^n_{2d}, \alpha' \in \tilde{\mathbb{N}}^n_2} \in \mathbb{R}$ be a sequence. If $\tilde{M}_d^n(w) \succeq 0$, then:
				\begin{enumerate}
					\item $L(x^{2\alpha}y_j^2) \geq 0$, $\forall 0 \leq |\alpha| \leq d$, $\forall j=1,\ldots,n$.
					\item $L(x^{\alpha+\beta} y^{\alpha'+\beta'})^2 \leq L(x^{2\alpha}y^{2\alpha'}) \cdot L(x^{2\alpha'}y^{2\beta'})$, where $\alpha, \beta \in \mathbb{N}_d^n$ and $\alpha',\beta' \in \tilde{\mathbb{N}}_1^n.$ 
				\end{enumerate}
			\end{lemma}
			The proof follows immediately from the fact that $\tilde{M}_d^n(w) \succeq 0$ and the definition of $L$.
			
			\begin{lemma} \cite[Theorem 4.12]{laurent2009sums}
				Let $w \in \mathbb{R}^n$, let $C>0$ and let $K\mathrel{\mathop{:}}=[-C,C]^n$. Then $w$ has a representing measure supported by the set $K$ if and only if $M(y) \succeq 0$ and there is a constant $C_0>0$ such that $|w_{\alpha}| \leq C_0C^{|\alpha|}$ for all $\alpha \in \mathbb{N}^n$. 
			\end{lemma}
			
			We now give two propositions before proving Lemma \ref{lem:laurent_adjacent}.
			
			\begin{proposition} \label{prop:bounded}
				Let $(w_{\alpha, \alpha'})_{\alpha \in \mathbb{N}^n_{2d}, \alpha' \in \tilde{\mathbb{N}}^n_2} \in \mathbb{R}$ be a sequence.
				If $\tilde{M}_d^n(w) \succeq 0$, then $$L(x^{\alpha}y^{\alpha'}) \leq \lambda_d \mathrel{\mathop{:}}=\max_{i=1,\ldots,n, j=1,\ldots,n} \{L(x_i^{2d}y_j^2), L(y_j^2)\}, \text{ for any } 0 \leq |\alpha| \leq 2d \text{ and for } |\alpha'|=2.$$
			\end{proposition}
			
		\begin{proof}
			First, we prove the claim: $L(x^{2\alpha}y_j^2) \leq \lambda_d$ for any $0 \leq |\alpha| \leq d$, $i=1,\ldots,n$ and $j=1,\ldots,n$.
			We proceed by induction on $d$.  If $d=0,1$, the result is obvious. Assume $d \geq 1$ and the result holds for $d-1$, i.e., $L(x^{2\alpha}y_j^2) \leq \lambda_{d-1}$ for any $0 \leq |\alpha|\leq d-1$ and $j=1,\ldots,n$. We show that $L(x^{2 \alpha}y_j^2) \leq \lambda_{d}$ for any $0 \leq |\alpha| \leq d$ and $j=1,\ldots,n$. Let $j \in \{1,\ldots,n\}$ and $\alpha$ such that $|\alpha|=d-1$.  It is enough to show that $L(x^{2\alpha}y_j^2) \leq \lambda_d$, as the case where $|\alpha| \leq d-2$ will then follow from the induction hypothesis. If there exists $k \in \{1,\ldots,n\}$ such that $L(x^{2\alpha}y_j^2) \leq L(y_k^2)$, then the result holds. If not, then for all $k \in \{1,\ldots,n\}$, $L(x^{2\alpha}y_j^2) \geq L(y_k^2)$. 
			Construct vectors $\gamma, \gamma'$ such that $\gamma+\gamma'=2\alpha$, $|\gamma|=d-2$ and $|\gamma'|=d$. We then have from Lemma \ref{lemma:ineq}, combined to the induction hypothesis and the fact that $L(x^{2\alpha}y_j^2) \geq L(y_k^2)$ for all $k$,
			$L(x^{2\alpha}y_j^2)^2 \leq L(x^{2\gamma} y_j^2) \cdot L(x^{2\gamma'}y_j^2) \leq L(x^{2\alpha}y_j^2) \cdot L(x^{2\gamma'}y_j^2).$
			Thus, $L(x^{2\alpha}y_j^2) \leq L(x^{2\gamma'}y_j^2)$. As $|2\gamma'|=2d$, the claim follows.
			Consider now $0<|\alpha| \leq 2d$ and $|\alpha'|=2$. Write $(\alpha,\alpha')=(\beta,\beta')+(\gamma,\gamma')$ where $|\beta|, |\gamma| \leq d$ and $|\beta'|=|\gamma'| =1$. Using Lemma \ref{lemma:ineq}, we have that 
			$L(x^{\alpha}y^{\alpha'})^2 \leq L(x^{2\beta}y^{2\beta'}) \cdot L(x^{2\gamma}y^{2\gamma'}) \leq \lambda_d^2,$
			which concludes the proof. 
			\end{proof}
			
			\begin{proposition} \label{prop:laurent}
				As in Lemma \ref{lem:laurent_adjacent}, let $f(x,y)=y^TM(x)y$ and let $\Psi_{d'}(x,y)=\sum_{i=1}^n y_i^2(1+\sum_{j=1}^n x_j^{2d'})$. Consider the program:
				\begin{align}\label{eq:opt}
				\epsilon_t^* \mathrel{\mathop{:}}= \inf_{\{w_{\alpha,\alpha'} \}\in \mathbb{R}^{\mathbb{N}^n_{2t-2} \times \tilde{\mathbb{N}}_2^n}} f^Tw \text{ s.t. } \tilde{M}_{t-1}(w)\succeq 0, w^T \Psi_{t} \leq 1,
				\end{align}
				where we use $f$ and $\Psi_t$ for the coefficients of the corresponding polynomials. The following hold:
				\begin{enumerate}[(i)]
					\item $- \infty <\epsilon_t^* \leq 0$ and the infimum is attained in (\ref{eq:opt}).
					\item For $\epsilon \geq 0$, the polynomial $f+\epsilon\Psi_t$ is a sum of squares if and only if $\epsilon \geq -\epsilon_t^*$. In particular, $f$ is a sum of squares if and only if $\epsilon_t^*=0$. 
					\item If $f$ is nonnegative on $[-1,1]^n \times \mathbb{R}^n$, then $\lim_{t \rightarrow \infty} \epsilon_t^*=0$.
				\end{enumerate}
			\end{proposition}
			
			\begin{proof}
			The proof is identical to that of \cite[Proposition 7.9]{laurent2009sums}, noting that $$1 \geq w^T \Psi_t= L(\Psi_t(x,y)) \geq \sum_{i=1}^n L(y_i^2)+\sum_{i=1}^n \sum_{j=1}^{n} L(y_i^2 x_j^{2t}) \geq 0$$
			implies that $\lambda_t \leq 1$ (and hence the conclusion of Proposition \ref{prop:bounded}) and that $f(x,y)$ being nonnegative over $[-1,1]^n \times \mathbb{R}^n$ implies nonnegativity over $[-1,1]^{2n}$.
			\end{proof}
			
			\begin{proof}[Proof of Lemma \ref{lem:laurent_adjacent}.]
			The result follows from Proposition \ref{prop:laurent}, (ii) and (iii). Let $\epsilon>0$. As $\lim_{d \rightarrow \infty} \epsilon_d^*=0$, $\exists d_0' \in \mathbb{N}$ such that $\epsilon_{d'}^* \geq -\epsilon$ for all $d' \geq d_0'$. Thus, $f+ \epsilon \Psi_{d'}$ is sos for all $d' \geq d_0'$.
			\end{proof}

		\subsection{Proof of Proposition \ref{prop:weierstrass}} \label{appendix:weierstrass}
		The proof of Proposition \ref{prop:weierstrass} requires the following lemma.
		\begin{lemma}[e.g., Theorem 6.7 in \cite{fellhauer2016approximation}]\label{lemma:Bernstein}	
			Consider the Bernstein multivariate polynomial of degree $d$ and in $n$ variables, defined over $[0,1]^n$:
			\begin{footnotesize}
				\begin{align*}
				B_d(f,x)=\sum_{j_1+\ldots+j_n=d} f\left(\frac{j_1}{d},\ldots,\frac{j_n}{d}\right) C_d^{j_1}\ldots C_d^{j_n}x_1^{j_1}(1-x_1)^{d-j_1}\ldots x_n^{j_n}(1-x_n)^{d-j_n},
				\end{align*}
			\end{footnotesize}
			where $C_d^{j_i}=\frac{d!}{j_i!(d-j_i)!}.$ Let $m$ be an integer and assume that $f$ is $m$ times continuously differentiable. Let $k=(k_1,\ldots,k_n)$ be a multi-index such that $\sum_{i=1}^n |k_i|\leq m$ and denote by $\partial^k f=\frac{\partial^k f(x)}{\partial x_1^{k_1}\ldots \partial x^{k_n}}.$	
			Then, $\forall k$ such that $\sum_{i=1}^{n}|k_i|\leq m$, $\lim_{d\rightarrow \infty}\sup_{x\in [0,1]^n}|\partial^k B_d(f,x)-\partial^k f(x)| =0$.
		\end{lemma}
		
		This result can easily be extended to hold over any box $B \subset \mathbb{R}^n$ by simply scaling and translating the variables in the Bernstein polynomials. 
		
		\begin{proof}[Proof of Proposition \ref{prop:weierstrass}.] We start with part (i).  Let $\epsilon>0$ and $M \mathrel{\mathop{:}}=\max_{x \in B} \frac12 \sum_{i=1}^n x_i^2$. From Lemma \ref{lemma:Bernstein}, as $f$ is twice continuously differentiable, there exists a polynomial $q$ of degree $d$ s.t.
		\begin{align}\label{eq:second.deriv}
		\sup_{x \in B} |f(x)-q(x)| \leq \frac{\epsilon}{2(1+2nM)} \text{ and } \sup_{x \in B} \left|\frac{\partial^2 f(x)}{\partial x_i \partial x_j}- \frac{\partial^2 q(x)}{\partial x_i \partial x_j}\right| \leq \frac{\epsilon}{2(1+2nM)}, \forall i,j.
		\end{align}
		Let $\Delta H(x)=H_q(x)-H_f(x)$. As $f$ and $q$ are twice continuously differentiable, the entries of $\Delta H(x)$ are continuous in $x$. This implies that $x \mapsto \lambda_{\min}(\Delta H(x))$ is continuous \cite[Corollary VI.1.6]{bhatia2013matrix}. Hence, if we let  $\Lambda\mathrel{\mathop{:}}=\min_{x\in B} \lambda_{\min}(\Delta H(x)),$ it follows that there exists $x_0 \in B$ such that $\Lambda=\lambda_{\min} \Delta H(x_0).$ We now bound this quantity. Recall that for a symmetric $n \times n$ real-valued matrix $M$, $||M||_{\max}$ is the max-norm of $M$, i.e., its largest entry in absolute value, $||M||_2=\max \{|\lambda_{\min}(M)|,|\lambda_{\max}(M)|\},$ and $||M||_2 \leq n ||M||_{\max}$. From (\ref{eq:second.deriv}), we have that 
		$||\Delta H(x_0)||_{\max} \leq \frac{\epsilon}{2(1+2nM)},$
		which implies that   $\max \{|\lambda_{\min}(\Delta H(x_0))|, |\lambda_{\max}(\Delta H(x_0))| \} \leq \frac{n\epsilon}{2(1+2nM)}, $
		and so
		$$-\frac{n\epsilon}{2(1+2nM)}\leq \Lambda \leq \frac{n\epsilon}{2(1+2nM)}.$$ By definition of $\Lambda$, we thus have $\Delta H(x)\succeq -\frac{n\epsilon}{2(1+2nM)}$ for all $x \in B.$ Now, consider $p(x)\mathrel{\mathop{:}}=q(x)+\frac{n\epsilon}{2(1+2nM)}x^Tx.$
		For any $x \in B$, we have
		$|f(x)-p(x)|\leq |f(x)-q(x)|+|q(x)-p(x)| \leq \frac{\epsilon}{2(1+2nM)}+\frac{n\epsilon}{2(1+2nM)}\cdot 2M<\epsilon.$
		Using our previous result on $\Delta H(x)$, the definition of $p$, and the fact that $H_f(x)\succeq 0$,
		$H_p(x)=H_p(x)-H_q(x)+H_q(x)-H_f(x)+H_f(x) \succeq \frac{2n\epsilon}{2(1+2nM)}I-\frac{n\epsilon}{2(1+2nM)}I \succ 0.$
		From this, it follows that there exists a degree $d$ and a polynomial $p \in C_{n,d}$ such that $\sup_{x \in B} |f(x)-p(x)|<\epsilon.$ The definition of $g_d$ as the minimizer of $\sup_{x \in B} |f(x)-g(x)|$ for any $g \in C_{n,d}$ gives us the result.
		
		We now show part (ii). Let $\epsilon>0$ and take $M=\max_{x \in B}||x||_{\infty}$. From Lemma \ref{lemma:Bernstein}, there exists a polynomial $q$ of degree $d$ such that $$\max_{x \in B} |f(x)-q(x)| \leq \epsilon' \text{ and } \max_{x \in B} \left| \frac{\partial f(x)}{\partial x_i}- \frac{\partial q(x)}{\partial x_i} \right| \leq  \epsilon',$$
		where $\epsilon'=\frac{\epsilon}{nM+1}>0$.
		Recall that $I^+ \cap I^- =\emptyset$ and define $\rho$ as in Theorem \ref{thm:mon}. Consider $p(x) \mathrel{\mathop{:}}=q(x)+\epsilon' \sum_{i=1}^n \rho_i x_i.$ Let $x \in B$. We have:
		$$|p(x)-f(x)| \leq |p(x)-q(x)|+|q(x)-f(x)| \leq \epsilon' \sum_{i=1}^n |x_i| + \epsilon' \leq \epsilon' \cdot (nM+1)=\epsilon.$$
		Now, let $i \in I^+$, i.e., $\rho_i=-1$. As $\frac{\partial f(x)}{\partial x_i} \leq K_i^+ \text{ and } \frac{\partial q(x)}{\partial x_i} \leq \frac{\partial f(x)}{\partial x_i} +\epsilon',$ it follows that
		$\frac{\partial p(x)}{\partial x_i} \leq K_i^++ \epsilon'-\epsilon'=K_i^+.$ Likewise, let $i \in I^-$, i.e., $\rho_i=1$. As $\frac{\partial f(x)}{\partial x_i} \geq K_i^- \text{ and } \frac{\partial q(x)}{\partial x_i} \geq \frac{\partial f(x)}{\partial x_i} -\epsilon'$, it follows that $\frac{\partial p(x)}{\partial x_i} \geq K_i^--  \epsilon'+\epsilon'=K_i^-.$ 
		\end{proof}
		We show this result here in the case where $I^+ \cap I^- = \emptyset$ as that is the setting of interest for our consistency result. However, Proposition \ref{prop:weierstrass} (ii) also holds when both $K_i^+$ and $K_i^-$ are finite.
		
		\subsection{Proof of Proposition \ref{prop:lym.glyn}} \label{appendix:lim.glyn}
		
		We first show that our estimators $\bg$ and $\bh$ are uniformly upper bounded and Lipschitz continuous (with Lipschitz constants that do not depend on the data) over certain boxes contained within $B$. For this purpose, we introduce the following notation: let $\eta$ be a scalar such that 
		\begin{align}\label{eq:def.eta}
		0<\eta < \min_{i=1\ldots,m} \frac{u_i-l_i}{2}
		\end{align} and let $B_{\eta}\mathrel{\mathop{:}}=\{x~|~l_i+\eta \leq x_i \leq u_i-\eta,~i=1,\ldots,n \}.$ We have that (\ref{eq:def.eta}) holds if and only if $B_{\eta}$ is full-dimensional and $B_{\eta} \subset B$. The proof of Lemma \ref{lem:prop.gd} uses some ideas of \cite{lim2012consistency}, but has to account for being over a box, rather than over $\mathbb{R}^n$. The proof of Lemma \ref{lem:prop.hd} is new.
		
		\begin{lemma} \label{lem:prop.gd}
			Let $g_d$ be defined as in Proposition \ref{prop:weierstrass} and $\bg$ defined as in (\ref{eq:opt.bg}). Furthermore, let $\eta$ be a scalar such that (\ref{eq:def.eta}) holds. We have the following properties:
			\begin{enumerate}[(i)]
				\item $ \exists c_{\eta}>0$, which is independent of the data $(X_1,Y_1),\ldots(X_m,Y_m)$, such that $|\bg(x)| \leq c_{\eta}$ a.s. for all $x \in B_{3\eta/4}$.
				\item $\exists M_{\eta}>0$, which is independent of the data $(X_1,Y_1),\ldots(X_m,Y_m)$, such that $|\bg(x)-\bg(y)| \leq M_{\eta} ||x-y||$ a.s. for all $x,y \in B_{\eta}$, i.e., $\bg$ is $M_{\eta}$-Lipschitz over $B_{\eta}$.
				\item $\exists N_{\eta}>0$, which is independent of the data $(X_1,Y_1),\ldots(X_m,Y_m)$, such that $|g_d(x)-g_d(y)| \leq N_{\eta} ||x-y||$ for all $x,y \in B_{\eta}$, i.e., $g_d$ is $N_{\eta}$-Lipschitz over $B_{\eta}$.
			\end{enumerate}
		\end{lemma}
		
		\begin{proof} (i) The idea here is to control the value of $\bg$ at the corners and the analytic center of $B$. Convexity of $\bg$ enables us to conclude that $\bg$ is upper and lower bounded a.s. over $B_{3\eta/4}$. We start by using Step 2 in \cite{lim2012consistency} with $\hat{g}_n=g_{m,d}$ and $g^*=g_d$. This provides us with the following a.s. bound on $\frac1m \sum_{i=1}^m\bg^2(X_i)$:
			$$\frac1m \sum_{i=1}^m \bg^2(X_i) \leq 9E[(Y_1-g_d(X_1))^2]+3E[g_d^2(X_1)]=\mathrel{\mathop{:}} \beta \text{ a.s.}
			$$
			We use this bound to show the existence of sample points $X_i$ in the ``corners'' and around the analytic center of $B$ such that $|\bg(X_i)|$ is uniformly bounded (in $m$). To do this, we define for each vertex $i, i=1,\ldots,2^n$, of $B$, a box $B^v_i$ which is included in $B$, has vertex $i$ as a vertex, and has edges of length $\eta/4$. In other words, if vertex $i_0$ of $B$ is given by $(l_1,u_1,u_2,\ldots,u_n)$, then the corresponding box $B^v_{i_0}$ is defined as
			$B^v_{i_0}\mathrel{\mathop{:}}=\{x \in \mathbb{R}^n ~|~ l_1 \leq x_1 \leq l_1+\frac{\eta}{4}, u_2-\frac{\eta}{4} \leq x_2 \leq u_2, \ldots, u_n -\frac{\eta}{4} \leq x_n \leq u_n\}.$
			We further define 
			$B_0^v\mathrel{\mathop{:}}=\{x \in \mathbb{R}^n~|~\frac{u_i+l_i}{2}-\frac{\eta}{8} \leq x_i \leq \frac{u_i+l_i}{2}+\frac{\eta}{8}, i=1,\ldots,n\}.$
			We refer the reader to Figure \ref{fig:proof.lemma.2} for illustrations of these boxes and their relationships to other boxes appearing in the proof. Note that, for all $i=0,\ldots,2^n$, $B^v_i \subset B$ and is full dimensional. However, when $i\geq 1$, $B_i^v \cap B_{\eta/2} = \emptyset$ whereas $B_0^v \subseteq B_{\eta} \subseteq B_{3\eta/4}$. Let $\gamma_i \mathrel{\mathop{:}}=P(X \in B_i^v), i=0,\ldots,2^n, \text{ and } \gamma \mathrel{\mathop{:}}=\min \{\gamma_0, \ldots, \gamma_{2^n}\}.$
			As $B_i^v$ is full-dimensional for all $i$, it follows that $\gamma>0$. Leveraging (16) in \cite[Step 4]{lim2012consistency}, we obtain, for each $i \in \{0,\ldots,2^n\}$ and for a positive scalar $r$ such that $\frac{\beta}{r^2} \leq \frac{\gamma}{2}$, when $m$ is large enough
			$\frac1m \sum_{j=1}^m P(X_j \in B_i^v, |\bg(X_j)| \leq r)=\frac{\gamma}{2}>0.$
			We use this to obtain upper and lower bounds on $\bg(x)$ over $B_{3 \eta/4}$ which only depend on the probability distribution of $X_i$ and $B_{\eta}$ (i.e., these bounds do not depend on the number of data points, nor on the data points themselves). The proof of the lower bound requires us to show that $\bg$ is actually upper bounded over $B_{\eta/2}$. As $B_{\eta/2}$ is a superset of $B_{3\eta/4}$, this will imply that $\bg$ is upper bounded over $B_{3 \eta/4}$.
			\vspace{1mm}
			
			\noindent \textbf{Upper bound:} We show that $B_{\eta/2}$ is a subset of the convex hull of $X_{I(1)},$ $\ldots,$ $X_{I(2^n)}$. This then implies that any $x$ in $B_{\eta/2}$ can be written as a convex combination of these points, and so, using convexity of $\bg$, we can conclude that $\bg(x) \leq r$. To see that $B_{\eta/2}$ is a subset of the convex hull of $X_{I(1)},\ldots, X_{I(2^n)}$, first note that $X_{I(i)} \notin B_{\eta/2}$ for all $i=1,\ldots,2^n$ as $B_i^v \cap B_{\eta/2}= \emptyset$. Hence, either $B_{\eta/2}$ is a subset of convex hull of $X_{I(1)},\ldots,X_{I(2^n)}$ or the two sets are disjoint. We show that the former has to hold. This follows from the fact that $X_0=\frac{1}{2^n} \sum_{i=1}^{2^n} X_{I(i)}$, which is in the convex hull of $X_{I(1)},\ldots,X_{I(2^n)}$, is also in $B_{\eta/2}$. To see this, note that for a fixed component $k$ of the vectors $\{X_{I(i)}\}_i$, there are exactly $2^{n-1}$ of these components that belong to $[l_k,l_k+\frac{\eta}{4}]$ and $2^{n-1}$ that belong to $[u_k-\frac{\eta}{4}, u_k]$. This implies that the $k$-th component of $X_0$ belongs to the interval
			$[\frac{u_k+l_k}{2}-\frac{\eta}{8}; \frac{u_k+l_k}{2}+\frac{\eta}{8}]$. As $\frac{u_k+l_k}{2}+\frac{\eta}{8} \leq u_k -\frac{\eta}{2}$ and $\frac{u_k+l_k}{2}-\frac{\eta}{8} \geq l_k+\frac{\eta}{2}$ by consequence of (\ref{eq:def.eta}), we get that $X_0$ is in $B_{\eta/2}$.
			
			\vspace{1mm}
			\noindent \textbf{Lower bound:} Let $x \in B_{3\eta/4}$. We use the lower bound proof in \cite[Step 4]{lim2012consistency} to obtain that  
			$\bg(x) \geq -3r.$ Taking $c_{\eta}=\max \{r,3r\}=3r$ gives us the expected result.	
			
			\vspace{1mm}			
			\noindent (ii) Similarly to \cite[Step 5]{lim2012consistency}, as $\bg$ is convex over $B$ and a.s. bounded on $B_{3\eta/4}$ by $c_{\eta}$ from (i), there exists a constant $M_{\eta}=\frac{8c_{\eta}}{\eta}$ which is independent of the data, such that $\bg$ is $M_{\eta}$-Lipschitz over $B_{\eta}$; for a proof of this, see \cite[Theorem A]{roberts1974another}. 
			
			\vspace{1mm}
			
			\noindent (iii) As $g_d$ is continuous over $B$, $g_d$ has a maximum over $B$. Furthermore, $g_d$ is convex over $B$. The result follows, using a similar argument to (ii).
		\end{proof}
		
		\begin{figure}[]
			\centering
			\includegraphics[scale=0.27]{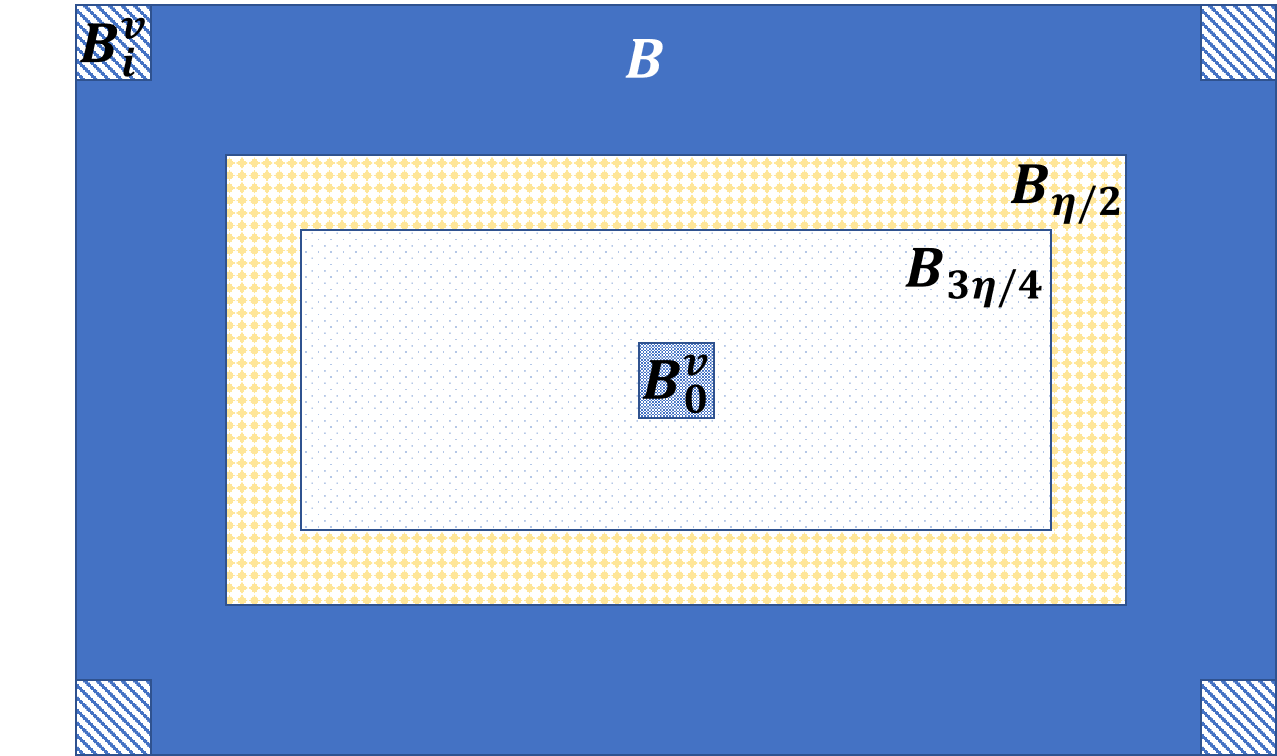}
			\caption{An illustration of the boxes that appear in the proof of Proposition \ref{prop:lym.glyn}.}\label{fig:proof.lemma.2}
		\end{figure}

		\begin{lemma} \label{lem:prop.hd}
			Let $h_d$ be defined as in Proposition \ref{prop:weierstrass} and $\bh$ defined as in (\ref{eq:opt.bh}). We have:
			\begin{enumerate}[(i)]
				\item $\exists c_{\eta}'>0$ which is independent of the data $(X_1,Y_1),\ldots, (X_m,Y_m)$ such that $|\bar{h}_{m,d}| \leq c_{\eta}'$ a.s. for all $x \in B_{3 \eta/4}$.
				\item $\exists M'_{\eta}>0$, which is independent of the data $(X_1,Y_1),\ldots(X_m,Y_m)$, such that $|\bh(x)-\bh(y)| \leq M'_{\eta} ||x-y||$ a.s. for all $x,y \in B_{\eta}$, i.e., $\bh$ is $M'_{\eta}$-Lipschitz over $B_{\eta}$.
				\item $\exists N'_{\eta}>0$, which is independent of the data $(X_1,Y_1),\ldots(X_m,Y_m)$, such that $|h_d(x)-h_d(y)| \leq N'_{\eta} ||x-y||$ for all $x,y \in B_{\eta}$, i.e., $h_d$ is $N'_{\eta}$-Lipschitz over $B_{\eta}$.
			\end{enumerate}
		\end{lemma}
		
		\begin{proof} (i) We follow the previous proof to conclude that there exists $X_{I(i)}\in B_i^v$ for $i=0,\ldots, 2^n$ such that $|\bar{h}_{m,d}(X_{I(i)})|\leq r$. Wlog, we take $\frac{\partial \bar{h}_{m,d}}{\partial x_i} \geq K_i^-$ for $i=1,\ldots,n$ (recall that $I^+ \cap I^- =\emptyset$). A similar proof holds for other cases. Let $x \in B_{3\eta/4}$ and let $i_M$ (resp. $i_m$) be such that $X_{I(i_M)} \geq x$ component-wise (resp. $X_{I(i_m)} \leq x$). These exist given the construction of $B_i^v$ and $B_{3\eta/4}$.
			For ease of notation, we denote $X_{I(i_M)}=X^M$ and $X_{I(i_m)}=X^m$ and have:
			\begin{align*}
			\bar{h}_{m,d}(x_1,x_2,\ldots,x_n)&=\bar{h}_{m,d}(X^M_1-(X^M_1-x_1),\ldots, X^M_n-(X^M_n-x_n))\\
			&\leq \bar{h}_{m,d}(X^M) + \sum_{i=1}^n |K_i^-| \cdot |X_i^M-x_i|
			\leq r +\max_{i=1,\ldots,n} |u_i-l_i| \cdot \sum_{i=1}^n |K_i^-|
			\end{align*}
			Likewise, 
			$\bar{h}_{m,d}(x_1,x_2,\ldots,x_n)\geq \bar{h}_{m,d}(X^m) - \sum_{i=1}^n |K_i^-| \cdot |x_i-X_i^m|\geq r-\max_{i=1,\ldots,n} |u_i-l_i| \cdot \sum_{i=1}^n |K_i^-|.$
			Thus, letting $c_{\eta}'=r+\max_{i=1,\ldots,n} |u_i-l_i| \cdot \sum_{i=1}^n |K_i^-|$, we obtain our result.
			
			(ii) Recall that $d$ is fixed. We use $(i)$ and Markov's inequality for polynomials (see, e.g., \cite{Kroo}) to obtain:
			$$\max_{x \in B_{\eta}} \left|\frac{\partial \bar{h}_{m,d}}{\partial x_i}\right| \leq \frac{4d^2}{\min_{i=1,\ldots,n} |u_i-l_i-2\eta|} \max_{x \in B_{\eta}} |\bar{h}_{m,d}| \leq \frac{4d^2}{\min_{i=1,\ldots,n} |u_i-l_i-2\eta|} \cdot c_{\eta'}.$$
			This implies the result by taking $M_{\eta}'= \frac{4d^2}{\min_{i=1,\ldots,n} |u_i-l_i-2\eta|} \cdot c_{\eta'}$.
			
			(iii) As $h_d$ has $K$-bounded derivatives over $B$, we get the result similarly to $(ii)$.
		\end{proof}
		
		We now give the proof of Proposition \ref{prop:lym.glyn} (i). The proof of (ii) goes through in a similar way using Lemma \ref{lem:prop.hd} instead of Lemma \ref{lem:prop.gd}. The proof of (i) has three steps: the first step shows that we can replace the set $C$ in the result by any $B_{\eta}$ for some $\eta$ such that (\ref{eq:def.eta}) holds and still get the conclusion over $C$. In the second step, we show that $g_d$ and $\bg$ are ``close'' on the random samples $X_i$; this is then used in the third step to show that the two functions are uniformly close. Step 1 is new compared to \cite{lim2012consistency}, Step 2 follows quite closely parts of the proof of \cite{lim2012consistency} but has to account for the box constraints among other minor details, Step 3 is immediate from \cite{lim2012consistency}.
		
		\begin{proof}[Proof of Proposition \ref{prop:lym.glyn} (i).] Fix $d \in \mathbb{N}$ and recall the definitions of $C_{n,d}$ and $g_d$. Let $\epsilon>0$.
		\paragraph{Step 1: going from $C$ to $B_{\eta}$.} 
		Let $C$ be any full-dimensional compact subset of $B$ such that no point of the boundary of $B$ is in $C$. As $C \cap int(B)=C$, there exists $\eta_C>0$ such that $C \subseteq B_{\eta_C}$. Furthermore, there exists $\eta_{\epsilon}>0$ such that
		\begin{align} \label{eq:rid.of.outliers}
		2\sqrt{2E[((Y_1-g_d(X_1))^2 \textbf{1}(X_1 \notin B_{\eta_{\epsilon}})]} \cdot \sqrt{5E[(Y_1-g_d(X_1))^2]} \leq \epsilon.
		\end{align}
		To see this, note that as $\eta \rightarrow 0$, $P(X_1 \notin B_{\eta}) \rightarrow 0$ with $P(X_1 \notin B)=0$ (this is a consequence of $P(X \in A)>0$ for any full-dimensional set $A$). Existence of $\eta_{\epsilon}$ then follows by expanding out the expression and using Assumptions \ref{assmpt:generation.X} and \ref{assmpt:generation.Y} together with the fact that $g_d$ is continuous over $B$ and so bounded over $B$. We let
		$\eta\mathrel{\mathop{:}}= \min\{\eta_C,\eta_{\epsilon}\}.$
		Thus defined, $\eta$ is such that (\ref{eq:def.eta}) holds as $C$ is full-dimensional and a subset of $B_{\eta}$. As a consequence, in the rest of the proof, we simply show that
		$\sup_{x \in B_{\eta}} |\bg(x)-g_d(x)| \rightarrow 0 \text{ a.s.}$
		when $m \rightarrow \infty$. 
		
		\paragraph{Step 2: showing that $\frac1m \sum_{i=1}^m (\bg(X_i)-g_d(X_i))^2 \rightarrow 0$ a.s. as $m\rightarrow \infty$.} Following Step~1 in \cite{lim2012consistency}, we have:
		\begin{equation} \label{eq:partition}
		\begin{aligned}
		\frac1m \sum_{i=1}^m (\bg(X_i)-g_d(X_i))^2 &\leq \frac2m \sum_{i=1}^m (Y_i-g_d(X_i))(\bg(X_i)-g_d(X_i)) \textbf{1}(X_i \in B_{\eta})\\
		&+\frac2m \sum_{i=1}^m (Y_i-g_d(X_i))(\bg(X_i)-g_d(X_i)) \textbf{1}(X_i \notin B_{\eta})
		\end{aligned}
		\end{equation}
		Step 3 of \cite{lim2012consistency} combined to (\ref{eq:rid.of.outliers}) enables us to show that for large enough $m$,
		\begin{align} \label{eq:terms.outside}
		& \frac2m \sum_{i=1}^m (Y_i-g_d(X_i))(\bg(X_i)-g_d(X_i)) \textbf{1}(X_i \notin B_{\eta}) \leq \epsilon \text{ a.s.}.
		\end{align}
		
		We now focus on the term that includes the sample points in $B_{\eta}$. As done in \cite{lim2012consistency}, we replace $\bg$ by a deterministic approximation and then apply the strong law of large numbers. Let 
		$\mathcal{C}=\{\text{polynomials }p:B_{\eta} \mapsto \mathbb{R} \text{ of degree }d,  M_{\eta}\text{-Lipschitz with }|p(x)| \leq c_{\eta}, \forall x \in B_{\eta}\},$
		where $M_{\eta}$ and $c_{\eta}$ are the constants given in Lemma \ref{lem:prop.gd}, which do not depend on the data $(X_1,Y_1),\ldots,(X_m,Y_m)$. Lemma \ref{lem:prop.gd} implies that $\bg$ belongs to $\mathcal{C}$ for large enough $m$. Furthermore, given that $\mathcal{C}$ is a subset of the set of continuous functions over the box $B_{\eta}$ and given that all functions in $\mathcal{C}$ are uniformly bounded and Lipschitz, it follows from Ascoli-Arzel\'a's theorem that $\mathcal{C}$ is compact in the metric $d(f,g)=\sup_{x \in B_{\eta}} |f(x)-g(x)|$. As a consequence, $\mathcal{C}$ has a finite $\epsilon$-net: we denote by $p_1,\ldots,p_R$ the polynomials belonging to it. Hence, for large enough $m$, there exists $r\in \{1,\ldots,R\}$ such that $\sup_{x \in B_{\eta}} |p_r(x)-\bg(x)|<\epsilon$, which is the deterministic approximation of $\bg$. Following Step 7 of \cite{lim2012consistency}, we show that:
		\begin{align*}
		&\frac2m \sum_{i=1}^m (Y_i -g_d(X_i))(\bg(X_i)-g_d(X_i)) \textbf{1}(X_i \in B_{\eta})\\
		&\leq \frac2m \cdot \epsilon \sum_{i=1}^m |Y_i-g_d(X_i)| + \max_{j=1,\ldots,R} \frac2m \sum_{i=1}^m (Y_i-g_d(X_i))(p_j(X_i)-g_d(X_i)) \textbf{1}(X_i \in B_{\eta}).
		\end{align*}
		As $g_d$ is bounded over $B$, we use the strong law of large numbers to obtain
		\begin{align} \label{eq:term.1}
		\frac1m \sum_{i=1}^m |Y_i-g_d(X_i)| \leq 2E[|Y_1-g_d(X_1)|] \text{ a.s., for large enough $m$. }
		\end{align} 
		We also have for any $j \in \{1,\ldots,R\}$,using Proposition \ref{prop:weierstrass} and Assumption \ref{assmpt:generation.Y},
		\begin{align*}
		&\frac2m \sum_{i=1}^m (Y_i-g_d(X_i))(p_j(X_i)-g_d(X_i)) \textbf{1}(X_i \in B_{\eta}) \\
		&= \frac2m \sum_{i=1}^m (Y_i-f(X_i)+f(X_i)-g_d(X_i))(p_j(X_i)-g_d(X_i)) \textbf{1}(X_i \in B_{\eta})\\
		&\leq \frac2m \sum_{i=1}^m \nu_i(p_j(X_i)-g_d(X_i)) \textbf{1}(X_i \in B_{\eta})+\frac2m \epsilon \sum_{i=1}^m |p_j(X_i)-g_d(X_i)| \textbf{1}(X_i \in B_{\eta}).
		\end{align*}
		Given Assumption \ref{assmpt:generation.X} and the fact that $h_j$ is uniformly bounded over $B_{\eta}$, from the strong law of large numbers, for any $j \in \{1,\ldots,r\}$ and for large enough $m$, we have
		\begin{align} \label{eq:term.2}
		\frac2m \sum_{i=1}^m \nu_i(p_j(X_i)-g_d(X_i)) \textbf{1}(X_i \in B_{\eta}) \leq \epsilon \text{ a.s.}
		\end{align}
		Similarly, using the strong law of large numbers again, for large enough $m$,
		\begin{align} \label{eq:term.3}
		\frac1m \sum_{i=1}^m |p_j(X_i)-g_d(X_i)| \textbf{1}(X_i \in B_{\eta}) \leq  2E[|p_j(X_1)-g_d(X_1)|\textbf{1}(X_1 \in B_{\eta})] \text{ a.s.}
		\end{align}
		Combining (\ref{eq:term.1}), (\ref{eq:term.2}), and (\ref{eq:term.3}), we conclude that for large enough $m$,
		\begin{align*}
		&\frac2m \sum_{i=1}^m (Y_i -g_d(X_i))(\bg(X_i)-g_d(X_i)) \textbf{1}(X_i \in B_{\eta}) \\ &\leq 2\epsilon \left(2E[|Y_1-g_d(X_1)|]+\max_{j=1,\ldots,R}2E[|p_j(X_1)-g_d(X_1)|\textbf{1}(X_1 \in B_{\eta})]\right) +\epsilon.
		\end{align*}
		Combining this with (\ref{eq:terms.outside}) in (\ref{eq:partition}), we obtain our conclusion.
		
		\paragraph{Step 3: showing that $\sup_{x \in B_{\eta}} |g_d(x)-\bg(x)| \rightarrow 0$ a.s. when $m \rightarrow \infty$.} This follows immediately from Step (8) in \cite{lim2012consistency}.
		\end{proof}
		
		\subsection{Proof of Proposition \ref{prop:approx.consistency}} \label{appendix:approx.consistency}

		Our first step is to show that, given a sequence of optimization problems, convergence of the objective values implies convergence of the minimizers under certain assumptions.
		
		\begin{theorem}\label{th:equ}
			Let $f: \mathbb{R}^n \mapsto \mathbb{R}$ be a strictly convex quadratic function and let $x^*$ be its (unique) minimizer (which must exist as $f$ is coercive) over a closed convex set $S \subseteq \mathbb{R}^n$. Suppose that a sequence $\{x_j\}_j \subseteq S$ is such that $\lim_{ j\rightarrow \infty} f(x_j) =f(x^*)$. Then, $\lim_{j \rightarrow \infty} ||x_j-x^*||=0.$
		\end{theorem}
		The proof of this theorem requires the following lemma.

			\begin{lemma}\label{lem:proof}
				Let $x^* \in \mathbb{R}^n$ and $\{S_j\}_j \subset \mathbb{R}^n$ be a sequence of sets satisfying (i) $S_j, j=1,2,\ldots$ is nonempty and compact, (ii) $S_{j+1} \subseteq S_j$, for $j=1,2,\ldots$, (iii) $\cap_j S_j=\{x^*\}$,
				then, $\forall \epsilon>0, ~\exists k_0$ such that $S_k \subseteq B(x^*,\epsilon),~\forall k \geq k_0$, where $B(x^*,\epsilon)$ is the ball centered at $x^*$ and with radius $\epsilon>0$.
			\end{lemma}
			
			\begin{proof} Let $d_k=diam(S_k)\mathrel{\mathop{:}}=\max_{y,z \in S_k} ||y-z||$, which is finite for all $k$ by virtue of (i). We claim that $\lim_{k \rightarrow \infty} d_k=0.$ Indeed, due to (ii), $\{d_k\}_k$ is nonincreasing and bounded below by zero. Thus, it converges to some limit $d^*$, with $d_k\geq d^*, \forall k$. Suppose for the sake of contradiction that $d^*>0$. As $S_k$ is compact, for all $k$, there exist $y_k$ and $z_k$ in $S_k$ such that $||y_k-z_k||=d_k$. As $\{y_k\}_k$ is a bounded sequence, it has a convergent subsequence $\{y_{\theta(k)}\}$. Likewise, $\{z_{\theta(k)}\}$ is a bounded sequence, so it has a convergent subsequence $\{z_{\phi(k)}\}$. Let
			$\bar{y}\mathrel{\mathop{:}}= \lim_{k \rightarrow \infty} \{y_{\phi(k)}\} \text{ and } \bar{z} \mathrel{\mathop{:}}=\lim_{k \rightarrow \infty} \{z_{\phi(k)}\}.$
			As $||y_{\phi(k)}-z_{\phi(k)}||=d_{\phi(k)}\geq d^*, \forall k,$ and $||\cdot||$ is continuous, we have that $||\bar{y}-\bar{z}|| \geq d^*$. Now, for any $i \in \mathbb{Z}_+$, (ii) implies that $y_{\phi(k)}, z_{\phi(k)} \in S_i$ for $\phi(k) \geq i$. Since $S_i$ is closed by virtue of (i), we must have $\bar{y}, \bar{z} \in S_i$. As the $i$ chosen before was arbitrary and using (iii), it follows that $\bar{y} \in \cap_k S_k=\{x^*\} \text{ and } \bar{z} \in \cap_k S_k=\{x^*\},$
			which contradicts $||\bar{y}-\bar{z}||\geq d^*>0$. This proves the claim.
			
			From the claim, we know that $\forall \epsilon>0,~\exists k_0$ such that $d_k<\epsilon,~\forall k\geq k_0$. As $d_k=\max_{y,z \in S_k} ||y-z|| \geq \max_{z \in S_k} ||x^*-z||,$
			$\max_{z \in S_k} ||x^*-z||<\epsilon$, $\forall k \geq k_0$. This implies that $S_k \subseteq B(x^*,\epsilon),\forall k \geq k_0.$
			\end{proof}
			
			\begin{proof}[Proof of Theorem \ref{th:equ}.]
			Let $S_k\mathrel{\mathop{:}}=\left\{x \in S~|~f(x) \leq f(x^*)+\frac1k\right\}.$
			The assumptions of Lemma \ref{lem:proof} are met: (i) holds as $S_k$ is nonempty (it contains $x^*$) and compact (it is contained in the sublevel set of a coercive function); (ii) holds as $S_{k+1}\subseteq S_k$ by definition; and (iii) holds as $x^*$ is contained in $S_k$ for all $k$ and it is the unique minimizer of $f$ over $S$. We can thus apply Lemma \ref{lem:proof}. Let $\epsilon>0$. From Lemma \ref{lem:proof}, 
			$\exists k_0 \text{ s.t. } \forall k\geq k_0,  x\in S \text{ and } f(x) \leq f(x^*)+\frac1k \Rightarrow ||x-x^*||\leq \epsilon.$
			Now, as $\lim_{j \rightarrow \infty} f(x_j)=f(x^*)$, then $\exists j_0$ such that $\forall j \geq j_0, |f(x_j)-f(x^*)|\leq \frac{1}{k_0}$. As $x_j \in S,\forall j$, it follows that $||x_j-x^*||\leq \epsilon$, which concludes the proof.
			\end{proof}
			
			\begin{proof}[Proof of Proposition \ref{prop:approx.consistency}.]
			Let $m,d,r \in \mathbb{N}$. We identify polynomials in $P_{n,d}$ with their coefficients in $\mathbb{R}^{\binom{n+d}{d}}$, using the same notation for both. We apply Theorem \ref{th:equ} with $f(c)=\sum_{i=1}^m (Y_i-c(X_i))^2,$ which is a strictly convex quadratic function by virtue of our assumptions on $m$ and on the data. To show (i), we let $S=\{c \in \mathbb{R}^{\binom{n+d}{d}}~|~H_c(x) \succeq 0, \forall x\in B\}$ with $\bg$ being $x^*$ and $\{\tilde{g}_{m,d',r}\}_{d'}$ being $\{x_j\}_j$. Similarly, for (ii), we let $S=\{c \in \mathbb{R}^{\binom{n+d}{d}}~|~ \frac{\partial c(x)}{\partial x_i} \geq K_i^-, \forall x\in B, i \in I^-, ~\frac{\partial c(x)}{\partial x_i} \leq K_i^+,\forall x \in B,i \in I^+\}$ with $\bh$ being $x^*$ and $\{\tilde{h}_{m,d',r}\}_{d'}$ being $\{x_j\}_j$. 
			
			It remains to show that $\{f(\tilde{g}_{m,d',r})\}_{d'}$ converges to $\{f(\bg)\}$ to show that (i) holds. Let $\epsilon>0$ and let $\epsilon'^2=\frac{\epsilon^2}{mn^2(n+9n(n-1))^2}$. Define $D \in \mathbb{R}^{n \times n}$ to be a diagonal matrix with $i^{th}$ entry $D_{ii}=\frac{1}{u_i-l_i}>0$ and $d \in \mathbb{R}$ to be a vector with $i^{th}$ entry $d_i=-1-\frac{l_i} {u_i-l_i}$. We then let $z=Dx+d$, with $z \in [-1,1]^n$ when $x \in B$. Consider now $\bar{g}_{m,d}^{\infty}(z)\mathrel{\mathop{:}}=\bar{g}_{m,d}(D^{-1}z-D^{-1}d)=\bar{g}_{m,d}(x)$. As $\bg$ is convex over $B$, $\bar{g}_{m,d}^{\infty}$ is convex over $[-1,1]^n$. By Theorem \ref{thm:conv}, there exists $d_0'$ such that, for all $d' \geq d'_0$, $f^{\infty}_{m,d'}(z)\mathrel{\mathop{:}}=\bar{g}_{m,d}^{\infty}(z)+\epsilon' \Theta_{d'}(z) \text{ is sos-convex}.$ We let $f_{m,d'}(x) \mathrel{\mathop{:}}=f^{\infty}_{m,d'}(Dx+d)$, which is also sos-convex. Then,
			$$\sum_{i=1}^m (f_{m,d'}(X_i)-\bar{g}_{m,d}(X_i))^2=\sum_{i=1}^m \epsilon'^2 (\Theta_{d'}(DX_i+d))^2 \leq \sum_{i=1}^m \epsilon'^2 ||\Theta_{d'}||_2^2 \cdot  ||DX_i+d||_2^2 \leq \epsilon^2,$$
			where the second inequality follows from the Cauchy-Schwarz inequality with $||\Theta_{d'}||_2^2$ being the 2-norm of the coefficients of $\Theta_{d'}$, and the third inequality follows by definition of $\epsilon'$ and the fact that $DX_i+d \in [-1,1]^n$. Now, by definition of $\tilde{g}_{m,d',r}$, as $f_{m,d'}$ is sos-convex, we have 
			$\sum_{i=1}^m (Y_i-\tilde{g}_{m,d',r}(X_i))^2 \leq \sum_{i=1}^m (Y_i-f_{m,d'}(X_i))^2$, which implies $\sum_{i=1}^m (\bg(X_i)-\tilde{g}_{m,d',r}(X_i))^2 \leq \sum_{i=1}^m (\bar{g}_{m,d}(X_i)-f_{m,d'}(X_i))^2 \leq \epsilon^2$ from the triangle inequality and the previous inequality. Thus,
			$$\sqrt{\sum_{i=1}^m (Y_i-\tilde{g}_{m,d',r}(X_i))^2} -\epsilon \leq \sqrt{\sum_{i=1}^m (Y_i-\bg(X_i))^2} \leq \sqrt{\sum_{i=1}^m (Y_i-\tilde{g}_{m,d',r}(X_i))^2}, $$
			where the first inequality is a consequence of the reverse triangle inequality and the second is by definition of $\bg$. It follows that $\lim_{d' \rightarrow \infty} f(\tilde{g}_{m,d',r})=f(\bg)$. We proceed similarly for (ii).
			\end{proof}


		\section{Proofs and Additional Figures for Section \ref{sec:sc.pr}}
		
		\subsection{Proof of Theorem \ref{th:charac}} \label{appdx:charc}
		\begin{lemma}\label{lem:hard.K.bded}
			Given a polynomial $p$ of degree $d \geq 3$, a vector $K$ as in Definition \ref{def:bdr}, and a box $B$ as defined in (\ref{eq:box}), it is strongly NP-hard to test whether $p$ has $K$-bounded derivatives over $B$.
		\end{lemma}
		
		\begin{proof}
		We first show the result for $d=3$ via a reduction from MAX-CUT. Recall that in an unweighted undirected graph $G=(V,E)$ with no self-loops, a \emph{cut} partitions the $n$ nodes of the graph into two sets, $S$ and $\bar{S}$. The size of the cut is the number of edges connecting a node in $S$ to a node in $\bar{S}$ and MAX-CUT is the problem: given a graph $G$ and an integer $k$, test whether $G$ has a cut of size at least $k$. It is well known that MAX-CUT is NP-hard~\cite{GareyJohnson_Book}. 
		
		Let $A$ be the adjacency matrix of the graph, i.e., $A \in \{0,1\}^{n \times n}$ with $A_{ij}=1$ if $\{i,j\} \in E$ and $A_{ij}=0$ otherwise, and let $\gamma\mathrel{\mathop{:}}=\max_{i} \{A_{ii}+\sum_{j\neq i} |A_{ij}|\}$. Note that $\gamma$ is the maximum degree in the graph and so an integer, and an upper bound on the largest eigenvalue of $A$ \cite{Gershgorin}.
		
		We show that $G$ does not have a cut of size $ \geq k$ if and only if the cubic polynomial 
		\begin{small}
			\begin{align} \label{eq:def.p}
			p(x_1,\ldots,x_n)&= \sum_{j=2}^n \frac{x_1^2A_{1j}x_j}{4}+\frac{x_1}{2}\sum_{1<i<j\leq n} x_i A_{ij}x_j-\frac{\gamma x_1^3}{12}-\frac{\gamma x_1}{4} \sum_{i=2}^n x_i^2 +x_1 \left( k+\frac{n\gamma}{4}-\frac{e^TAe}{4}\right)
			\end{align}
		\end{small}
		has $K$-bounded derivatives on $B=[-1,1]^n$, where 
		\begin{align} \label{eq:K.hard}
		K_1^-=0,~ K_1^+=n^2+k,~ K_2^-=\ldots=K_n^-=-n,~ K_2^+=\ldots=K_n^+=n+1.
		\end{align}
		Letting $x=(x_1,\ldots,x_n)^T$, we compute the partial derivatives of $p$:
		\begin{align*}
		\frac{\partial{p}(x)}{\partial x_1}&=\frac{1}{4} x^T(A-\gamma I)x+(k+\frac{n\gamma}{4}-\frac14 e^TAe),\\
		\frac{\partial{p}(x)}{\partial x_i}&=\frac{1}{4} x_1^2A_{1i}+\frac12 x_1 \cdot \sum_{j \neq i, j>1} x_i A_{ij}-\frac{\gamma}{2} x_1 \cdot x_i, ~i=2,\ldots,n.
		\end{align*}
		As $x \in [-1,1]^n$ and $\gamma \leq n$, it is straightforward to check that 
		$\frac{\partial p(x)}{\partial x_1} \leq K_1^+ \text{ and } K_i^- \leq \frac{\partial p(x)}{\partial x_i} \leq K_i^+, $ for $i=2,\ldots,n,~\forall x\in B.$
		The statement to show thus becomes: $G$ does not have a cut of size $\geq k$ if and only if 
		$\frac{\partial p(x)}{\partial x_1}=\frac14 x^T(A-\gamma I)x+k+\frac{n\gamma}{4}-\frac14 e^TAe \geq K_1^-=0,~\forall x\in B.$
		The converse implication is easy to prove: if $\frac{\partial p(x)}{\partial x_1}\geq 0$ for all $x\in B$, then, in particular, $\frac{\partial p(x)}{\partial x_1}\geq 0$ for $x \in \{-1,1\}^n.$ When $x \in \{-1,1\}^n$, $\gamma x^Tx=\gamma n$, and so 
		$k \geq \frac{1}{4}e^TAe-\frac{1}{4}x^TAx,~\forall x\in \{-1,1\}^n.$
		Any cut in $G$ can be encoded by a vector $x \in \{-1,1\}^n$ with $x_i=1$ if node $i$ is on one side of the cut and with $x_i=-1$ if node $i$ is on the other. The size of the cut is then given by $\frac{1}{4}e^TAe-\frac{1}{4}x^TAx$~\cite{maxcut_gw1}. Hence, this is equivalent to stating that all cuts in $G$ are of size less than or equal to $k$.
		
		For the implication, if $G$ does not have a cut of size greater than or equal to $k$, then, as established above, $k \geq \frac{1}{4}e^TAe-\frac{1}{4}x^TAx,~\forall x\in \{-1,1\}^n$. This is equivalent to
		\begin{align} \label{eq:ineq}
		\frac14 x^T(A-\gamma I)x \geq -k-\frac{n\gamma}{4}+\frac14 e^TAe,~\forall x\in \{-1,1\}^n.
		\end{align} 
		Now, by definition of $\gamma$, $A-\gamma I \preceq 0$, i.e., $x^T (A-\gamma I) x$ is concave. Let $y\in B$. We have $y=\sum_{i=1}^{2^n} \lambda_i x_i$ where $x_i$ are the corners of $B$, which are in $\{-1,1\}^n$, $\lambda_i \geq 0$, $i=1,\ldots,2^n$, and $\sum_{i=1}^{2^n} \lambda_i=1$. As $y \mapsto y^T(A-\gamma I)y$ is concave and (\ref{eq:ineq}) holds, 
		$$\frac{1}{4} y^T(A-\gamma I) y \geq \sum_{i=1}^{2^n} \lambda_i x_i^T(A-\gamma I)x_i \geq \sum_{i=1}^{2^n} \lambda_i (-k-\frac{n \gamma }{4}+\frac14 e^TAe)=-k-\frac{n \gamma }{4}+\frac14 e^TAe.$$
		This concludes the proof for $d=3$.
		For $d\geq 4$, we define
		$$\tilde{p}(x_1,\ldots,x_n,x_{n+1}) \mathrel{\mathop{:}}=p(x_1,\ldots,x_n)+x_{n+1}^d \in P_n^d, ~\tilde{K}\mathrel{\mathop{:}}=(K,0,1), \text{ and } \tilde{B} \mathrel{\mathop{:}}=B \times [0,1],$$
		with $p$ as in (\ref{eq:def.p}), $K$ as in (\ref{eq:K.hard}), and $B =[-1,1]^n$. We compute:
		$$\frac{\partial{\tilde{p}}(x_1,\ldots,x_{n+1})}{\partial x_i}=\frac{\partial p(x_1,\ldots,x_n)}{\partial x_i} \text{ for } i=1,\ldots,n, \text{ and } \frac{\partial{\tilde{p}}(x_1,\ldots,x_{n+1})}{\partial x_{n+1}}= d\cdot x_{n+1}^{d-1}.$$
		As $d\cdot x_{n+1}^{d-1} \in [0,1]$ when $x_{n+1} \in [0,1]$, it follows that $\tilde{p}$ has $\tilde{K}$-bounded derivatives over $\tilde{B}$ if and only if $p$ has $K$-bounded derivatives over $B$.
		\end{proof}
		
		\begin{lemma}[\cite{ahmadi2019complexity,hall2018optimization1}] \label{lem:conv.hard}
			Given a polynomial $p$ of degree $d \geq 3$ and a box $B$ as defined in (\ref{eq:box}), it is strongly NP-hard to test whether $p$ is convex over $B$.
		\end{lemma}
		
		\begin{lemma}[\cite{nicolaides1972class}]\label{lem:poly.interp}
			Let $S_n$ be a closed full-dimensional simplex in $\mathbb{R}^n$ with vertices $\theta_0,\ldots,\theta_n$. Let $d \geq 1$ be an integer and let $Z_d$ denote the set of numbers $\{0,1/d,2/d,\ldots,1\}$. We associate with $S_n$ the discrete point set $\Gamma(n,d)$:
			\begin{align} \label{def:gam}
			\Gamma(n,d)\mathrel{\mathop{:}}=\{x \in \mathbb{R}^n ~|~ x=\sum_{i=0}^n \lambda_i \theta_i, \lambda_i \in Z_d, \sum_{i=0}^n \lambda_i=1\}.
			\end{align} 
			Let $w_i, i=1,\ldots,\binom{n+d}{d}$ be the points contained in $\Gamma(n,d)$ and let $f_i, i=1,\ldots,\binom{n+d}{d}$ be arbitrary numbers in $\mathbb{R}$, then there is exactly one polynomial $p \in P_{n,d}$ such that $p(w_i)=f_i, i=1,\ldots \binom{n+d}{d}$.
		\end{lemma}
		
		\begin{proof}[Proof of Theorem \ref{th:charac}.]
		Let $d\geq 3$. We only show NP-hardness of BD-DER-REG-$d$ (the proof of NP-hardness of CONV-REG-$d$ is analogous). This is done via a reduction from the problem in Lemma \ref{lem:hard.K.bded}. Consider an instance of the problem, i.e., a polynomial $\tilde{p}$ of degree $d$, a vector $\tilde{K}$, and a box $\tilde{B}$ and construct a reduction by taking $K=\tilde{K}$, $B=\tilde{B}$, and $t=0$. Recalling that $\tilde{B}=[\tilde{l}_1,\tilde{u}_1] \times \ldots \times [\tilde{l}_n,\tilde{u}_n]$, set 
		$\theta_0=(\tilde{l}_1,\tilde{l}_2,\ldots,\tilde{l}_n), \theta_1=(\tilde{u}_1,\tilde{l}_2, \ldots,\tilde{l}_n), \ldots, \theta_n=(\tilde{l}_1,\tilde{l}_2,\ldots,\tilde{u}_n)$
		and take $X_1,\ldots, X_{m}$, where $m=\binom{n+d}{d}$, to be the points contained in $\Gamma_{n,d}$ as in (\ref{def:gam}). Note that $X_1,\ldots,X_m$ are rational, can be computed in polynomial time, and belong to $B$ as $\theta_0,\ldots,\theta_n \in B$. We then take $Y_i=\tilde{p}(X_i), i=1,\ldots,m$. It is easy to see that the answer to BD-DER-REG-$d$ is YES if and only if $\tilde{p}$ has $\tilde{K}$-bounded derivatives on $\tilde{B}$. The converse is immediate by taking $p=\tilde{p}$. For the implication, if the answer to BD-DER-REG-$d$ is YES, then there exists a polynomial $p$ of degree $d$ with $K$-bounded derivatives over $B$ such that $\sum_{i=1}^m (Y_i-p(X_i))^2=0$, i.e., $p(X_i)=Y_i$ for $i=1,\ldots,m$. From Lemma \ref{lem:poly.interp}, as $\tilde{p} \in P_{n,d}$, it must be the case that $p=\tilde{p}$ and $\tilde{p}$ has $K$-bounded derivatives over $B$.
		
		Now, let $d=1$ and denote by $\mathbf{K}^+=(K_1^+,\ldots,K_n^+)$ and by $\mathbf{K}^-=(K_1^-,\ldots,K_n^-)$. The answer to BD-DER-REG-$d$ is YES if and only if the optimal value of
		$\min_{a_0 \in \mathbb{R}, ~K^- \leq a \leq K^+~} \sum_{i=1}^m (Y_i -a^TX_i-a_0)^2$
		is less than or equal to $t$ (here the inequalities are component-wise). As this is a QP, BD-DER-REG-$d$ is in P when $d=1$. Similarly, the answer to CONV-REG-$d$ is YES if and only if the optimal value of $\min_{a_0 \in \mathbb{R}, a\in \mathbb{R}^n} \sum_{i=1}^m (Y_i -a^TX_i-a_0)^2$ is $ \leq t$ (as any affine function is convex). This can be done by solving a system of linear equations, which implies that CONV-REG-$d$ is in P for $d=1$.
		
		Let $d=2$ and let $S^n$ be the set of $n \times n$ symmetric matrices. We let $p(x)=x^T Qx+b^Tx+c$, where $Q\in S^n,$ $b\in \mathbb{R}^n,$ $c\in \mathbb{R}$. We answer YES to BD-DER-REG-$d$ if and only if the optimal value of
		\begin{equation}\label{eq:hardness}
		\begin{aligned}
		&\min_{Q\in S^n, b\in \mathbb{R}^n, c \in \mathbb{R}} \sum_{i=1}^m (Y_i -X_i^TQX_i -b^TX_i-c)^2\\
		&\text{s.t. } \mathbf{K}^- \leq 2Qx+b\leq \mathbf{K}^+, \forall x\in B
		\end{aligned}
		\end{equation}
		is less than or equal to $t$. As $B$ is a compact, convex, and full-dimensional polyhedron, one can use \cite[Proposition I.1.]{handelman1} to rewrite the condition $2Qx+b\leq \mathbf{K}^+, \forall x\in B$ equivalently as $\mathbf{K}^+ -2Qx-b =\lambda+\Lambda^-(x-l)+\Lambda^+(u-x),~ \lambda \geq 0, ~\Lambda^+,\Lambda^-\geq 0$
		where $\lambda \in \mathbb{R}^n$, $\Lambda^{\pm} \in \mathbb{R}^{n \times n}$ are additional variables. A similar technique can be used for $2Qx+b\geq \mathbf{K}^-, \forall x\in B$. Thus, (\ref{eq:hardness}) is equivalent to a QP and so testing whether its objective value is less than or equal to $t$ can be done in polynomial time. Now, the answer to CONV-REG-$d$ is YES if and only if the optimal value of 
		$\min_{Q\in S^n, Q\succeq 0, b\in \mathbb{R}^n, c \in \mathbb{R}} \sum_{i=1}^m (Y_i -X_i^TQX_i -b^TX_i-c)^2$
		is less than or equal to $t$. As this is a polynomial-size SDP, the result follows.
			\end{proof}

		\subsection{Additional Experimental Results for Section \ref{sec:sc.pr}} \label{appendix:add.comp}
		
		We provide here a detailed comparison of the optimal value of the sample problem against that of the sos-shape constrained problem, as explained in Section \ref{sec:sc.pr}. We consider the application in Section \ref{subsec:opt.transport} and plot in Table \ref{tab:opt.transport} the ratio of the optimal value of the sample problem against that of the sos-shape constrained problem for varying $d$ and $L ,\ell$, with $n=3$, $R=1$, and $N'=100$. 
		\begin{table}
			\centering
			\footnotesize
			\begin{tabular}{|c|c|c|c|c|c|c|}
				\hline
				& \multicolumn{3}{c|}{$d=3$} & \multicolumn{3}{|c|}{$d=5$} \\
				\hline & $L=2$ & $L=5$ & $L=10$ & $L=2$ & $L=5$ & $L=10$\\
				\hline
				$\ell=0.1 $ & 0.927 & 0. 930 & 0.939 & 0.849 & 0. 967 & 0.956 \\
				$\ell=0.5 $ & 0.980 & 0.979 & 0.977 & 0.956 & 0.965 & 0.963 \\
				$\ell =1 $ & 0.999 & 0.999 & 0.999 & 0.998 & 0.998 & 0.999\\
				\hline
			\end{tabular}
			\caption{Ratio of the objective of the sample problem against that of the sos-shape constrained problem for varying $d$, $L$ and $\ell$ for the application in Section \ref{subsec:opt.transport}.}
			\label{tab:opt.transport}
		\end{table}

			\section{Proofs and Additional Figures for Section \ref{sec:apps}}\label{appendix:sec.apps}

			\begin{proof}{Proof of Proposition \ref{prop:opt.val.proof}.} We prove each statement separately. For (i), for fixed $y \in \mathbb{R}^m$, $b \mapsto -\langle b,y \rangle$ is linear. Following \cite[Section 3.2.3]{BoydBook}, $v_D(b,c)$ is convex in $b$. As $v_D(b,c)=v(b,c)$ by strong duality, the result follows. For (ii), for fixed $x \in \mathbb{R}^n$, $c \mapsto \langle c,x \rangle$ is linear. Following \cite[Section 3.2.3]{BoydBook}, $v_P(b,c)$ is convex in $c$. As $v_P(b,c)=v(b,c)$ by strong duality, the result follows. For (iii), we refer here to (\ref{eq:conic.P}) by $P_{b,c}$ to reflect the dependency of (\ref{eq:conic.P}) on $b$ and $c$. Let $b\preceq_{\mathcal{K}} b'$. Any feasible solution to $P_{b,c}$ is feasible to $P_{b',c}$. Indeed, if $x$ is a feasible solution to $P_{b,c}$, i.e., $b-Ax \in \mathcal{K}$, then $b'-Ax=b'-b+b-Ax \in \mathcal{K}$. It follows that the feasible set of $P_{b,c}$ is a subset of that of $P_{b',c}$, and so $v(b',c) \leq v(b,c)$.
			\end{proof}
			
			\begin{proof}{Proof of Proposition \ref{prop:shape.inventory}.}
				Property (i) follows from Proposition \ref{prop:opt.val.proof} (i) and (\ref{eq:constraint.L}). Proprety (ii) follows from Proposition \ref{prop:opt.val.proof} (ii). For (iii), it is straightforward to see that $v$ does not decrease with $\beta^+$ and $\beta^-$ as $\sum_{t=2}^T z_{t}^{\pm} \geq 0.$ It is also straightforward to see that $v$ does not decrease with $L$ as, when $L$ increases, the left-hand side inequality in (\ref{eq:constraint.L}) becomes harder to satisfy, and so the feasible set shrinks. We give a proof by contradiction that $v$ does not increase with $\alpha^+$ (a similar reasoning applies for $\alpha^-$). Let $\alpha_1<\alpha_2$ and let $v_1$ (resp. $v_2$) be the optimal value of (\ref{eq:semi.inf.opt}) when $\alpha^+=\alpha_1$ (resp $\alpha_2$). By way of contradiction, we assume that $v_1>v_2$. We denote an optimal solution to (\ref{eq:semi.inf.opt}) when $\alpha^+=\alpha_2$ by $\{\tilde{w}_t\}$, $\{\tilde{z}_t^+\}$, $\{\tilde{z}_t^-\}$,$\tilde{C}$,$\{\tilde{y}_{\tau}^t\}$,$\{\tilde{q}_{\tau}^t\}$, $\{\tilde{v}_{\tau}^t\}$, $\{\tilde{u}_{\tau}^t\}$ and take $\forall t, \tau$,
				\begin{align}\label{eq:new.sol}
				w_t=\tilde{w}_t,   z_t^{\pm}=\tilde{z}_t^\pm,~ C=\tilde{C},~ y_{\tau}^t=\tilde{y}_{\tau}^t, ~q_{\tau}^t=\tilde{q}_{\tau}^t,~ v_{\tau}^t=\tilde{v}_\tau^t, ~u_{\tau}^t=\frac{\alpha_2}{\alpha_1}\tilde{u}_{\tau}^t.
				\end{align}
				One can check that as $\alpha_2/\alpha_1 \geq 1$, this solution is feasible to (\ref{eq:semi.inf.opt}) when $\alpha^+=\alpha_1$. It is also easy to see that the corresponding optimal value equals $v_2$. As $v_1>v_2$, this contradicts optimality of $v_1$.
			\end{proof}

			\begin{figure}[]
				\setlength\tabcolsep{1pt}
				\begin{center}
					\begin{tabular}{cccc}
						\subfloat[$\ell = 0.1$, $L = 2$]{\includegraphics[width = 0.2\linewidth]{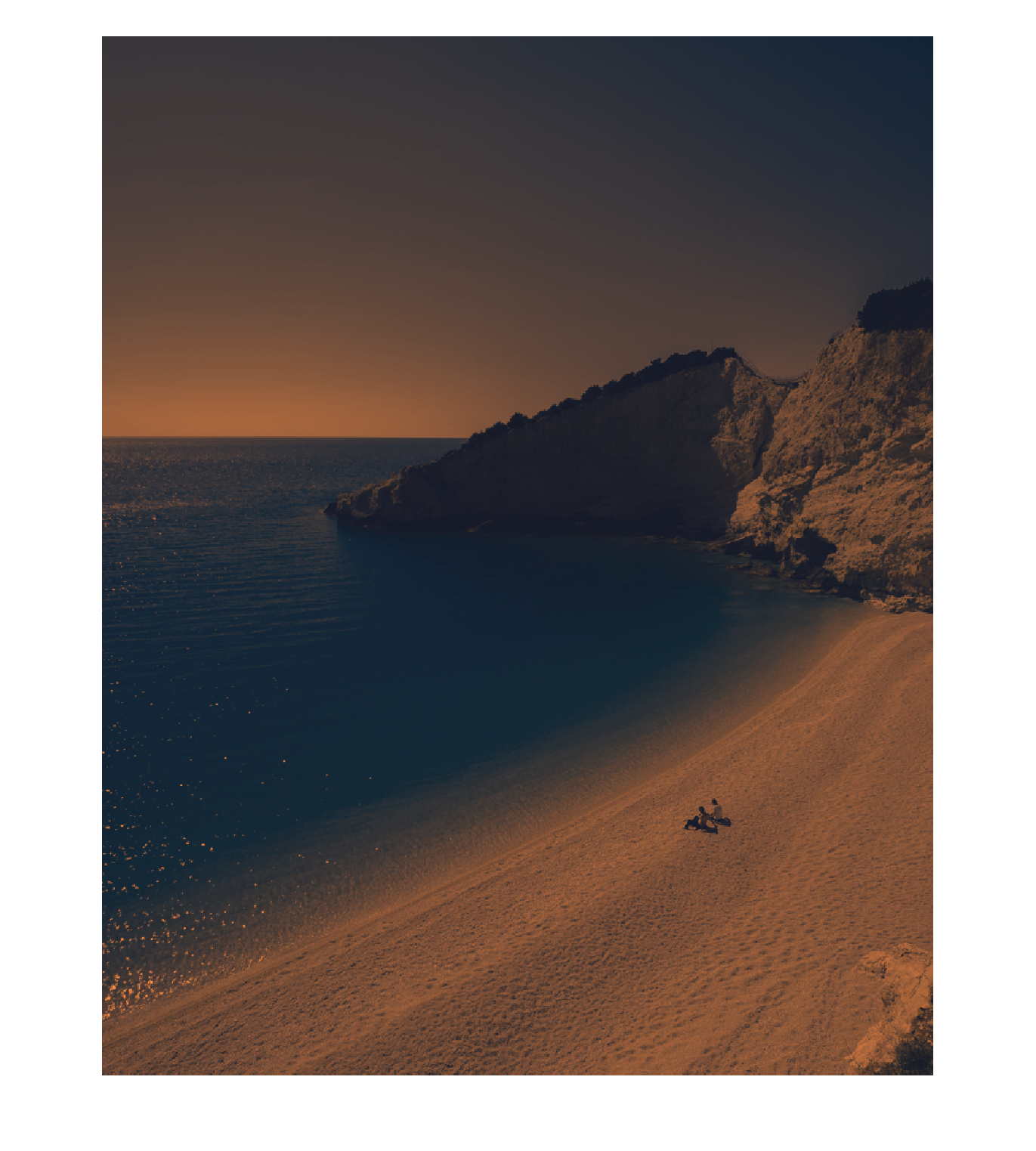}} &
						\subfloat[$\ell = 0.1$, $L = 10$]{\includegraphics[width = 0.2\linewidth]{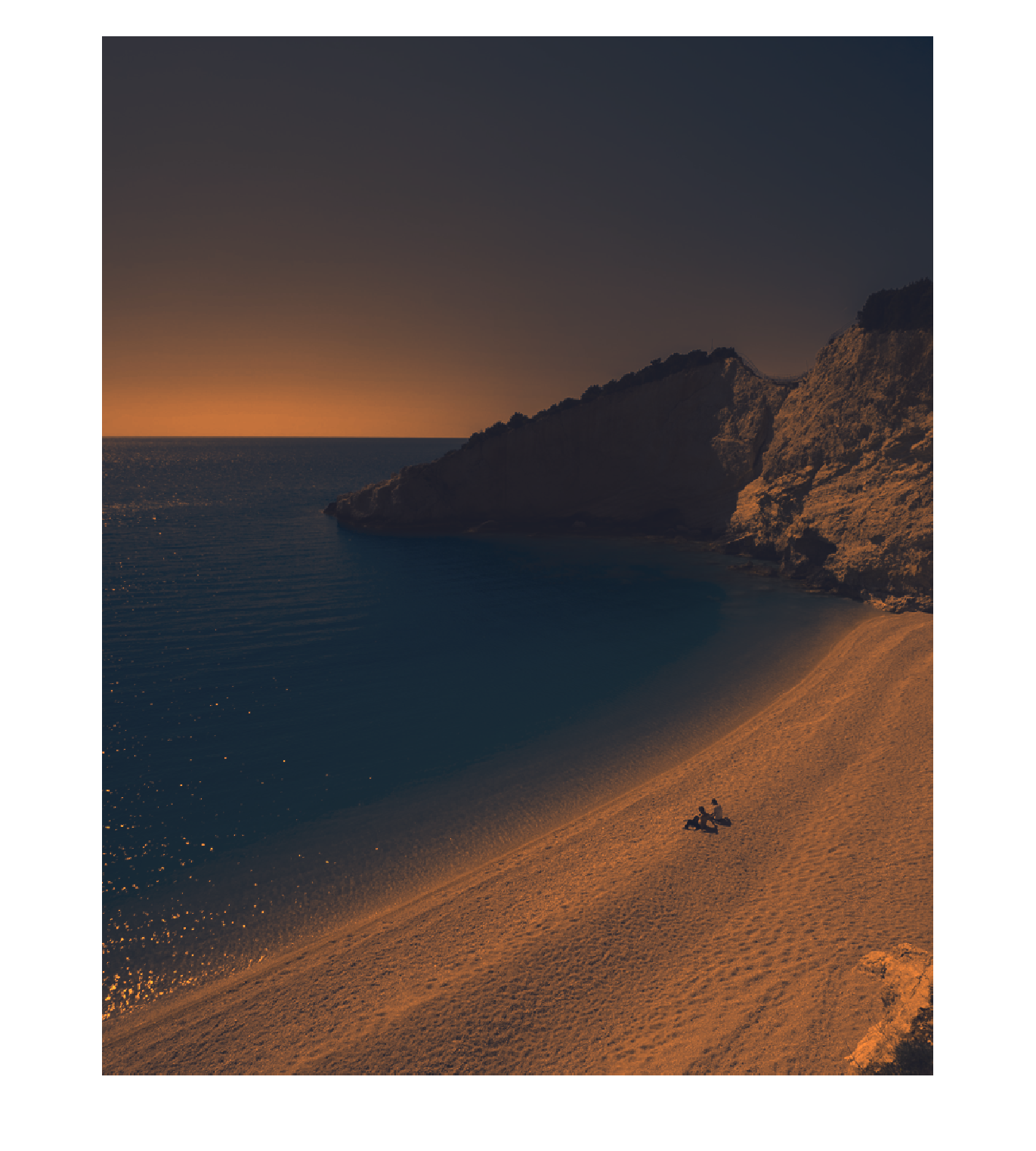}} &
						\subfloat[$\ell = 0.5$, $L = 2$]{\includegraphics[width = 0.2\linewidth]{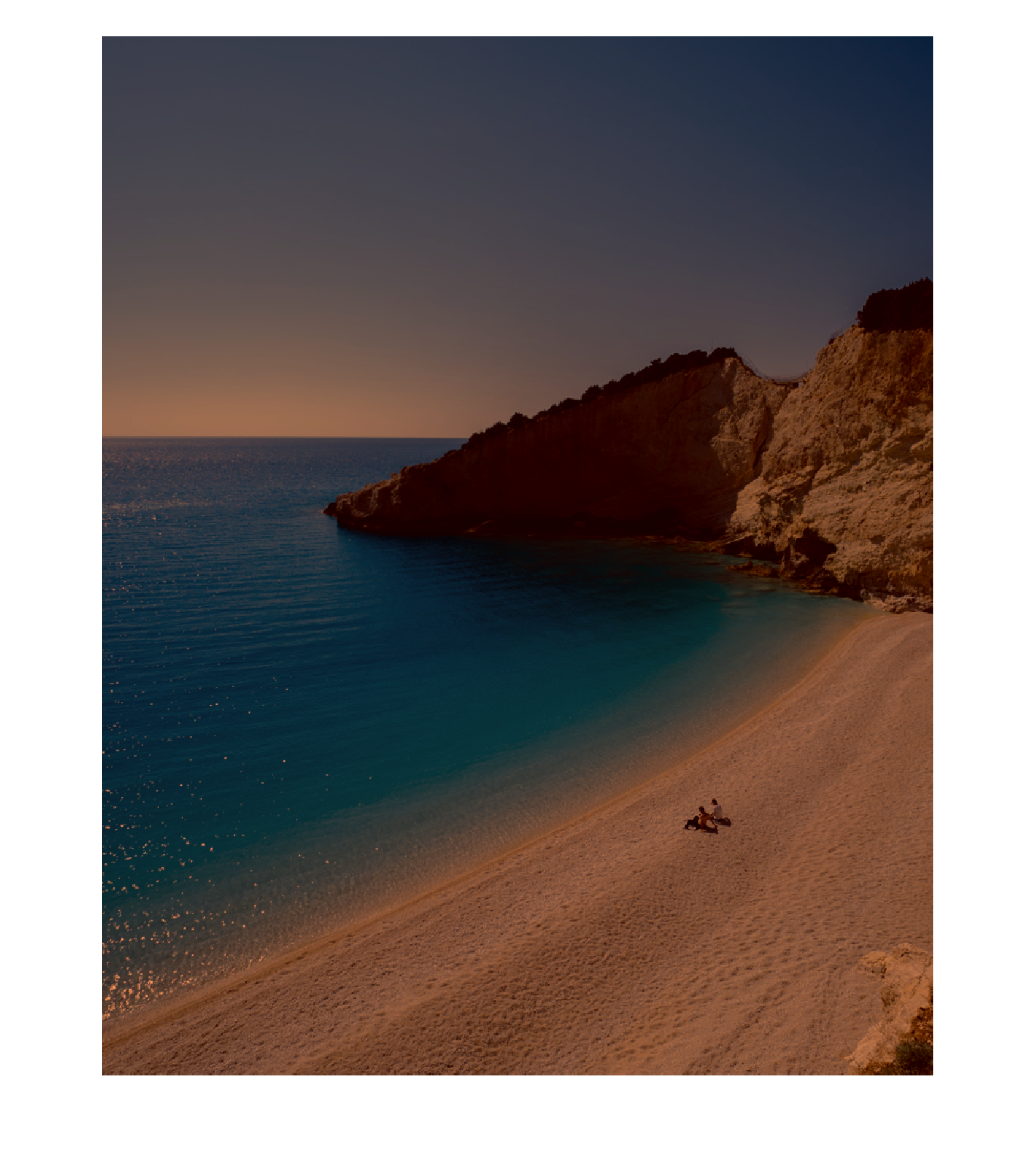}} &
						\subfloat[$\ell = 0.5$, $L = 10$]{\includegraphics[width = 0.2\linewidth]{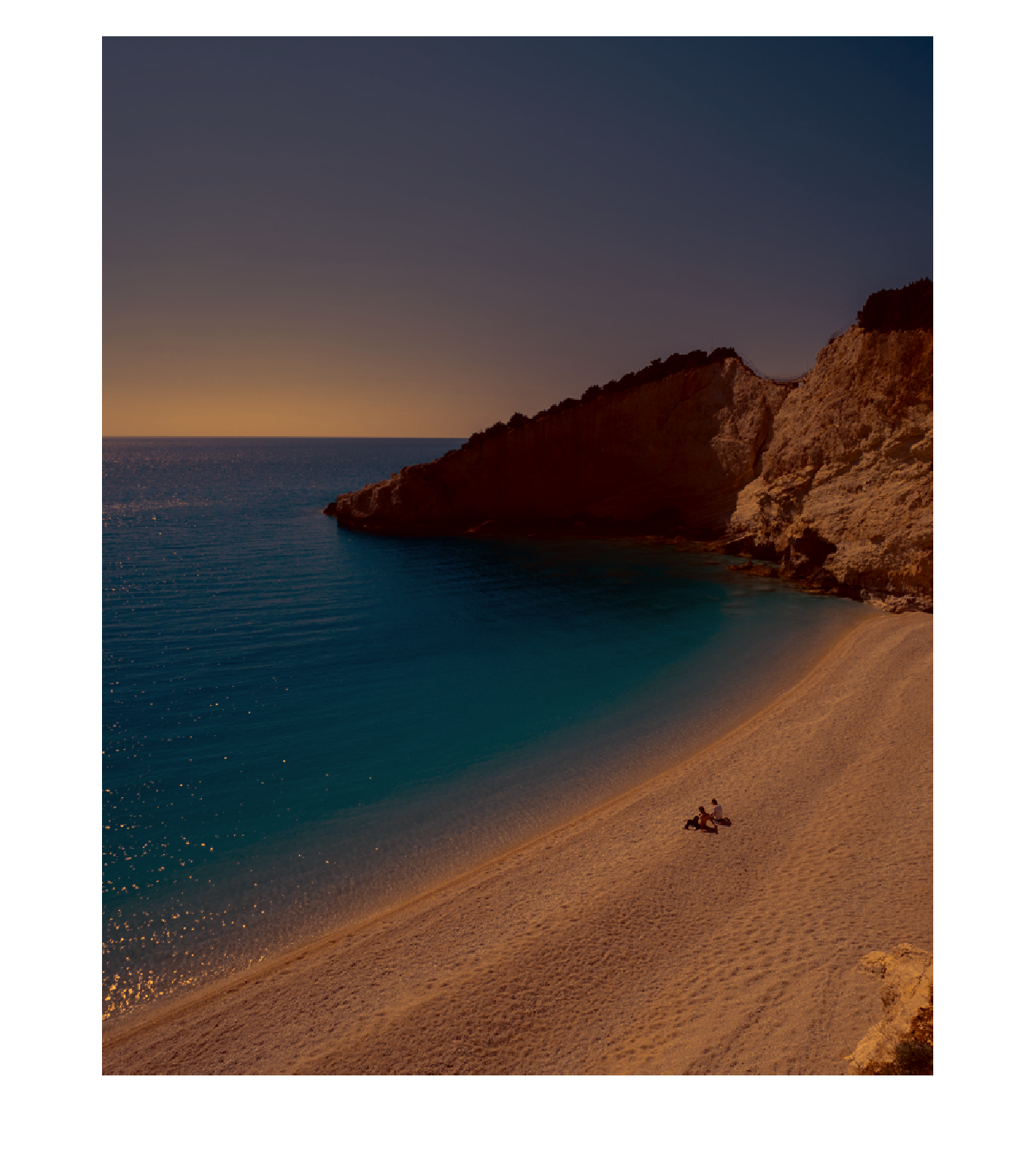}}\\
						\subfloat[$\ell = 1$, $L = 2$]{\includegraphics[width = 0.2\linewidth]{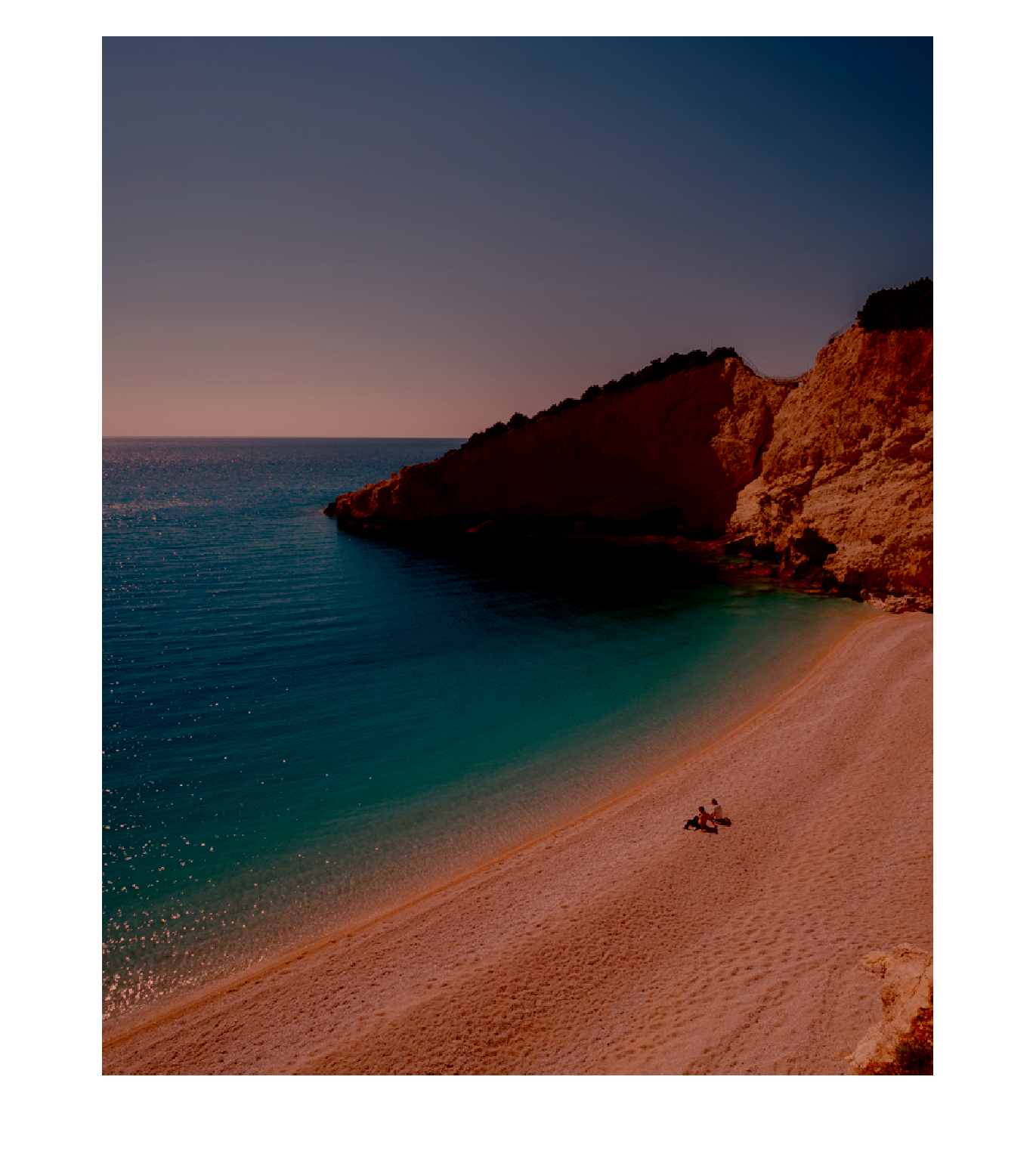}} &
						\subfloat[$\ell = 1$, $L = 10$]{\includegraphics[width = 0.2\linewidth]{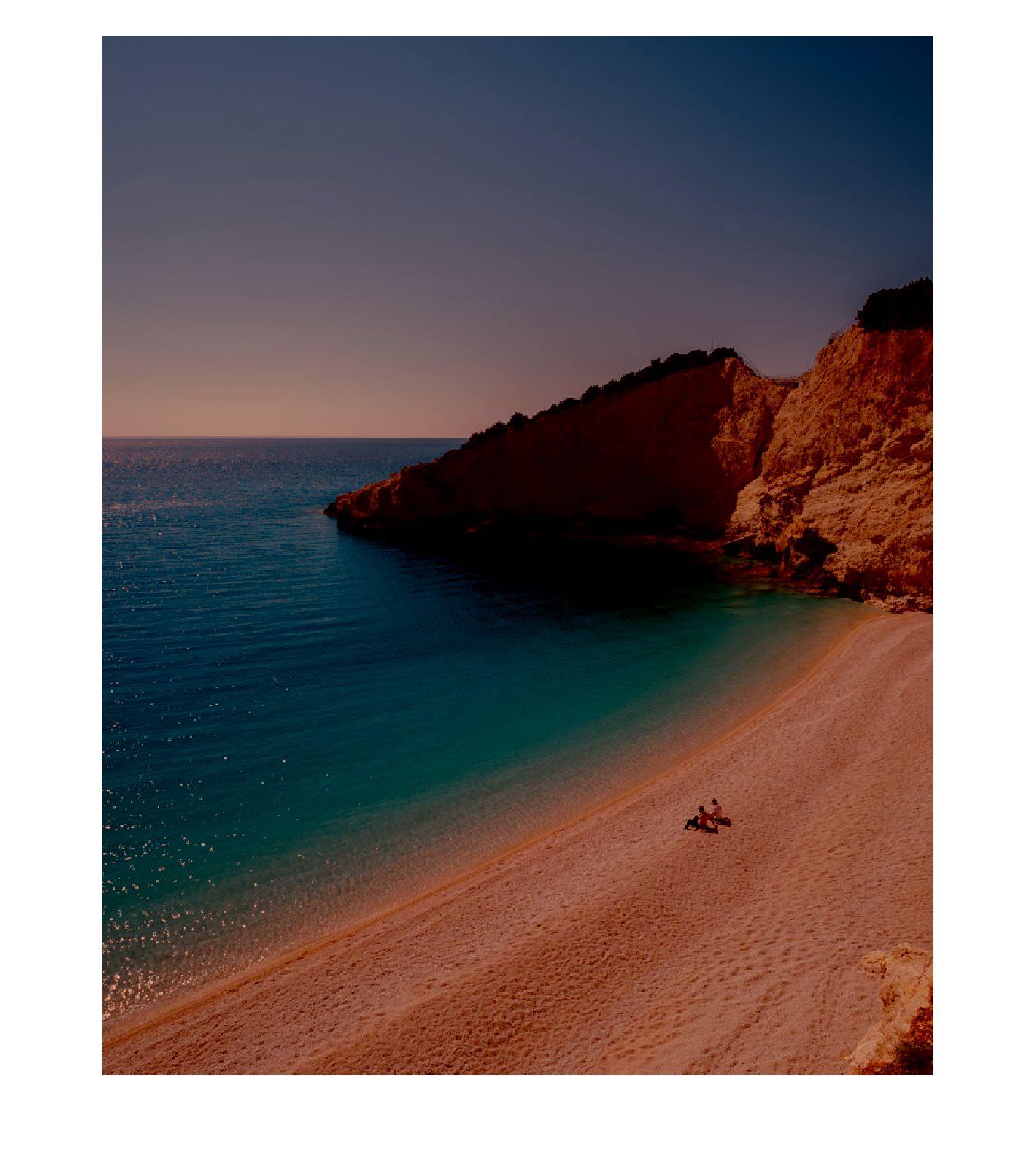}} &
						\subfloat[$\ell = 2$, $L = 5$]{\includegraphics[width = 0.2\linewidth]{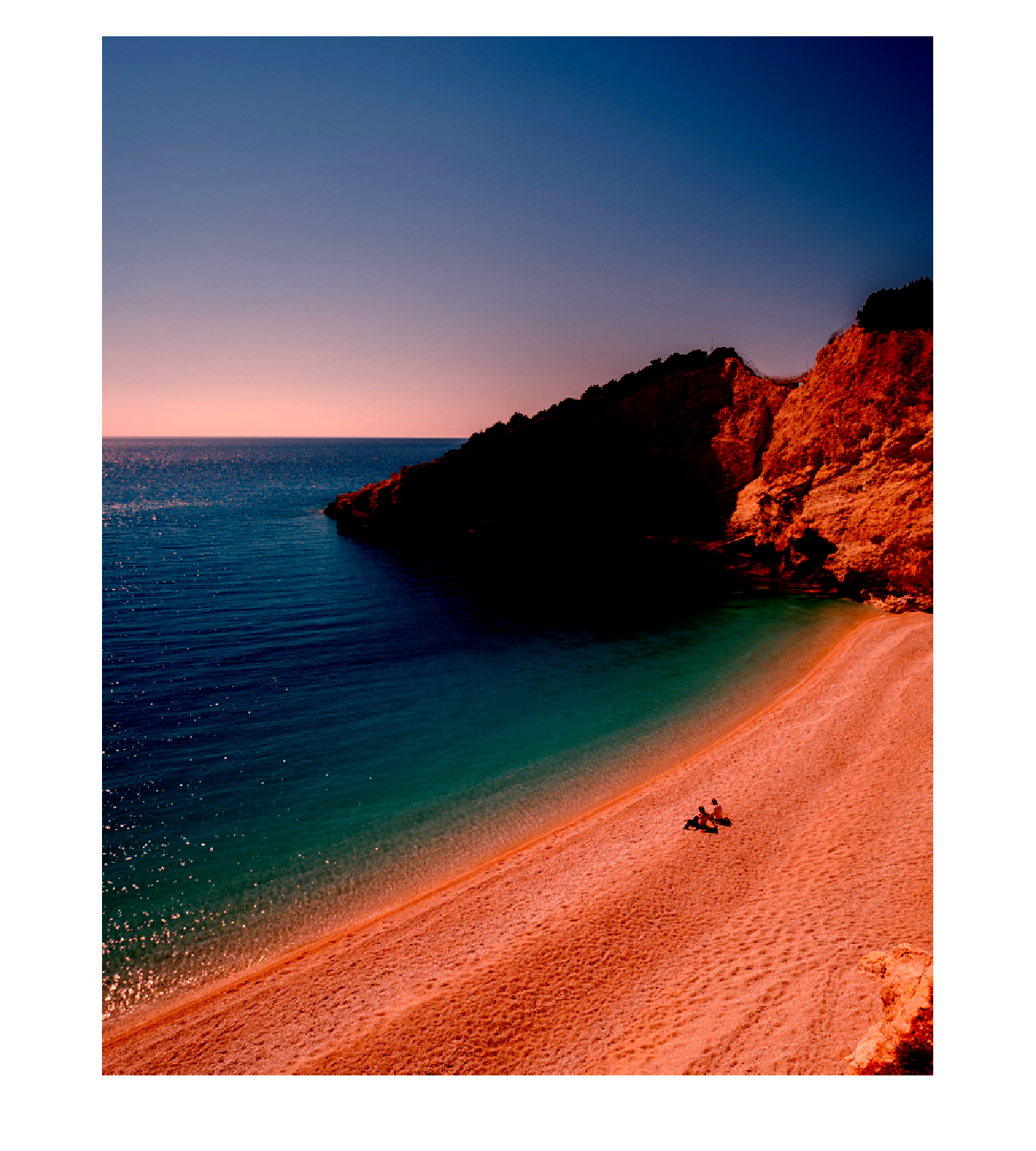}} &
						\subfloat[$\ell = 2$, $L = 10$]{\includegraphics[width = 0.2\linewidth]{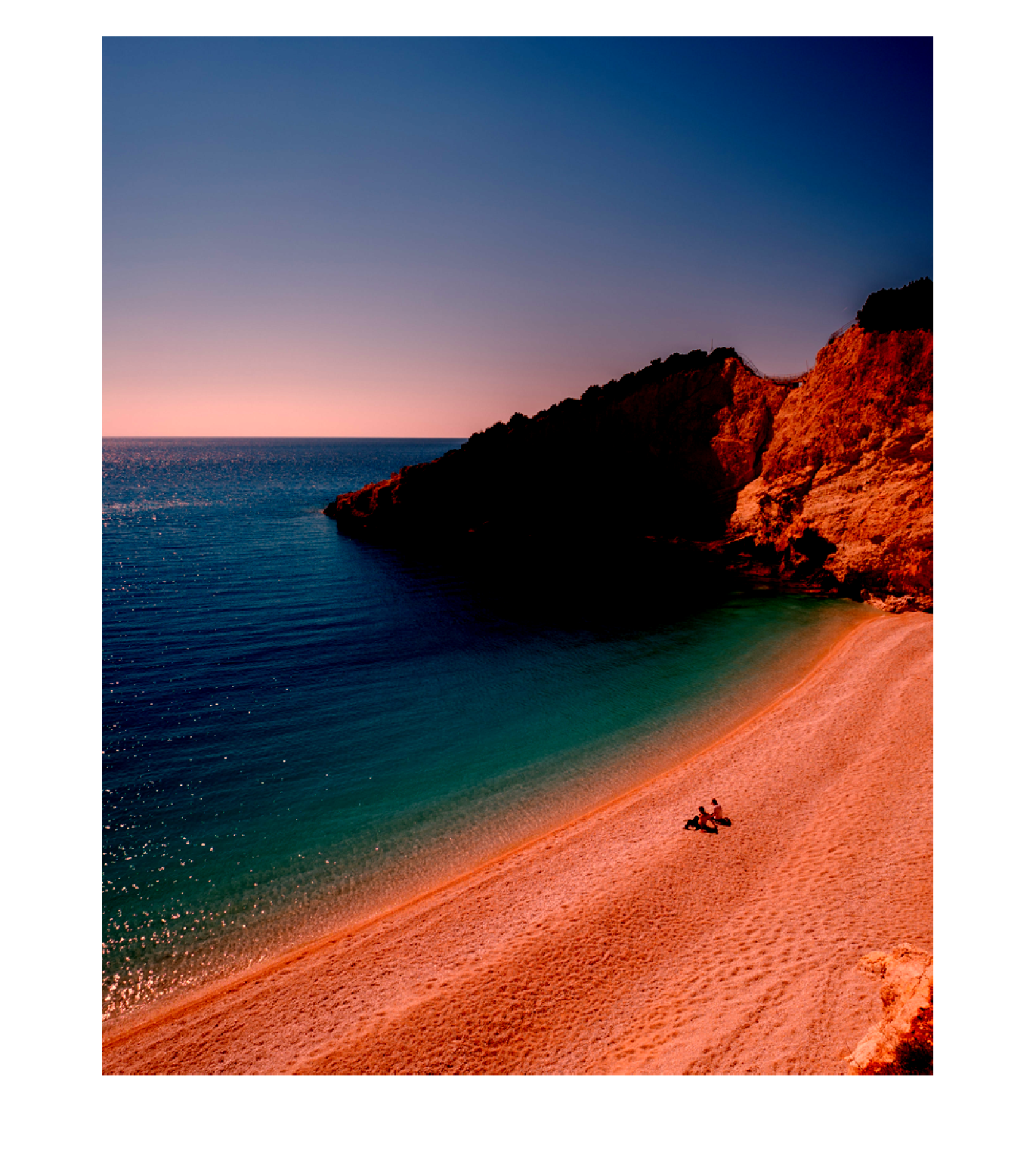}}
					\end{tabular}
					\caption{Color transfer outputs for different regularization parameters $\ell$ and $L$; see Section \ref{subsec:opt.transport}.}
					\label{fig:comparison.color.transfer}
				\end{center}
			\end{figure}

\section*{Acknowledgments}
We would like to thank Amir Ali Ahmadi for bringing the problem of shape-constrained regression to our attention as well as for his suggestions and feedback. We would also like to thank Ioana Popescu and Johan L\"ofberg for their helpful comments and Rahul Mazumder and Dimitrije Ruzic for providing access to relevant datasets.


\begin{small}
\bibliography{pablo_amirali,sample,thesis,supp_bib}
\end{small}
\end{document}